\documentclass[11pt]{amsart}
\usepackage{amssymb,amsmath,amsthm}
\usepackage{a4wide}
\usepackage{graphicx}
\usepackage[inline]{asymptote}
\usepackage{xcolor}
\usepackage{epsfig}
\usepackage{dsfont}
\usepackage{longtable}
\usepackage{mathrsfs}
\usepackage[hidelinks]{hyperref}
\theoremstyle{plain}
\begingroup

\newtheorem*{theoremB'}{Theorem B'}
\newtheorem*{theoremB''}{Theorem B''}

\newtheorem{theorem}{Theorem}[section]

\newtheorem{lemma}[theorem]{Lemma}
\newtheorem{proposition}[theorem]{Proposition}
\newtheorem{corollary}[theorem]{Corollary}

\newtheorem{definition}[theorem]{Definition}
\newtheorem{rmk}[theorem]{Remark}

\endgroup

\numberwithin{equation}{section}

\newcommand\blfootnote[1]{%
  \begingroup
  \renewcommand\thefootnote{}\footnote{#1}%
  \addtocounter{footnote}{-1}%
  \endgroup
}

\theoremstyle{remark}
\begingroup
\endgroup

\numberwithin{equation}{section}

\newcommand{\Z}{\mathbb Z}
\newcommand{\N}{\mathbb N}
\newcommand{\R}{\mathbb R}

\newcommand{\ep}{\varepsilon}
\newcommand{\E}{\mathcal{E}} 
\newcommand{\VV}{W}

\title[Crystallization to the square lattice for a two-body potential]{Crystallization to the square lattice for a two-body potential}

\author{Laurent B\'etermin}
\address{Faculty of Mathematics, University of Vienna, Oskar-Morgenstern-Platz 1, 1090 Vienna, Austria}
\email{laurent.betermin@univie.ac.at}
\author{Lucia De Luca}
\address{Dipartimento di Matematica, Universit\`a di Pisa, Largo Bruno Pontecorvo 5, 56127 Pisa, Italy}
\email{lucia.deluca@unipi.it}
\author{Mircea Petrache}
\address{PUC Chile, Facultad de Matem\'aticas, Av. Vicuna Mackenna 4860, 6904441, Santiago, 
Chile}
\email{decostruttivismo@gmail.com}

\begin{document}
\begin{abstract}
We consider two-dimensional zero-temperature systems of $N$ particles to which we associate an energy of the form 
$$
\mathcal{E}[V](X):=\sum_{1\le i<j\le N}V(|X(i)-X(j)|),
$$
where $X(j)\in\R^2$ represents the position of the particle $j$ and $V(r)\in\R$ is the {pairwise interaction} energy potential of two particles placed at distance $r$. We show that under suitable assumptions on the single-well potential $V$, the ground state energy per particle converges to an explicit constant $\overline{\mathcal E}_{\mathrm{sq}}[V]$ which is the same as the energy per particle in the square lattice infinite configuration. We thus have
$$
N{\overline{\mathcal E}_{\mathrm{sq}}[V]}\le \min_{X:\{1,\ldots,N\}\to\R^2}\mathcal E[V](X)\le N{\overline{\mathcal E}_{\mathrm{sq}}[V]}+O(N^{\frac 1 2}).
$$
Moreover $\overline{\mathcal E}_{\mathrm{sq}}[V]$ is also re-expressed as the minimizer of a four point energy.

In particular, this happen{s} if the potential $V$ is such that $V(r)=+\infty$ for $r<1$, $V(r)=-1$ for $r\in [1,\sqrt{2}]$, $V(r)=0$ if $r>\sqrt{2}$, in which case ${\overline{\mathcal E}_{\mathrm{sq}}[V]}=-4$.

To the best of our knowledge, this is the first proof of crystallization to the square lattice for a two-body interaction energy.
\end{abstract}

\date{\today}
\maketitle \blfootnote{
{\bf Acknowledgments:}  
{LB acknowledges support by VILLUM FONDEN via the QMATH Centre of Excellence (grant no. 10059) during his stay at University of Copenhagen and by the WWTF research project ``Variational Modeling of Carbon Nanostructures" (no. MA14-009) at University of Vienna}. LDL is a member of the INdAM-GNAMPA group and wishes to thank the  Scuola Internazionale di Studi Superiori Avanzati where she worked in the early stage of this project. MP is supported by the Fondecyt Iniciaci\'on grant number 11170264 entitled ``Sharp asymptotics for large particle systems and topological singularities''.}
\tableofcontents
\section{Introduction}

\subsection{Our energy minimization problem} If $X_N:=\{x_1,\ldots,x_N\}$ ($N\in\N$) is a finite subset of $\mathbb R^2$, referred to as {\it configuration}, and $V:[0,+\infty)\to\mathbb R\cup\{+\infty\}$ is a function, referred to as {\it pairwise interaction potential}, the $V$-energy of $X_N$ is defined by
\begin{equation}\label{energy}
\mathcal E[V](X_N):=\frac{1}{2}\sum_{i\neq j}V(|x_i-x_j|).
\end{equation}
We are interested in the minimization of the energy $\mathcal E[V]$ amongst $N$-point configurations under {\it isotropic singular one-well potentials} $V$ which decay as $|x|\to\infty$ (this means that $\lim_{r\downarrow 0}V(r)=+\infty, \lim_{r\to\infty}V(r)=0$ and $r\mapsto V(r)$ is decreasing on $(0,r_0)$ and increasing on $(r_0,\infty)$, for some $r_0>0$). We will normalize $V$ below and assume that
\begin{equation}\label{minv-1}
 \min_{r>0}V(r)=-1.
\end{equation}
Since $\mathcal E[V]$ is invariant under isometries of $\mathbb R^2$, we study minimizers up to isometry, and we are interested in properties which hold for large $N$. We find here conditions (see Theorems \ref{thm1}, \ref{thm2} and \ref{thm3}) under which in three different situations the minimum energy problem for \eqref{energy} is asymptotically solved by a {\it square lattice $t\mathbb Z^2$}, for some $t>0$. This means that, setting 
\begin{equation}\label{defen}
 \mathcal E[V](N):=\min\{\mathcal E[V](X_N)\ :\ \sharp X_N= N\},
\end{equation}
it holds $\mathcal E[V](N)=N\overline{\mathcal E}_{\mathrm{sq}}[V]+\mathrm{O}(N^{1/2})$ as $N\to\infty$, where $\overline{\mathcal E}_{\mathrm{sq}}[V]$ is the minimum energy per point taken amongst all square lattices:
\begin{equation}\label{eppz2_intro}
\overline{\mathcal E}_{\mathrm{sq}}[V]:=\min_{t>0}\lim_{R\to\infty}\frac{\mathcal E[V](t\mathbb Z^2\cap B_R)}{\sharp(t\mathbb Z^2\cap B_R)}.
\end{equation}
Here and throughout the paper $B_R=B(0,R)$, where $B(x,\rho)$ denotes the open ball centered at $x$ and having radius equal to $\rho$.

We note that  for some pairwise interaction potentials $V$ the minimizer of the energy \eqref{energy} is a {\it triangular lattice}, i.e., a rescaled copy of $\mathsf A_2=(1,0)\mathbb Z+(1/2,\sqrt3/2)\mathbb Z$. We refer to Subsection \ref{sec_comparisontri} for a more detailed description of such results and for a general discussion on the optimization problems solved by $\mathsf A_2$. To the best of our knowledge this seems to be the first rigorous proof of crystallization to a square lattice for a two-body potential, a result suggested already in \cite[p. 212]{T} in 2006, and towards which more evidence appeared recently in \cite[Section 1.3]{BP}. We refer to Section \ref{sec:prevres-triang} for more details about results on the optimality of a square lattice. Our three main theorems are stated in Subsection \ref{sec_intro_12} below.

\subsection{Description of the main results}\label{sec_intro_12}
We highlight the basic geometric phenomenon at work in our result by considering a very simple $V$. Let 
\begin{equation}\label{hr_V}
V(r):=\left\{\begin{array}{ll}
              +\infty, & r\in[0,1),\\[3mm]
              -1,& r\in[1,r_{max}],\\[3mm]
              0,&r>r_{max}.
             \end{array}
 \right.
\end{equation}
This is a simple family including the Heitmann-Radin ``sticky disk'' potential \cite{HeRa} for $r_{max}=1$, a case in which the interval of ``favourable distances'' on which $V(r)=-1$ {is} reduced to the single point $\{1\}$, giving $\mathsf A_2$ as the asymptotical optimizer of $\mathcal E[V]$ . Our starting consideration was that for $r_{max}=\sqrt2$, asymptotically $\mathbb Z^2$ is instead the optimizer (further discussion of potentials including \eqref{hr_V} will be the aim of a separate paper \cite{P_family}). This follows from two elementary geometry considerations:
\begin{enumerate}
 \item[(a1)] If no points are allowed to get closer than distance $1$ then the maximum number of points $x_i\neq x_j$ from $X_N$ that are within $\sqrt2$-distance of a given $x_j$, needs to be at most $8$.
 \item[(a2)] If $x_i$ has precisely $8$ ``neighbors'' at distances lying in $[1,\sqrt2]$, and each one of these neighbors has precisely $8$ neighbors as well, then the neighbors of $x_i$ must form a perfect square, i.e. they form, together with $x_i$ itself, configuration isometric to $\{-1,0,1\}^2\subset\mathbb R^2$.
\end{enumerate}
These two ingredients give the basic rigidity result on which our paper is based. The main new idea that we exploit, compared to other energy-minimization problems, is to ``look beyond the next neighbors''. It is worth to mention an equivalent rigidity result, also useful later, which says that ``small energy quadrilaterals are squares'' (see Lemma \ref{rmk_squarerigid}):
\begin{enumerate}
 \item[(b)]  If a quadrilateral $Q$ has sidelengths $\ge 1$ and lengths of diagonals $\le \sqrt2$ then $Q$ is a square.
\end{enumerate}
The proof of (b) uses the same kind of methods as (a1)-(a2). {Lemma \ref{rmk_squarerigid}} also describes a way to obtain it as a corollary of (a1)-(a2) directly.

\medskip

This rigidity {argument} gives our first result (see Theorem \ref{thm:crystal_infinite} for a full statement).
\begin{theorem}\label{thm1}
Let $r_{max}=\sqrt{2}$ in \eqref{hr_V}. 
Then, with the notation \eqref{defen} and \eqref{eppz2_intro}, we have
 $$
N\overline{\mathcal E}_{\mathrm{sq}}[V] \le \E[V](N)\le N\overline{\mathcal E}_{\mathrm{sq}}[V] +O(N^{{1}/{2}})\qquad\textrm{as }N\to+\infty,
$$
where
$$
\overline{\mathcal E}_{\mathrm{sq}}[V]=-4.
$$
\end{theorem}

In the above statement and throughout all the paper $O=O(T)$ denotes a continuous function on $\R^+$ such that $\lim_{T\to +\infty}\frac{|O(T)|}{T}$ is finite.

\medskip

The fact that rigidity is ensured once we look up to a large enough number of next-neighbors, is a natural idea, exploited successfully in the work by Hales on the best-packing in $3$ dimensions \cite{hales3d}. However, as shown in that work, it could lead to somewhat tedious case examinations, in the absence of a machinery which allows to streamline the bookkeeping of the energy during the optimization (a striking example of such machinery, in which all layers are studied at the same time via Fourier analysis, are the recent papers \cite{CElkies,CK,V8,CKMRVPacking24,CKMRV}).

\medskip

In order to deal with interaction potentials that are more regular than the one in \eqref{hr_V} we follow  the idea (b) above, introducing as a building block the $4$-point energy defined as
\begin{equation}\label{defEWintro}
 \mathcal E_4[V](x_1,x_2,x_3,x_4):= \frac12\sum_{j=1}^2\sum_{i=1}^4V(|x_i-x_{i+j}|),\quad \forall i\in\{1,2,3,4  \}, x_i\in \R^2,
\end{equation}
where indices are considered modulo $4$. 
Note that ``diagonal pairs'' $(x_1,x_3), (x_2,x_4)$ are counted twice in the above sum, while ``nearest neighbors'' are counted only once.

\medskip

Calling {\it elementary square} a $4$-ple of points whose vertices form a ``small'' deformation of a unit square (see Definition \ref{def:deform}), the use of $\E_4[V]$ is clear in view of the following considerations:
\begin{itemize}
\item Up to boundary contributions, the energy of a square-lattice configuration is the sum over all elementary squares of $\mathcal E_4[V]$ {(see Figure \ref{fig-resum})}. 

\item On the other hand, (as in point (b) above) if $V$ is a ``short-enough-range'' potential, a minimizing configuration is the one for which each elementary square separately optimizes $\mathcal E_4[V]$ (see Section \ref{sec_4ptmin}).
\end{itemize}
 \begin{figure}[h]
 \includegraphics[width=6cm]{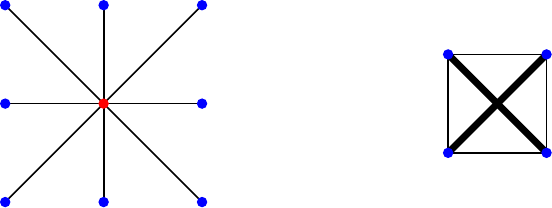}
\caption{We will use the fact that {energy contributions appearing in $\mathcal E_4[V]$} correspond to energy contributions of a single point as indicated in the above figure. This resummation trick will be used for regions of our configuration which are small {deformations} of regions in $\mathbb Z^2$.}\label{fig-resum}
\end{figure}
 
 \medskip

Since the 4-point energy functional $\E_4[V]$ plays a fundamental role in our anlysis, we focus on convexity and minimality properties of $\mathcal{E}_4[V]$; such an analysis has appeared in \cite{fritheil} in the case of a potential modeling elastic responses, and more general calculations of the type that we perform in Section \ref{sec_4ptmin} seem to have a long history, starting from Maxwell's work \cite{maxwell} (see also \cite{calladine}), and appear in the study of stability and oscillation modes of frameworks, see e.g. \cite{izmestiev} for a geometric introduction to this subject, and the references therein.

\medskip

Now we state our second result which holds for more regular finite-range potentials $V$. We assume that $V$ satisfies the following properties, for suitable $0<\alpha'<\alpha<\alpha''$, $\epsilon'>0$  (small) and $K>0$ (large):

 \begin{itemize}
\item[(A)] $\mathcal E_4[V]$ has a strict minimum at the unit square, and this is its unique minimum amongst all $4$-point configurations with interpoint distances in $E_{\alpha''}$.
This happens in particular if $V$ is piecewise-$C^2$ in $(0,\infty)$, {i.e. $V\in C^2_{pw}((0,\infty))$}, and satisfies the explicit derivative bounds $(1)$ and $(2)$ from Section \ref{sec_condV} below);
\item[(B)]   $-1=\min V\le V(r)\le -1+\epsilon'$, for $r\in E_{\alpha'}$,
where 
 \begin{equation}\label{eq_Ealpha}
  E_\beta:=E_\beta^1\cup E_\beta^2,\quad E_\beta^1:=(1-\beta,1+\beta)\textnormal{ and } E_\beta^2:=(\sqrt{2}-\beta,\sqrt{2}+\beta)\quad\textrm{for any } \beta>0;
\end{equation}
\item[(C)] $V(r)>-\frac{1}{2}$ if $r\notin(1-\alpha,\sqrt 2 +\alpha)$;
\item[(D)] $V(r)\ge K$ for $r\le1-\alpha$;
\item[(E)] $V(r)=0$ for $r  \ge\sqrt2+\alpha''$.
 \end{itemize}

\begin{figure}[h]
\includegraphics[width=7cm]{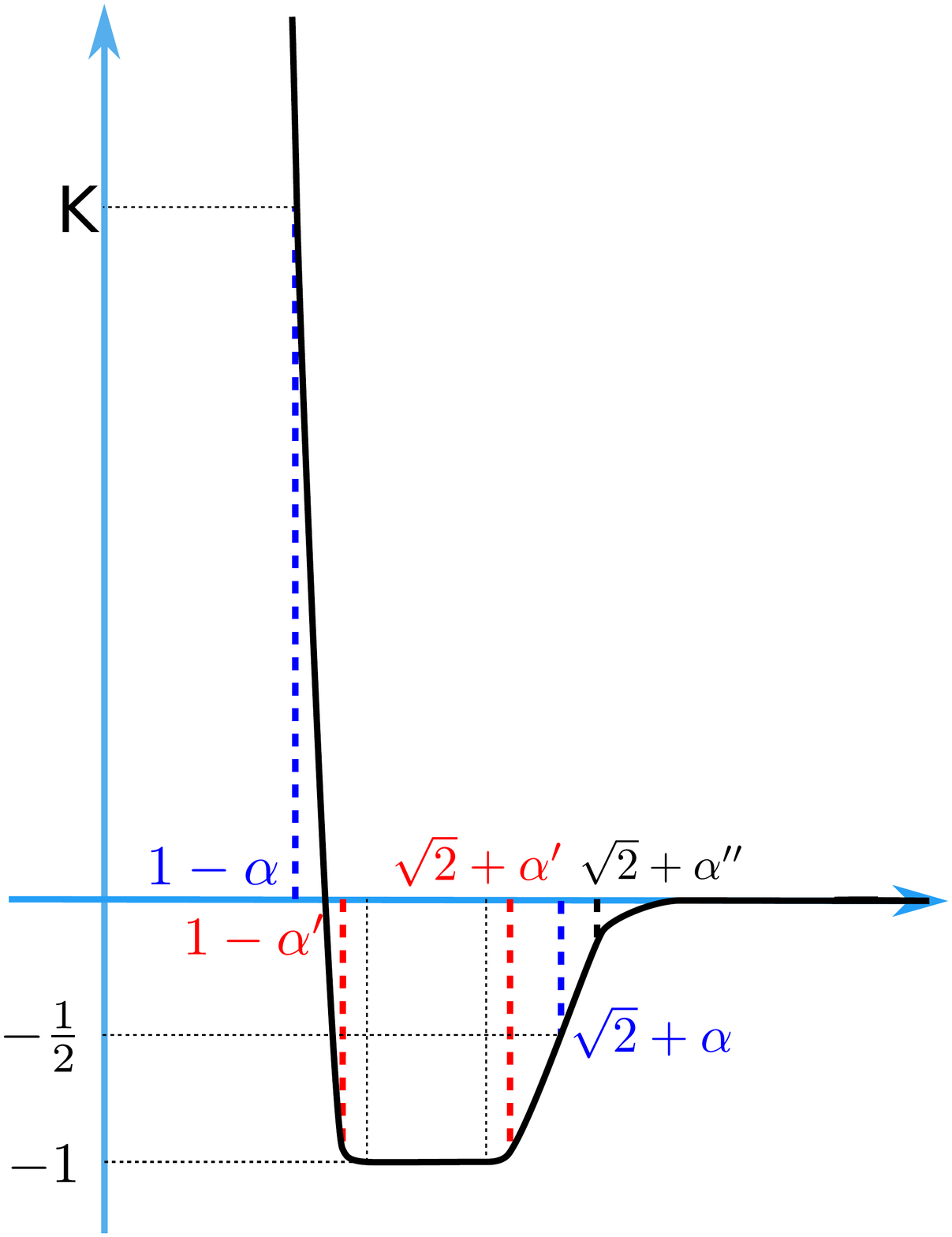}\qquad\includegraphics[width=7cm]{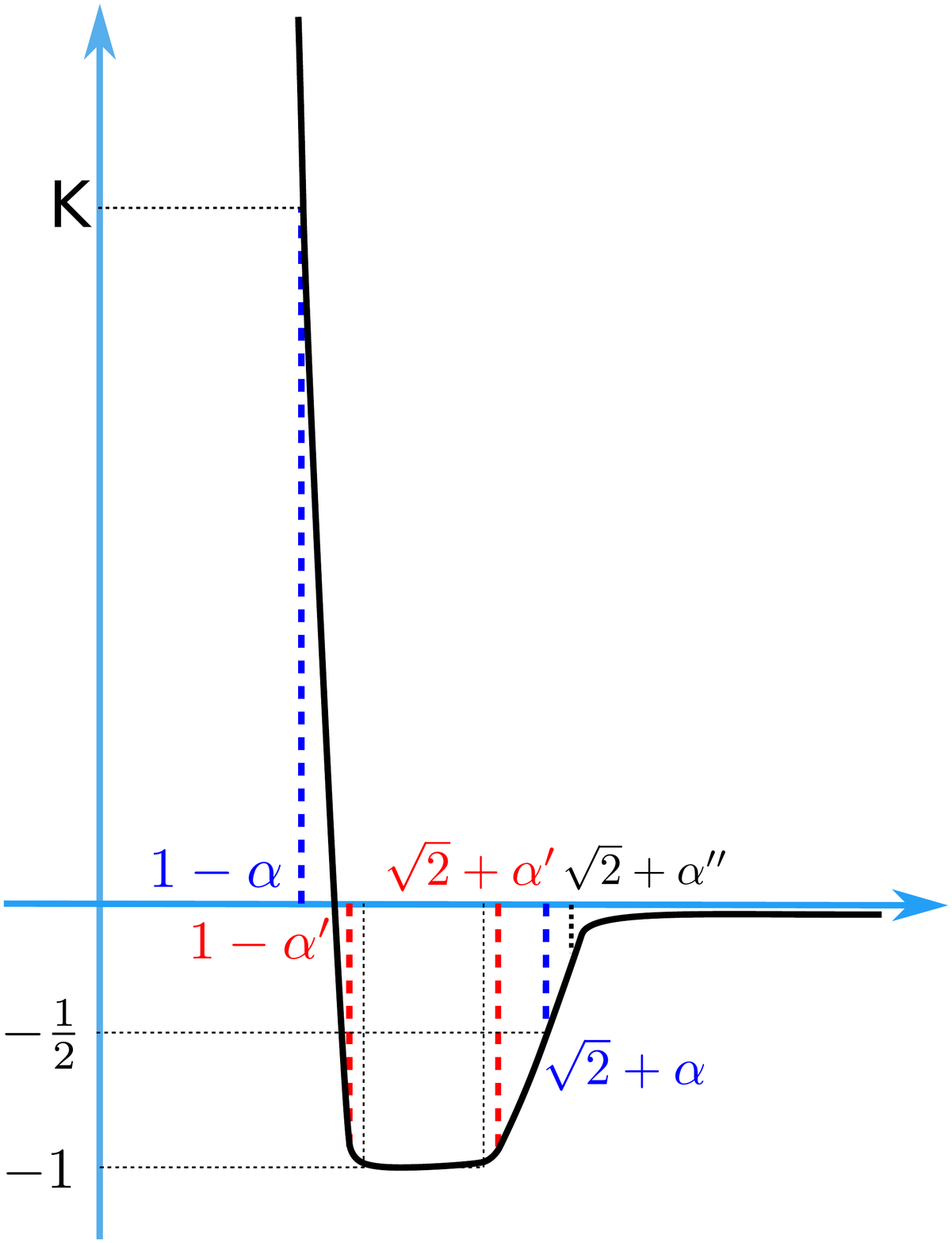}
\caption{A potential satisfying properties (A)-(E) required for Theorem \ref{thm2}, and one satisfying properties (A)-(E') needed for Theorem \ref{thm3} (note that the values of $\alpha,\alpha',\alpha''$ are exaggerated).}\label{V2}
\end{figure}

See Figure \ref{V2}(left) for a potential satisfying these properties. This control allows us to get a result similar to Theorem \ref{thm1} (see Theorem \ref{thm_fr_crystal} for a more complete statement).
\begin{theorem}\label{thm2}
There exists $\bar\alpha, \bar \epsilon>0$ such that for all $ \alpha''\in(0,\bar\alpha]$ and all $\epsilon'\in(0,\bar\epsilon]$ there exist $0<\alpha'<\alpha<\alpha''$ and $K=K(\alpha,\epsilon')$ such that if $V\in C^2_{pw}((0,\infty))$ satisfies above conditions (A)-(E) then
\begin{equation}
N\overline{\mathcal E}_{\mathrm{sq}}[V]\le \mathcal E[V](N)\le N\overline{\mathcal E}_{\mathrm{sq}}[V]+O(N^{1/2})\qquad\textrm{as $N\to +\infty$,}
\end{equation}
where
\begin{equation}\label{min_dens_intro_fr}
\overline{\mathcal E}_{\mathrm{sq}}[V]=\min \mathcal E_4[V].
\end{equation}
\end{theorem}

The above theorem uses, besides the thorough study of $\mathcal E_4[V]$ which we already mentioned, also the quantitative version of the phenomena valid for \eqref{hr_V}, which are included in a geometric rigidity result (Lemma \ref{lem:combintometric}). This result, whose proof is based on an elementary study of configurations close to $\{-1,0,1\}^2\subset \mathbb R^2$, is another ingredient of the proofs which is new compared to the $\mathsf A_2$-crystallization results. It plays an important role, allowing to avoid losing the combinatorial order between neighbors lying in the ``interior'' of a configuration, i.e. in a ``neighborhood'' of a point having 8 ``nearest neighbors''. Special emphasis on the combinatorial setup is included in Section \ref{sec_combinsetup}: to keep track of this structure we use a differential-geometric language allowing to keep the model-space $\mathbb Z^2$ from the actual energy competitor at hand. Using geometric ideas for organizing energy contributions was an idea already present in \cite{T} and further developped e.g. in \cite{DLF1} and \cite{DLF2} in the $2$-dimensional setup adapted to the study of $\mathsf A_2$-crystallization (see also \cite{DNP} for an application of this setting to the study of polycrystals made by $\mathsf A_2$ grains). 

\medskip

The final result we prove is for long-range potentials $V$. The required assumptions on $V$ coincide with conditions (A)-(D) above, plus an assumption on the ``fast decay'' of the tail, i.e., condition (E) 
\medskip

We can now state our theorem, whose detailed statement is Theorem \ref{thm_lr_crystal}.
\begin{theorem}\label{thm3}

There exists $\bar{\bar\alpha}, \bar{\bar \epsilon},\bar{\bar \epsilon}_0>0$ such that for all $ \alpha''\in(0,\bar{\bar\alpha}]$, $\epsilon'\in(0,\bar{\bar\epsilon}]$  and $\epsilon\in [0,\bar{\bar \epsilon}_0]$ there exist $0<\alpha'<\alpha<\alpha''$ {and $K=K(\alpha,\epsilon')$} such that if $V\in C^2_{pw}((0,\infty))$ satisfies the above conditions (A)-(E') then
\begin{equation}
N\overline{\mathcal E}_{\mathrm{sq}}[V]\le \mathcal E[V](N)\le N\overline{\mathcal E}_{\mathrm{sq}}[V]+O(N^{1/2})\qquad\textrm{as $N\to +\infty$,}
\end{equation}
where
\begin{equation}\label{min_dens_intro}
\overline{\mathcal E}_{\mathrm{sq}}[V]=\min_{t>0}  \lim_{R\to\infty}\frac{\mathcal E[V](t\mathbb Z^2\cap B_R)}{\sharp(t\mathbb Z^2\cap B_R)}.
\end{equation}

\end{theorem}
\rm{Notice that condition (E') with $\epsilon=0$ gives exactly condition (E); therefore Theorem \ref{thm2} is implied by Theorem \ref{thm3}, once provided that, if (E) holds, then the right hand sides of \eqref{min_dens_intro} and of \eqref{min_dens_intro_fr} coincide. 
In this paper, we do not pursue this strategy.
Instead we find more instructive to give first the proof of Theorem \ref{thm2}, since in this case the scheme of the proof is somehow ``cleaner'' and does not require all the tools needed in the proof of Theorem \ref{thm3} to account for long-range interactions. 
} 

Roughly speaking, the starting point in the proof of Theorem \ref{thm3} consists in showing that $\overline{\E}_{\mathrm{sq}}[V]=\min\E_4[V_*]$, where $V_*$ is  a kind of ``long-range potential defined on the distances of $\Z^2$'' (see \eqref{defntildev}). The main phenomenon at work in the above result is that the tail of our potential $V$ is decaying so fast that actually $V_*$ is nothing but
a small enough perturbation of $V$ so that 
 the proof of Theorem \ref{thm3} can be reduced to the one of Theorem \ref{thm2} for $V_*$. The same method was also at the base of the main result of \cite{T} for the triangular lattice (see also \cite{ELi,FESBrenner} for applications of the same ideas to the honeycomb lattice with an additional three-body potential), however in our case new difficulties arise due to the fact that we have to account for energies coming from ``sides''  and ``diagonals'' of squares of different scales in our configurations. Thus we need new tools such as Lemma \ref{lem_splitgraph} which then {allow resummation methods that yield a lower bound of the energy $\E[V]$ via} the sum of $4$-point energies on $r$-squares.

\medskip

For the proof of Theorem \ref{thm3} we have at the same time done an exercise of simplifying and extending to our situation the methods of proofs from \cite{T}. The main technical improvements compared to \cite{T} are that:
\begin{itemize}
 \item we put a focus on separating the use of the combinatorial information, the metric information and the information about the embedding to $\mathbb R^2$ of our configurations;
 \item we avoid the use of the Friesecke-James-M\"uller rigidity estimate, employed in \cite{T}: instead, we use a rougher estimate based on John's earlier result (see Lemma \ref{johnlemma}), which makes the proof self-contained without changing the decay hypotheses needed on $V$;
 \item we make more explicit the method of proof started in \cite{T}, i.e. the idea of controlling long-range deformations via the Hessian of the microscale energy,  by separating the self-contained result of the existence of minimizers for small perturbations of the potential $V$ (Proposition \ref{prop_square}).
\end{itemize}

\subsection{Previous results on the optimality of the square lattice}\label{sec:prevres-triang}

Only few rigorous results exist about the crystallization on a square lattice, i.e. the fact that $\Z^2$ is a ground state of an interaction energy, with either one or several types of particles. Also note that, in $3$ dimensions, there is only one chemical element which has a simple cubic structure (i.e. Polonium) and the only ionic solid having a simple cubic basis is the Sodium Chloride NaCl (rock-salt structure). However, in dimensions $2$ and $3$, $\Z^2$ and $\Z^3$ are some of the very few lattices (together with the triangular, the BCC and the FCC lattices) that are ``density-stable", i.e. they can be critical points of the lattice energy per point associated with any absolutely summable interaction potential $V$ for densities in an open interval (see \cite{BMorse19} for a proof). It is then reasonable to think that they are  good candidates for ground states of energies such as $\mathcal E[V]$ for two-body isotropic one-well potentials $V$.

\medskip

The first rigorous result in this direction seems to be the work of Mainini, Piovano and Stefanelli \cite{MPS}, who proved the optimality of a subset of $\Z^2$ for a combination of (short-range) two-body and three-body angular potentials which favour right angles. Regarding the analogy with two-ion compounds, Friedrich and Kreutz \cite{FrKr} have shown the energy-optimality of a subset of $\Z^2$ composed of two types of particles under short-ranged repulsive/attractive interactions (modelling a rock-salt structure, in two dimensions).

\medskip

Several potentials have been designed for stabilizing a square or a cubic lattice. Exploring the different structures that can be obtained by using a decreasing convex potential, Marcotte, Stillinger and Torquato have defined in \cite[Section III.A]{Torquatosquare} an example of potential such that $\Z^2$ is a ground state at fixed density 1, the same being also done in \cite{BTheta16}. In \cite{Torquatocubic}, Rechtsman, Stillinger and Torquato proposed the potential $V(r)=r^{-12}-2.7509 e^{-32.2844(r-\sqrt{2})}$ that has (numerically) $\Z^3$ as the ground state of the pairwise energy.

\medskip

Concerning the search of ground states amongst periodic configurations, it has been numerically shown in \cite{BLocal18} that the square lattice is the ground state of the Lennard-Jones potential $V(r)=r^{-12}-2r^{-6}$ among Bravais lattices of fixed density belonging to $(0.79,0.87)$. This was conjectured to still hold true for general differences of completely monotone functions in \cite{BLocal18} and was investigated for the Morse potential in \cite{BMorse19}. For the 3-block copolymer case, Luo, Ren and Wei \cite{LRW} proved in $2$-dimensions the optimality among Bravais lattices of a square lattice of alternating types of species (two kinds with different sizes, a third one being considered as a background) under Coulomb interactions, under the condition that the parameter $b$ of their system -- depending on the size and a weight associated to each species --  belongs to a certain explicit interval. Finally, two of the authors of this paper have constructed in \cite{BP} several examples of two-body one-well potentials $V$ such that a square lattice has lower $V$-energy per point than a triangular one.

\subsection{Previous literature on the triangular lattice and comparison with our results}
\label{sec_comparisontri}
Many {two}-dimensional optimization problems give as a (proved or conjectured) minimizer $\mathsf A_2$. {These include the} best-packing problem \cite{FTpacking2}, optimal-transport type problems \cite{BPT}, the best-covering problem \cite{kershner} and the quantizer problem \cite{thothie} (see also \cite{Gruberhex} for more examples). Furthermore $\mathsf A_2$ is conjectured to be asymptotically minimizing for \eqref{energy} when $V$ is any Lennard-Jones type potential (i.e. a difference of inverse power laws), see \cite{BLY,BTheta16,BLocal18,BP}, as well as for the Morse potential \cite{BMorse19}. It was recently conjectured in \cite{CK} that in fact $\mathsf A_2$ is {\it universally optimal}, i.e. it optimizes among fixed-density configurations all energies for which $W$ defined such that $V(r)=W(r^2)$ has nonnegative Laplace transform ($W$ is called completely monotone), a result so far known only for algebraically simpler to treat lattices in $8$ and $24$ dimensions \cite{CKMRV}. {A related conjecture is} the {\it Abrikosov conjecture}, which again postulates that $\mathsf A_2$ is optimal at any fixed density for the renormalized energy, i.e. under potentials $V$ with heavy tails such as the Coulomb potential from Electrostatics \cite{Abrikosov,SSVortex,RSCoulombGas,PSRiesz,BS}. {Recently such a conjecture was shown to be equivalent to the Cohn-Kumar conjecture in \cite{PS2} for Coulomb potentials and some Riesz potentials, and further extension to all Riesz potentials may be possible.}

\medskip

{Concerning} crystallization, the first rigorous proof of crystallization in two dimensions under a one-well isotropic potential seems to be the one by Heitmann and Radin \cite{HeRa} of 1980, who consider the potential from \eqref{hr_V} with $r_{max}=1$, called the ``sticky disk'' potential. Such a result is actually a finite crystallization result, i.e. Heitmann and Radin proved that for every $N\in\N$ all the minimizers of the sticky disc energy lie, up to rotations and translations, on $\mathsf{A}_2$, using the minimal value of the energy, which was found  in turn  by Harborth \cite{Har} (see \cite{DLF1} for a more transparent proof).

\medskip

Later, Radin \cite{R} considered a slightly different version of the potential from \cite{HeRa}
\[
 V_{Rad}(r):=\left\{
 \begin{array}{ll}
+\infty,& r<1,\\
24r-25,& 1\le r<25/24,\\
0,&r>25/24.\end{array}\right.
\]
In this case {it is shown} that all nearest-neighbors of the minimizers are at distance precisely $1$ and then apply the basic rigidity principle that selects the triangular lattice ground state configuration, however the techniques are not sufficient for allowing smoother $V$. Our Theorems \ref{thm1} and \ref{thm2} can be considered {as} the asymptotic versions of the results \cite{HeRa} and \cite{R} in the square lattice case. Actually, the basic rigidity principle on which such proofs are based, is analogous to points (a1)-(a2) and (b) in Section \ref{sec_intro_12} and seems to be older. Indeed, it appears in the solution of the 2 dimensional packing problem, appearing e.g. in the paper \cite{FTpacking2} by Fejes T{\'o}th from 1943.

\medskip

Finally, the result about $\mathsf A_2$ which is perhaps closer in spirit to our Theorems \ref{thm2} and \ref{thm3} is the 2006 paper \cite{T} by Theil, in which the main theorem assumes that $V\in C^2((1-\alpha,\infty))$, for some $\alpha>0$ sufficiently small, satisfies the following conditions:
\begin{itemize}
\item[(i)] $V''(r)\ge 1$ for $r\in(1-\alpha,1+\alpha)$; 
\item[(ii)]  $V(r)\ge -\alpha$ for $r\in[1+\alpha,4/3]$;
\item[(iii)] $V(r)\ge \frac 1 \alpha $ for $r\le 1-\alpha$;
\item[(iv)] $|V''(r)|\le \alpha r^{-7}$ for $r>4/3$;
\item[(v)] The minimal energy per point $\min_{r>0} \frac{1}{2}\sum_{p\in \Z^2\backslash \{0\}} V( r |p|)$ is achieved for $r=1$.
\end{itemize}

Under these conditions, the conclusion of the main theorem in \cite{T} gives crystallization to $\mathsf A_2$ in exactly the same sense as expressed in the conclusions of our Theorem \ref{thm3} for the square lattice.

\medskip

We note that the above conditions (ii), (iii) and (iv) above are triangular lattice equivalents of conditions (C), (D), (E') respectively. 
Here, as in \cite{T}, the role of these conditions is to suitably normalize $V$ and to allow to apply the basic rigidity principles as appearing in e.g. \cite{FTpacking2} for the triangular lattice, and the apparently new ones (a1)-(a2), (b) for our new result on the square lattice. Furthermore, assumption (v) is here to force the minimizer to be exactly $\mathsf{A}_2$, which is not the case in our case where $t\Z^2$ ($t$ given by \eqref{min_dens_intro}) is an asymptotic minimizer of our energy.

\medskip

On the other hand, assumption (i) above, similarly to condition (A), has the main role of allowing precise Hessian bounds. For a comparison to \cite{T}, note that the Hessian of $\mathcal E_4[V]$ used here would correspond to the one of the $2$-point energy $\mathcal E_2[V](\{x,y\}):=V(|x-y|)$ in the triangular lattice setup \cite{T}, in which case it is sufficient to use the quantity $V''$ instead.

\medskip

Finally, condition (B) from Theorem \ref{thm3} is still related to (i) above, and it appears due to the fact that we need to get coercivity control at interpoint distances lying in a whole interval $[1,\sqrt2]$ and not just near a minimum point of $V$ as in the study of $\mathsf A_2$.

\subsection{Summary of hypotheses on $V$ used throughout the paper}\label{sec_condV}
We include here, and briefly discuss, several requirements on $V$ that will be useful during the proofs. Firstly, we will use the change of variables
\begin{equation}\label{VforW}
 W(s):=W(r^2):=V(r), \quad s:=r^2,\quad r>0,
\end{equation}
which allows slightly more elegant Hessian computations.

\medskip

Furthermore, note that in the rest of the paper we will use three small deformation parameters which will satisfy
\[
 0<\alpha'<\alpha<\alpha''<\frac{2-\sqrt2}{4},
\]
and whose use will be the following:
\begin{enumerate}
 \item[(a)] The parameter $\alpha$ will be used to measure the deformation of distances from a configuration $X$, with respect to the distances in the model space $\mathbb Z^2$.
 \item[(b)] The parameter $\alpha'$ will measure the small neighborhood $E_{\alpha'}\subset E_\alpha$ on which the potential under consideration only takes values very close to its absolute minimum.
 \item[(c)] The parameter $\alpha''$ will give us a larger neighborhood $E_{\alpha''}\supset E_\alpha$ on which we have convexity bounds on $V$ giving good growth control, and allowing to say that perturbing the distances to stay in $E_{\alpha'}$, decreases the energy.
 \end{enumerate}
 We are now ready to enumerate the various conditions which we will impose on $V,W$ in order to get the results in Theorems \ref{thm2} and \ref{thm3} (which correspond to the more precise statements in Theorems \ref{thm_fr_crystal} and \ref{thm_lr_crystal}). {We assume that $V,W\in C^2_{pw}((0,\infty))$ are related by \eqref{VforW} and satisfy: }
 {
\begin{enumerate}
\item[(0)] $\min_{s>0}W(s)=-1$.

\noindent
It is just a renormalization, that actually does not affect our results.

\item[(1)] $V$ is convex in $E_{\alpha''}$ and $V$ satisfies $\displaystyle \inf_{r\in E_{\alpha''}\setminus [1,\sqrt 2]}V''_{\pm}(r)\ge c$.

\noindent
Such a condition  provides good quantified convexity bounds on $V$, ensuring that $\E_4[V]$ and $\E_4[V_*]$ admit at most one global minimizer among the configurations whose interpoint distances lie in $E_{\alpha''}$, whenever $V_*$ is a $C^2$-small enough perturbation of $V$.
Condition (1) appears for the first time in Lemma \ref{lem:onlysquare} and then in Lemma \ref{lemmorobusto}.

\item[(2)] $\VV$ satisfies
\begin{equation}\label{cond_crit_intro}
 \VV'(1)+2\VV'(2)=0
\end{equation}
and there exists $c'>0$ such that
\begin{equation}\label{conditions_intro}
\begin{array}{ll}
\VV_-''(1)+\VV_+''(1)+2\VV'(1)>C_4c', &\qquad\VV_-''(2)+\VV_+''(2)>C_4c',\\[3mm]
\VV_-''(1)+\VV_+''(1)>C_4c', &\qquad \VV_\pm''(1)+4\VV_\pm''(2)>C_4c',
\end{array}
\end{equation}
where $C_4$ is a constant depending only on the dimension.

\noindent Condition \eqref{conditions_intro} ensures that the configuration formed by the vertices of a unit square, from now on denoted by $\boxtimes$, is a strict local minimum - up to rotations and translations - of $\E_4[V]$; in particular, \eqref{cond_crit_intro} guarantees that $\boxtimes$ is a critical point for $\E_4[V]$, whereas \eqref{conditions_intro} 
reduces to the requirement of Hessian eigenvalues being strictly larger than $C_4c'$ if $V$ is smooth, but extends to piecewise-$C^2$ potentials $V$ which seem easier to construct explicitly;
more precisely \eqref{conditions_intro} guarantees that $\nabla \E_4[V]$ is $c'$-monotone at $\boxtimes$.
The considerations above give precisely the content of Lemma \ref{lem_minsquare}, where the existence of the constant $C_4$ is proven.
Condition (2) appears also in Proposition \ref{newlucia}, Lemma \ref{lemmorobusto}, Proposition \ref{prop_square}.

\item[(3)] $\displaystyle \sup_{r\in E_{\alpha'}}V(r)<-\frac{15}{16}-c''$, for some constant $c''\in [0,\frac{1}{16})$.

\noindent
Condition (3) requires $V$ not to be much higher than its negative minimum in $E_{\alpha'}$.
Such a condition appears for the first time in Proposition \ref{prop_square}. Loosely speaking, the combination of (1) and (3) {implies} that the well of the potential is ``large enough''.

\item[(4)] $V(r)>-\frac12$ if  $r\notin (1-\alpha,\sqrt2+\alpha)$.

\noindent
Also this condition appears for the first time in Proposition \ref{prop_square}. The combination of conditions (3) and (4) implies that $V$ ``increases'' passing from $E_{\alpha'}$ to $\R^+\setminus (1-\alpha,\sqrt2+\alpha)$.  

\item[(5)] $V(r)\ge K$ if $0<r\le 1-\alpha$, for some suitable constant $K>0$.

\noindent
This assumption allows to say that the distance between two points of a minimal configuration is strictly larger than $1-\alpha$. The value of the constant $K$ is determined in Lemma \ref{lem:mindistance}; it depends on $\alpha''$ and on the constants $\epsilon$ and $p$ of assumption (6') below.

\item[(6)] $V(r)=0$ if $r\ge \sqrt 2+\alpha''$.

\noindent
This is just a short-range assumption.

\item[(6')] $V(r)\le 0$ if $r\ge 1$ and
$|V(r)|,\,r|V'(r)|,\,r^2|V''(r)|<\epsilon r^{-p}$ if $r\ge \sqrt{2}+{\alpha''}$, for some $\epsilon>0$ small enough and $p>4$.

\noindent
This assumption is the long-range version of (6). It ensures that the tail of the potential $V$ goes fast enough to 0. Notice that, up to changing $\epsilon$ by a constant factor, it is equivalent to require
$$
|V''(r)|\le \epsilon r^{-p-2} \mbox{ for }r\ge\sqrt 2 +\alpha'',\quad V(r)\to 0\quad (r\to\infty).
$$
\end{enumerate} 

Finally, we note that a one-well potential $V(r)=W(r^2)$ which satisfies the above (0)-(5) and (6') can be given by the following formulas, for $1\le r_1\le r_2\le r_3\le \sqrt2$ and parameters $q>0$, $p>4$, $a_i>0$ and $C>0$ chosen in such a way that $W$ is $C^1$. 
\begin{equation}\label{exV3}
W(s)=\left\{\begin{array}{ll}
a_1s^{-q/2}&\mbox{ for }s\in\left(0,(1-\alpha'')^2\right],\\[3mm]
-C+a_2(s - r_1^2)^2&\mbox{ for }s\in\left[(1-\alpha'')^2, r_1^2\right],\\[3mm]
-C&\mbox{ for }s\in\left[r_1^2,r_3^2\right],\\[3mm]
-C+a_3(s-r_3^2)^2&\mbox{ for }s\in\left[r_3^2,(\sqrt2+\alpha'')^2\right],\\[3mm]
a_4(s-r_2^2)^{-p/2}&\mbox{ for }s>(\sqrt2+\alpha'')^2.\\[3mm]
\end{array}
\right. 
\end{equation}
In particular, one can verify through a tedious verification that conditions (0)-(5) and (6') hold for suitable choices of the parameters, and can be achieved even for $r_1=1,r_3=\sqrt2$ yielding $W\in C_{pw}^2$, whereas if we leave the parameters $r_1,r_3$ a bit more free we can achieve $W\in C^2$ as well.

\section{Proof of Theorem \ref{thm1}}\label{sec:proofthm1}
In this section we prove Theorem \ref{thm1}. Therefore the interaction potential $V$ is the one defined in  \eqref{hr_V} with $r_{max}=\sqrt 2$.

For every $N\in\N\cup\{+\infty\}$ we denote by 
\begin{equation}\label{config_all}
 \mathcal X_N(\mathbb R^2):=\{X\subset\mathbb R^2:\ \sharp X=N\},     
\end{equation}
the set of $N$-point configurations.
Notice that if $X_N=\{x_1,\ldots,x_N\}\in  \mathcal X_N(\mathbb R^2)$ with $\mathcal E[V](X_N)<+\infty$, then $|x_i-x_j|\ge 1$ for every $i\neq j$.

\begin{subequations}\label{config1}
Therefore we define the families of configurations having locally finite energy as      
\begin{equation}\label{config_gen}
\mathcal C:=\left\{X\subset\mathbb R^2:\, \inf_{x\neq x'\in X}|x-x'|\ge 1\right\},\qquad\mathcal C_N:={\mathcal{C}\cap  \mathcal X_N(\mathbb R^2).}
\end{equation}
We define square-lattice configurations of locally finite energy as follows: 
\begin{equation}\label{config_gen_Z2}
\mathcal C^{\Z^2}:=\left\{X\subset\mathbb Z^2:\, \inf_{x\neq x'\in X}|x-x'|\ge 1\right\}, \qquad\mathcal C_N^{\Z^2}:={\mathcal{C}_N\cap\mathcal{C}^{\Z^2}}={\mathcal C^{\Z^2}\cap \mathcal X_N(\mathbb R^2)}.
\end{equation}
\end{subequations}
We define $\mathcal E[V](N)$ as in \eqref{defen} and
\begin{equation}\label{min_e}
\mathcal E^{\Z^2}[V](N):=\min_{X_N\in\mathcal C_N^{\Z^2}}\mathcal E[V](X_N).
\end{equation}
Then clearly we have $\mathcal E[V](N)\le \mathcal E^{\Z^2}[V](N)$. In order to prove Theorem \ref{thm1},  we introduce the graph associated to a configuration in $\mathcal C$. For every $X\in\mathcal C$, we set 
$$
 \mathcal{S}_{0}(X):=\{\{x,y\}\,:\, x,y\in X,\, |x-y|\in[1,\sqrt 2]\}\,
$$
and we denote by $\mathcal{G}_0(X)$ the graph $(X, \mathcal{S}_0(X))$ whose sets of nodes and edges are given by $X$ and $\mathcal{S}_0(X)$ respectively. We say that the points $x,y\in X$ are nearest neighbors if they are connected by an edge. 
Moreover, we denote by
\begin{equation}\label{bdry0}
\partial\mathcal{G}_0(X):=\{x\in X:\ x\mbox{ has less than 8 nearest neighbors }\}.
\end{equation}
Our first result states that to leading order $\mathcal E[V](N)$ and $\mathcal E^{\Z^2}[V](N)$ have {the same asymptotics equal to $-4N+o(N)$}, and that an infinite configuration is locally minimal if and only if it is an isometric copy of $\mathbb Z^2$.
\begin{theorem}\label{thm:crystal_infinite}
Let $V$ be as in \eqref{hr_V} with $r_{max}=\sqrt2$.
\begin{enumerate}
\item[(i)] It holds
\begin{subequations}
\begin{equation}\label{limit_sq0}
-4N\le \mathcal{E}[V](N)\le -4N+O(N^{\frac 1 2})
\end{equation}
where
\begin{equation}\label{limit_sq}
-4=\lim_{N\to\infty}\frac{\mathcal E^{\Z^2}[V](N)}{N}=\lim_{R\to\infty}\frac{\mathcal E[V](\mathbb Z^2\cap B_R)}{\sharp(\mathbb Z^2\cap B_R)}.
\end{equation}
\end{subequations}
\item[(ii)] If $X\in \mathcal C$ and if a point $x\in X$ has $8$ {nearest} neighbors in $\mathcal{G}_0(X)$, each of which in turn has $8$ {nearest} neighbors in $\mathcal{G}_0(X)$, then $\overline{B}(x,\sqrt 2)\cap X$ equals up to rotation and translation $\overline{B}(0,\sqrt 2)\cap \Z^2=\{-1,0,1\}^2$.
\end{enumerate}
\end{theorem}
 In the theorem above and throughout the paper $\overline{B}(x,\rho)$ denotes the closed ball centered at $x$ and having radius $\rho$. Although we prove more general results which imply the above theorem below, we give a direct proof of (i),
 whereas we refer the reader to Corollary \ref{toprovethm} in Appendix  \ref{proof_lem:combintometric} for the proof of (ii). The proof uses some elementary geometry arguments developed in Appendix \ref{proof_lem:combintometric}.

\begin{proof}
We first prove \eqref{limit_sq0}.

Let $N\in\N$.  Trivially, it is enough to prove the first inequality only for configurations in $\mathcal{C}_N$ and the second inequality only for configurations in $\mathcal{C}^{\Z^2}_N$. 

Let $X_N\in \mathcal{C}_N$. Notice that every $x\in X_N$ has at most $8$ neighbors in $\mathcal{G}_0(X_N)$. Indeed, if there were $x\in X$ and $9$ points $x_0,\ldots,x_8\in X_N\setminus \{x\}$ such that $|x-x_i|\in [1,\sqrt 2]$ for all $i\in\{0,\ldots,8\}\simeq\Z/9\Z$ then, assuming that the points are ordered such that the angular coordinate centered at $x$ is increasing and indices are taken modulo $9$, then there exists $i\in\Z/9\Z$ such that $\widehat{x_ixx_{i+1}}$ is smaller than $360^\circ/9=40^\circ$, thus contradicting Corollary \ref{toprove(i)} in the Appendix \ref{proof_lem:combintometric}. 
As a consequence, for every $N\in\N$ and for every $X_N\in\mathcal{X}_N(\R^2)$, it holds  
\begin{equation}\label{per_bound0} 
\mathcal E[V](X_N)\ge -4N,
\end{equation}
i.e. the first inequality in \eqref{limit_sq0}.

\medskip

Let $X_N\in \mathcal{C}^{\Z^2}_N$.  We first note that each point in the neighbor graph $\mathcal{G}_0(\Z^2)$ has precisely $8$ neighbors. Thus we have, by \eqref{per_bound0}, with notation \eqref{bdry0}, for any $X_N\subset \Z^2$,
\begin{equation}\label{per_bound}
-4N\le \mathcal E[V](X_N)\le-4\sharp\left(X_N -\partial\mathcal G_0(X_N)\right) =-4N+ 4\sharp \partial\mathcal G_0(X_N),
\end{equation}
and since we may find a sequence $X_N\subset\Z^2$ such that  $\sharp(\partial \mathcal G_0(X_N))=O(N^{\frac 1 2})$ as $N\to\infty$, the second equality in \eqref{limit_sq0} follows. This concludes the proof of \eqref{limit_sq0} and shows the first equality in \eqref{limit_sq}.

{For proving the second} equality \eqref{limit_sq}, it is enough to notice that $\mathcal E[V](\Z^2\cap B_R)=-4\sharp (\Z^2\cap B_R)+O(R)$.
\end{proof}

The content of the following lemma, whose proof is obtained directly by Theorem \ref{thm:crystal_infinite}(ii), is nothing but property (b) in Subsection \ref{sec_intro_12}.

\begin{lemma}\label{rmk_squarerigid}
Let $\{x_1,x_2,x_3,x_4\equiv x_0\}\in \mathcal{C}_4$ be such that $[x_{i-1},x_i]$ are the sides of a quadrilateral $Q$ for $i=1,\ldots,4$. Assume moreover that $\{x_1,x_3\}, \{x_2,x_4\},\{x_{i-1},x_i\}\in \mathcal S_0$ for every $i=1,\ldots,4$. Then $Q$ is a square with sidelength equal to one.
\end{lemma}
\medskip

\begin{proof}
We first note that $Q$ is convex, as can be seen by applying the law of cosines. Therefore tiles congruent to $Q,-Q$ can tile the plane (to find the neighbors of $Q$, apply a reflection with respect to the midpoint of each side, and using the fact that the internal angles of $Q$ sum to $360^\circ$ obtain that this procedure can be iterated without generating overlaps). Let $X$ denote the vertices of such tessellation and let $x\in X$. By construction, we have that there exist eight points $x_1,\ldots,x_8$ such that $|x-x_i|\in \mathcal{S}_{0}$ for every $i=1,\ldots,8$. Moreover, for the same reason for every $i=1,\ldots,8$ there are eight points $x_{i1},\ldots,x_{i8}$ in $X$ with $|x_i-x_{ij}|\in \mathcal{S}_{0}$ for every $j=1,\ldots,8$. By Theorem \ref{thm:crystal_infinite}(ii), we get that $B(x,\sqrt 2)\cap X$ equals up to a rotation and a translation $B(0,\sqrt 2)\cap\Z^2$, so that the original $Q$ was a unitary square.
\end{proof}

\section{Smoothed potentials and proof of Theorem \ref{thm2}}
The goal of this section is to prove the crystallization in the sense of the thermodynamic limit for a perturbation of \eqref{hr_V}.
\subsection{Minimum distance between points for minimizers}\label{sec_mindist}
\begin{lemma}\label{lem:mindistance}
For every $C_1> 0,C_2>0$, $r_{min}>0$, $r_0>r_{min}$, and $p>2$,  there exists $K>0$, depending on $C_1,C_2,r_{min},r_0,p$ such that if
\begin{equation}\label{hyp_minsepar}
\left\{\begin{array}{ll}V(r)\geq K&\textrm{for }0<r \le r_{min}\,,\\
V(r)\geq-C_1r^{-p} &\textrm{for }r \ge r_0\,,\\
V(r)\geq-C_2& \textrm{for }r>0,\\
\displaystyle \lim_{r\to\infty}V(r)=0,&
\end{array}\right.
\end{equation}
then for every $N\in\N$ all the minimizers $X_N=\{x_1,\ldots,x_N\}$ of $\mathcal{E}[V]$ in $\mathcal{X}_N(\R^2)$ satisfy 
\begin{equation}\label{mindist}
\min_{i\neq j} |x_i-x_j|>r_{min}.
\end{equation}

Moreover, there exists a constant $K'>0$, depending on $C_2,r_{min},r_0$ such that if
\begin{equation}\label{hyp_minsepar0}
\left\{\begin{array}{ll}
V(r)\ge K'&\textrm{for }0<r \le r_{min}\,,\\
V(r)\ge -C_2& \textrm{for }r>0\,,\\
V(r)=0&\mbox{for }r \ge r_0,
\end{array}\right.
\end{equation}
then 
for every $N\in\N$ all the minimizers $X_N=\{x_1,\ldots,x_N\}$ of $\mathcal{E}[V]$ in $\mathcal{X}_N(\R^2)$ satisfy \eqref{mindist}.
\end{lemma}

\begin{rmk}
 {\rm 
 Note that assumptions (0) and (6') in Subsection \ref{sec_condV} give exactly \eqref{hyp_minsepar} with $r_{min}=1-\alpha$, $r_0=\sqrt 2+\alpha''$, $C_1=\epsilon$ and $C_2=1$, whereas conditions (0) and (6) are the same as \eqref{hyp_minsepar0} for the same choice of parameters.
 }
\end{rmk}

\begin{proof}
We prove the claim only in the case $C_1>0$ whereas the proof for $C_1=0$ is left to the reader. For simplicity, we will denote in the below by $C$ any constant depending only on $C_1,C_2$ from the theorem, which may change from line to line.

\medskip

We follow along the lines of \cite[Lemma 2.2]{T}, but for the benefit of the reader we include the proof in self-contained form. For every $N\in\N$ we set
\[M=M(r_{min}, N):=\max\sharp\left\{X_N\cap B\left(y,\frac{r_{min}}{2}\right):\ y\in \mathbb R^2,\ \left.\begin{array}{l}X_N\mbox{ is a minimizer of }\mathcal{E}[V]\\ \mbox{in } \mathcal{X}_N(\R^2)\end{array}\right.\right\}.
\]
For the remainder of the proof we fix $N\in\N$ and a minimizer $X_N=\{x_1,\ldots,x_N\}\subset\mathbb R^2$ of $\mathcal{E}[V]$ in $\mathcal{X}_N(\R^2)$ which achieves the above maximum $M$. By translation invariance, we may assume that $y=0$  and we write $B= B\left(0,\frac{r_{min}}{2}\right)$.

\medskip

We need to show that $M=1$ for $K$ large enough.

\medskip

Let $I\subset\{1,\ldots,N\}$ be the indices such that $x_i\in B$, so that $\sharp I=M$. As $V(r)\geq K$ on $(0,r_{min})$, we have
\begin{equation}\label{eq-lowerboundball}
\sum_{\substack{i,j\in I\\ i\neq j}}V(|x_i-x_j|)\geq K\ M(M-1).
\end{equation}
We now claim that
\begin{equation}\label{eq-argumentXY}
\sum_{\substack{i\in I\\j\notin I}} V(|x_i-x_j|)+\frac12\sum_{ \substack{i,j\in I\\ i\neq j}} V(|x_i-x_j|)\leq 0.
\end{equation}
Since $X_N$ is a minimizer of $\mathcal{E}[V]$ in $\mathcal{X}_N(\R^2)$, for every $Y_N=\{y_1,\ldots,y_N\}\in\mathcal{X}_N(\R^2)$ we have
\begin{align}\label{v_split}
&\mathcal E[V](X_N)=\frac12\sum_{i\neq j}  V(|x_i-x_j|)\nonumber\\
&=\frac12\sum_{\substack{i,j\in I\\i\neq j}} V(|x_i-x_j|)+\sum_{\substack{i\in I\\j\notin I}} V(|x_i-x_j|)+\frac12\sum_{ \substack{i,j\notin I\\i\neq j}} V(|x_i-x_j|)\\
&\leq \mathcal E[V](Y_N)=\frac12\sum_{\substack{i,j\in I\\i\neq j}} V(|y_i-y_j|) + \sum_{i\in I}\sum_{j\notin I}V(|y_i-y_j|).\nonumber
\end{align}
In particular we can construct configurations $Y_N$ from $X_N$ by keeping $y_j=x_j$ if $j\notin I$ while for $i\in I$ we can move $y_i$ towards infinity and away from each other, so that the quantity
\[
 \min_{\substack{i\in I\\j\neq i}}|y_i-y_j|
\]
gets arbitrarily large. Since by hypothesis $V(r)\to 0$ as $r\to+\infty$, we obtain from \eqref{v_split}
\[
 \frac12\sum_{i\neq j}V(|x_i-x_j|)\le \frac12\sum_{\substack {i,j\notin I\\i\neq j}} V(|x_i-x_j|)+ M\ (N-1)\ \lim_{r\to\infty}V(r)= \frac12\sum_{\substack {i,j\notin I\\i\neq j}} V(|x_i-x_j|),
\]
which yields \eqref{eq-argumentXY}.

\medskip

Combining \eqref{eq-argumentXY} and \eqref{eq-lowerboundball}, we get
\begin{equation}\label{eq-upperboundball}
\sum_{\substack{i\in I\\ j\notin I}} V(|x_i-x_j|)\leq -K\frac{M(M-1)}{2}.
\end{equation}
We now rewrite $\R^2\setminus B=\bigcup_{k=1}^\infty A_k$ where $A_k:=\left\{x\in \R^2:\ |x|\in\left(k\frac{r_{min}}{2},(k+1)\frac{r_{min}}{2}\right]\right\}$ for every $k\in\N$. It follows that
$$
\sum_{\substack{i\in I\\ j\notin I}} V(|x_i-x_j|)=\sum_{i\in I}\sum_{k=1}^\infty \sum_{j: x_j\in A_k} V(|x_i-x_j|).
$$
By the third condition of \eqref{hyp_minsepar}, for every $k\in \N$ it holds
\begin{subequations}\label{boundsv}
\begin{equation}
\sum_{\substack{i\in I\\ x_j\in A_k}} V(|x_i-x_j|)\geq-CM \sharp (A_k\cap X_N).
\end{equation}
Let now $k_0$ be such that $\mathrm{dist}(B,A_k)=\frac{r_{min}(k-1)}{2}\ge r_0$ for $k\geq k_0$ with $r_0$ as in \eqref{hyp_minsepar}.
By the second condition of \eqref{hyp_minsepar}, for every $k\ge k_0$,  we have
\begin{equation}
\sum_{\substack{i\in I\\ x_j\in A_k}} V(|x_i-x_j|)\geq - \frac{C\ M\ \sharp(A_k\cap X_N)}{\mathrm{dist}(B, A_k)^p}=-\frac{2^p\ C\ M\ \sharp (A_k\cap X_N)}{r_{min}^p(k-1)^p}.
\end{equation}
\end{subequations}
Moreover, by covering $A_k$ by copies of $B$ and using the maximality property of $B$, one can easily check that $\sharp (A_k\cap X_N)\leq C M k$, for some geometric constant $C>0$, independent of $k$. Thus, by appropriately summing the bounds \eqref{boundsv} and inserting into \eqref{eq-upperboundball}, we have
\begin{eqnarray}\label{compare_eq}
-K\frac{M(M-1)}{2}&\ge&\sum_{\substack{i\in I\\ j\notin I}}V(|x_i-x_j|)\nonumber\\
&=&\sum_{i\in I}\left(\sum_{k=1}^{k_0-1}\sum_{j: x_j\in A_k}V(|x_i-x_j|)+\sum_{k=k_0}^{+\infty}\sum_{j:x_j\in A_k} V(|x_i-x_j|)\right)\\
&\ge& -C\ M\left(M\frac{k_0(k_0-1)}{2} +  \frac{2^p M^2}{r_{min}^p}\sum_{k=k_0}^\infty \frac{k}{(k-1)^p}\right).\nonumber
\end{eqnarray}
Notice that if $M\ge 2$, then for $K\to +\infty$  the left-hand-side in \eqref{compare_eq} tends to $-\infty$ whereas the  right-hand-side remains finite since $p>2$; therefore, there 
exists $K=K(C_1,C_2,r_{min},r_0,p)>0$ large enough such that $M=1$.
\end{proof}
\subsection{Combinatorial setup}\label{sec_combinsetup}
From now on we slightly change notations, in order to be able to think of our configurations optimizing the energy as \emph{discrete manifolds}. 

\medskip

We have three types of data: labels of points, combinatiorial information (graphs, edges, boundaries, etc.) and metric information (distances, angles, etc.). To keep track of this we use the following notation conventions:

\begin{itemize}
\item Sets of labels, with no further structure useful to us, will be indicated by greek capital letters like $\Xi, \Lambda,\ldots$. 
\item Sets of which we are interested in the \emph{combinatorial} structure will be indicated by capital calligraphic letters like $\mathcal G, \mathcal Z_\boxtimes, \mathcal S,\ldots$. 
\item Sets of which we are interested in the \emph{metric} structure will be indicated by capital letters like $X, U,\ldots$.
\end{itemize}

\medskip

The combinatorial model-space will be
\begin{equation}\label{zboxtimes}
\mathcal Z_\boxtimes=(\Z^2,\{\{a,b\}:\ a,b\in\mathbb Z^2,\ |a-b|\in\{1,\sqrt2\}\}).
\end{equation}
In general, the notation $\mathcal{G}=(\Xi,\mathcal S)$ will be used to denote a graph with vertex set $\Xi$ and edge set $\mathcal S$.

\medskip

The first notations we introduce are
\begin{itemize}
 \item $\Xi$ are the labels of our configurations. Till now we had $\Xi=\{1,\ldots,N\}$, but putting an order structure on our labels could be confusing and we avoid it. We write $\Xi_N$ when we want to stress that $\Xi$ is a set of $N$ labels.
 \item $X\subset\mathbb R^2$ will be a finite metric subspace. We also denote by $X$ injective maps $\Xi\to \mathbb R^2$, whenever only the image $X(\Xi)$ is of interest to us
 and we use the notation $x_p:=X(p)$ for every $p\in\Xi$.  Till now we had $X_N=X(\Xi)=X(\{1,\ldots,N\})$ and $x_i=X(i)$.

\end{itemize}
We next introduce some notations reminiscent of the ones of \cite{T} adapted to our setting {(see also Figure \ref{fig-graph})}. Below $X$ and $\Xi$ are as above, and $p$ denotes a point in $\Xi$:
\begin{subequations}\label{not_graph}
\begin{eqnarray}
\mathcal S_{\alpha}=\mathcal{S}_{\alpha}(X)&:=&\left\{ \{p,q\}:\ p,q\in \Xi,\ |x_p-x_q|\in (1-\alpha,\sqrt{2}+\alpha) \right\},\\
\mathcal G_{\alpha}=\mathcal G_{\alpha}(X)&:=&(\Xi,\mathcal S_{\alpha}(X)),\\
\mathcal N_\alpha(p)=\mathcal N_\alpha(X,p)&:=&\{ q\in \Xi: \{p,q\}\in \mathcal S_{\alpha}(X)\}\cup \{p\},\\
\partial \mathcal G_{\alpha}&:=&\{ p\in \Xi: \mathcal{N}_\alpha(p)\neq 9\},\\
\mathcal G_{\alpha}|_\Lambda=\mathcal G_{\alpha}(X)|_\Lambda&:=&\big(\Lambda\ ,\ \{\{p,q\}: p,q\in \Lambda\}\cap \mathcal S_{\alpha}(X)\big).
\end{eqnarray}
\end{subequations}
\begin{figure}
\includegraphics[width=5cm]{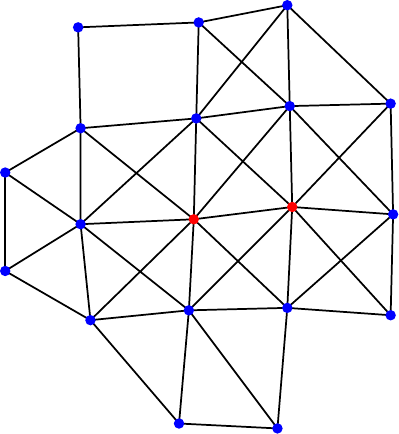}
\caption{A configuration with $19$ points, blue vertices correspond to $\partial\mathcal G_\alpha$ and red points are interior ones (with $\alpha=0.3$ here).}\label{fig-graph}
\end{figure}

Notice that $p\in \mathcal{N}(p)$ by definition. Lemma \ref{lem:mindistance}, applied with $r_{\min}=1-\alpha$ ensures that for energy-minimizing configurations there holds $|x_p-x_q|>1-\alpha$ for $p\neq q\in \Xi$, thus the energy of any minimizer can be written as follows
\begin{equation}\label{eq:energysplit1}
\mathcal{E}[V](X)=\sum_{\{p,q\}\in \mathcal{S_\alpha}(X)}V(|x_p-x_q|) + \sum_{\substack{\{p,q\}\notin\mathcal{S_\alpha}(X)\\ p,q\in\Xi,p\neq q}}V(|x_p-x_q|).
\end{equation}
Notice that the notation in \eqref{not_graph} is coherent with the one introduced in Section \ref{sec:proofthm1} for $\alpha=0$, since 
$$
\mathcal S_0=\bigcap_{\alpha>0}\mathcal{S}_\alpha=\lim_{\alpha\to 0^+}\mathcal S_\alpha.
$$
Anyway, wherever not specified, all the results of the remainder of the paper refer to the case $\alpha>0$. Whenever it is clear from the context, the dependence on $\alpha$ is omitted in the notations.
\begin{lemma}\label{lem_lowerbound}
There exists $\bar\alpha\in (0,1)$ such that for every $\alpha\in[0,\bar\alpha)$ the following holds: If $X$ satisfies \eqref{mindist} with $r_{min}=1-\alpha$, i.e.
\begin{equation}\label{mindistnew}
\min_{\substack{p,q\in\Xi\\ p\neq q}}|x_p-x_q|>1-\alpha,
\end{equation}
then $\sharp\mathcal{N_\alpha}(p)\leq 9$ for every $p\in\Xi$.
\end{lemma}
\begin{proof}
Let $p\in \Xi$ such that $\sharp \mathcal{N}_\alpha(p)\geq 10$, then there exists two points $q,q'\in \mathcal{N}_\alpha(p)\backslash \{p\} $ such that $\theta:=\widehat{x_qx_px_{q'}}\leq 40^\circ$. Therefore, by the cosine law and by the definition of $\mathcal{N}_\alpha(p)$, we deduce that
\begin{align*}
|x_q-x_{q'}|^2&=|x_p-x_{q}|^2 + |x_p-x_{q'}|^2 -2|x_p-x_{q}||x_p-x_{q'}|\cos \theta\\
&\leq (\sqrt{2}+\alpha)^2 -2(1-\alpha)^2\cos 40^\circ.
\end{align*}
We now claim that, for $\alpha$ small enough,
$$
(\sqrt{2}+\alpha)^2 -2(1-\alpha)^2\cos 40^\circ\leq (1-\alpha)^2.
$$
It is indeed straightforward to rewrite the inequality above as
$$
2 \cos40^\circ \alpha^2-\left(2\sqrt{2}+2+4\cos40^\circ\right)\alpha +2\cos40^\circ-1\ge 0
$$
and to show that this inequality is true if and only if
$$
\alpha \in [0, \overline\alpha]\cup [\tilde\alpha,\infty),
$$
where, in particular, 
$$
\overline\alpha=\frac{\sqrt{2}+1+2\cos40^\circ-\sqrt{(\sqrt{2}+1+2\cos 40^\circ)^2-2\cos 40^\circ(2\cos40^\circ-1)}}{ 2\cos 40^\circ}\approx 0.068.
$$
and $\tilde\alpha>5$.
\medskip

Thus, for $\alpha\in[0,\overline\alpha)$, we have $|x_{q}-x_{q'}|< 1-\alpha$ which contradicts \eqref{mindist} and thus proves the {lemma}.
\end{proof}
\subsection{Local rigidity of configurations}\label{sec_rigidloc}
To proceed, we next include a definition, which will help us to track the deformations of our model configurations:

\begin{definition}[$\alpha$-deformed distances]\label{def:deform}
{\rm Let $\alpha\in [0,1)$ be a constant and $(X_1,d_1)$, $(X_2,d_2)$ be two metric spaces. We say that $(X_1,d_1)$ is an \emph{$\alpha$-deformation} of  $(X_2,d_2)$ and we write
\begin{equation}\label{alphadef}
 X_1\sim_\alpha X_2,
\end{equation}
if there exists a bijection $\phi:X_1\to X_2$, called the \emph{$\alpha$-deformation map} such that
\[
 \forall x, y\in X_1,\quad (1-\alpha)d_1(x,y)\le d_2(\phi(x),\phi(y))\le (1+\alpha)d_1(x,y).
\]
If $x,y\in X$ and $\phi$ is given, we denote the \emph{$\phi$-deformation of $\{x,y\}$} by 
\[
\delta_\phi(x,y):= |d_2(\phi(x),\phi(y)) - d_1(x,y)|.
\]
If $\phi$ is clear from the context, we omit it in the notation.}
\end{definition}

We also say that $A,B\subset\mathbb R^2$ are \emph{congruent} and we write $A\simeq B$ if there exists an isometry $T:\mathbb R^2\to\mathbb R^2$ such that $T$(A)$=B$. This corresponds to the $A\sim_\alpha B$ in the case $\alpha=0$, with the notation of Definition \ref{def:deform}.
\medskip

The next result tells us that whenever we have the combinatorial structure of a square lattice in $\mathcal G$ near a point, the metric structure is not much deformed. This will be our main tool for ``transforming'' combinatorial information to metric information.

\medskip

In the following $O=O(\alpha)$ is a continuous function defined in a right neighborhood of the origin such that $\displaystyle \limsup_{\alpha\to 0^+}\frac{|O(\alpha)|}{\alpha}$ is finite.

\begin{lemma}[combinatorics links to metric control]\label{lem:combintometric}
There exists $\alpha_0\in (0,1)$ such that for all $\alpha\in (0,\alpha_0)$ the following fact holds true: For every $X$ satisfying \eqref{mindistnew} 
and for every $p\in \Xi$ with $\mathcal N_\alpha(p)\cap \partial\mathcal G_\alpha=\emptyset$ there exists a bijection $\phi:\mathcal{N}_\alpha(p)\to \{-1,0,1\}^2$ such that $\phi(p)=(0,0)$ and
\begin{equation}\label{deform}
\begin{aligned}
 &\phi(p_1)&=&(\phantom{-}1,0),&   &\phi(p_2)&=&(\phantom{-}1,\phantom{-}1),& &\phi(p_3)&=&(0,\phantom{-}1),& &\phi(p_4)&=&(-1,\phantom{-}1),&\\
 &\phi(p_5)&=&(-1,0),& &\phi(p_6)&=&(-1,-1),& &\phi(p_7)&=&(0,-1),& &\phi(p_8)&=&(\phantom{-}1,-1),&
\end{aligned}
\end{equation}
where $p_1,\ldots,p_8$ are the nearest neighbors of $p$ in $\mathcal{G}_\alpha$ ordered counterclockwise around $p$ and such that $|X(p_1)-X(p)|=\min\{|X(p_j)-X(p)|\,:\,j=1\ldots,8\}$.
Moreover,
\begin{equation}\label{deform00}
\begin{aligned}
1-\alpha\le|X(p_{j+1})-X(p_j)|= 1+O(\alpha)& \textrm{\quad for every } j=1,\ldots,7,\\
1-\alpha\le |X(p_{2j+1})-X(p)|= 1+O(\alpha)& \textrm{\quad for every }j=0,\ldots,3,\\
|X(p_{2j})-X(p)|= \sqrt{2}+O(\alpha)& \textrm{\quad for every }j=1,\ldots,4,\\
\sqrt{2}+O(\alpha)\le |X(p_{2j-1})-X(p_{2j+1})|\le \sqrt 2+\alpha &\textrm{\quad for every }j=1,\ldots,4, 
\end{aligned}
\end{equation}
with the convention that $p_9\equiv p_1$.

In particular, there exists $C_3=C_3(\alpha_0)\in [1,\frac{1}{\alpha_0})$ such that
 \begin{equation}\label{deform0}
  X(\mathcal N_{ \alpha}(p)) \sim_{C_3\alpha} \{-1,0,1\}^2\ ,
 \end{equation}
 where $ \{-1,0,1\}^2$ is endowed with the induced metric from $\mathbb R^2$.
\end{lemma}
The proof of Lemma \ref{lem:combintometric} is quite long and is postponed in the Appendix \ref{proof_lem:combintometric}.
\begin{rmk}\label{nring}
{\rm 
In view of \eqref{deform}, it immediately follows that, writing $\phi_p=\phi$, it holds: $\phi_p(p)=(0,0)$, $\phi_p(p_{2j})=\phi_{p}(p_{2j-1})+\phi_p(p_{2j+1})$ and $\phi_p(p_{2j+1})=\frac{1}{2}(\phi_p(p_{2j})+\phi_{p}(p_{2j+2}))$, with the usual convention that  the numbering of the $p_i$'s is cyclic.}
\end{rmk}
\begin{rmk}\label{remun}
{\rm The bijection $\phi$ constructed in Lemma \ref{lem:combintometric} is unique up to a composition of a graph endomorphism of $\mathcal{Z}_{\boxtimes}$.}
\end{rmk}
The following result l is a generalization of Lemma \ref{rmk_squarerigid} to the case $\alpha>0$. It can be proved using the same tools as for Lemma \ref{lem:combintometric}, so that  also its proof is postponed to Appendix \ref{proof_lem:combintometric}.
\begin{lemma}\label{cor_squarerigid}
 There exist $\alpha_0'>0, C_3'>1$ with $C_3'\alpha_0'<1$ such that for all $\alpha\in (0,\alpha_0')$, if $\mathcal G_\alpha$ is isomorphic to the complete graph over $4$ vertices, then 
\begin{equation}
\label{deform0bis}
X(\Xi)\sim_{C_3'\alpha}\{0,1\}^2\subset \mathbb R^2.
\end{equation}
\end{lemma}

\subsection{Minima of $4$-point energy and perturbed potentials}\label{sec_4ptmin}
Now we study the $4$-point energy problem. The goal is to formulate sufficiently general conditions on potential $V$ under which the square is the unique minimizer. For the computations below it will be simpler to re-express as already mentioned in \eqref{VforW},
\begin{equation}\label{WfromV}
W(s):=V(\sqrt{s}),\quad\mbox{for}\quad s>0,
\end{equation}
and to perform the computations using the formula $W(|x-y|^2)$ rather than $V(|x-y|)$ for the pairwise interactions.

\medskip

We assume that $V,W\in C^2_{pw}((0,\infty))$. In view of \eqref{WfromV} we can rewrite the four-point energy $\E_4[V]:(\mathbb R^2)^4\to\mathbb R$ defined in \eqref{defEWintro} as
 \begin{equation}\label{defEW}
 \mathcal E_4[V](x_1,x_2,x_3,x_4):=\frac12\sum_{i=1}^4\VV(|x_i-x_{i+1}|^2) + \VV(|x_1-x_3|^2) + \VV(|x_2-x_4|^2),
 \end{equation}
where we identify indices up to equivalence modulo $4$.
The vectors in $(\R^2)^4$ will be denoted by $\vec h=(h_1,h_2,h_3,h_4)$ and $\vec k=(k_1,k_2,k_3,k_4)\in(\mathbb R^2)^4$.
Moreover we denote by $\{e_1,e_2\}$ the canonical orthonormal basis of $\R^2$ and we set
 \begin{equation}\label{boxtver}
\vec q:=(q_1,q_2,q_3,q_4)=\frac 1 2 ((-1,-1),(-1,1),{(1,1),(1,-1)}).
 \end{equation}

\begin{lemma}[Taylor expansion of energy close to a square]\label{lemma_taylorsquare} 
For every $\VV\in C^2((0,\infty))$ we have 
\begin{equation}
\label{der_e4h}
\partial_{\vec h}\mathcal E_4[V](\vec q)= \left(\VV'(1)+2\VV'(2)\right)\left(\langle h_3-h_1,e_1+e_2\rangle +\langle h_2-h_4, e_2-e_1\rangle\right),\\
\end{equation}
\begin{equation}\label{hes_e4h}
\begin{aligned}
\partial^2_{\vec h, \vec k}\mathcal E_4[V](\vec q)=&\sum_{i=1}^4\left(\sum_{j=1,2}\VV'(j)\langle h_i-h_{i+j},k_i-k_{i+j}\rangle\right)\\
&+2\sum_{i=1}^4\left(\sum_{j=1,2}\VV''(j)\langle h_i-h_{i+j}, q_i-q_{i+j} \rangle\langle k_i-k_{i+j}, q_i-q_{i+j} \rangle\right).
\end{aligned} 
 \end{equation}
 If $\VV\in C^2_{pw}((0,\infty))$, then \eqref{der_e4h} holds true whereas \eqref{hes_e4h} is replaced by
\begin{equation}\label{hes_e4hpw}
\begin{aligned}
\partial^2_{\vec h, \vec k}\mathcal E_4[V](\vec q)=&\sum_{i=1}^4\left(\sum_{j=1,2}\VV'(j)\langle h_i-h_{i+j},k_i-k_{i+j}\rangle\right)\\
&+\sum_{i=1}^4\left(\sum_{j=1,2}\VV_{\pm}''(j)\langle h_i-h_{i+j}, q_i-q_{i+j} \rangle\langle k_i-k_{i+j}, q_i-q_{i+j} \rangle\right),
\end{aligned} 
 \end{equation} 
 where  $\VV_+'', \VV_-''$ are the  second derivatives of $\VV$ taken from the right and the left  respectively and $\pm$ are chosen to match the sign of $\langle h_l-h_m,q_l-q_m\rangle$.

\medskip

If $\VV\in C^2((0,\infty))$ then $\mathrm{Hess}\ \mathcal E_4[V](\vec q)$ has the following eigenvectors and eigenvalues:
\begin{itemize}
 \item $(v,v,v,v), v\in\mathbb R^2$ with eigenvalue $0$ (corresponding to infinitesimal translations),
 \item $(v,-v,v,-v), v\in\mathbb R^2$ with eigenvalue $4(\VV'(1)+\VV''(1))$ (corresponding to translating diagonals in opposite directions),
 \item $\vec q=(q_1,q_2,q_3,q_4)$, with eigenvalue $2\VV'(1)+4\VV'(2)+4\VV''(1)+16\VV''(2)$ (corresponding to infinitesimal dilations),
 \item $(q_2,q_3,q_4,q_1)$ with eigenvalue $2\VV'(1)+4\VV'(2)$ (corresponding to infinitesimal rotations),
 \item $(q_4,q_3,q_2,q_1)$ with eigenvalue $2\VV'(1)+4\VV'(2) + 4\VV''(1)$ (infinitesimal deformation which rotates the diagonals with respect to each other),
 \item $(q_1,q_4,q_3,q_2)$ with eigenvalue $2\VV'(1)+4\VV'(2) + 16\VV''(2)$ (infinitesimal deformation which squeezes one diagonal and dilates the other while keeping sidelengths constant).
\end{itemize}
\end{lemma}

The basic tool to prove Lemma \ref{lemma_taylorsquare} comes from the expansion of the $N=4$ energy for configurations close to $\{0,1\}^2\subset \mathbb R^2$. The proof is omitted because it is a direct computation. The lemma slightly generalizes the result of \cite[Lemma 6.1]{fritheil}, where only a special choice of potential modeling elastic springs was considered instead.

\medskip

If $\VV\in C^2_{pw}((0,\infty))$ then we may still apply formula \eqref{hes_e4hpw} 
 to compute the second-order variations of $\mathcal E_4[V]$ along vectors $\vec h$ expressed in the above basis of infinitesimal deformations. We cannot call these vectors ``eigenvectors'' anymore, but we can use the geometric decomposition of the above basis in order to understand, for the case of $\VV\in C^2_{pw}$, what conditions ensure that $\vec q$ is a strict local minimum. The result of this computation is stated in Lemma \ref{lem_minsquare} below.

\medskip

We first introduce some notations. By abuse of notation we write $\mathcal E_4[V]$ also for the induced functional on $4$-ples of points defined up to rotations and translations, thus we write 
\begin{equation}\label{x4}
\mathcal E_4[V]:\mathcal X_4(\mathbb R^2)/\mathrm{Isom}(\mathbb R^2)=\{X_4\subset\mathbb R^2:\ \sharp X_4=4\}/\mathrm{Isom}(\mathbb R^2)\to \R.
\end{equation}
We remark that this is possible since $\mathcal E_4[V]$ is invariant under permutations and under isometries of $\mathbb R^2$. Note that the above space $\mathcal X_4(\mathbb R^2)/\mathrm{Isom}(\mathbb R^2)$ is a manifold of dimension $5$, because $\mathcal X_4(\mathbb R^2)$ has dimension $8$ and it is quotiented by a free action of a group of dimension $3$. The scalar product of $\mathbb R^8=(\mathbb R^2)^4$ is also invariant thus induces a natural Riemannian manifold structure on $\mathcal X_4(\mathbb R^2)/\mathrm{Isom}(\mathbb R^2)$.

\medskip

Recalling the definition of $\vec q$ in \eqref{boxtver}, we denote by $\boxtimes$ the equivalence class of $\vec q$ with respect to the isometries of $\R^2$, i.e. $\boxtimes$ is the undeformed square of sidelength $1$ in $\mathcal X_4(\mathbb R^2)/\mathrm{Isom}(\mathbb R^2)$.

\medskip

Recall that if $T_pM$ is the tangent space at $p\in M$ where $(M,g)$ is a Riemannian manifold, for each $v\in T_pM$ there exists a unique geodesic $\gamma_v:[0,1]\to M$ such that $\gamma_v(0)=p$ and $\gamma_v'(0)=v$. This allows to locally define the \emph{exponential map}
$\mathrm{exp}_p:B_r(0)\subset T_pM\to U\subset M$ by $\mathrm{exp}_p(v):=\gamma_v(1)$. For $r>0$ small enough this map is bijective and is called an \emph{exponential chart}. 
We use this terminology for $M:=U$ where $U$ is a neighborhood of $\boxtimes$ in $\mathcal X_4(\mathbb R^2)/\mathrm{Isom}(\mathbb R^2)$. 
Furthermore $\nabla \mathcal E_4[V]$ calculated at $\widetilde \boxtimes$ takes values into $T_{\widetilde \boxtimes}U$, 
and $\mathrm{exp}_{\boxtimes}^*\nabla \mathcal E_4[V]$ uses the differential of $\mathrm{exp}_\boxtimes$ to go back to corresponding vectors 
in $T_{\mathrm{exp}_\boxtimes^{-1}(\widetilde \boxtimes)}\widetilde U\simeq \mathbb R^5$.

\medskip

With the above notation, we will say that the gradient $\nabla \mathcal E_4[V]:U\to \mathbb R^5\simeq T\mathcal X_4(\mathbb R^2)/\mathrm{Isom}(\mathbb R^2)$ is \emph{$c$-monotone at $\boxtimes$}, if there exist $c\in \R$ and a small neighborhood $\widetilde U\ni \boxtimes$ such that the exponential chart $\mathrm{exp}_\boxtimes : \widetilde U \subset T_\boxtimes \mathcal X_4(\mathbb R^2)/\mathrm{Isom}(\mathbb R^2)\to U$ satisfies 
\begin{equation}\label{eq:cmonotone}
\langle \mathrm{exp}_\boxtimes^*\nabla \mathcal E_4[V]( \eta')-\mathrm{exp}_\boxtimes^*\nabla \mathcal E_4[V](\eta''),  \eta'- \eta''\rangle>c\qquad \textrm{ for all }  \eta',\eta''\in \widetilde U\textrm{ with } \eta'\neq \eta''.
\end{equation}
We say that $\nabla\mathcal E_4[V]$ is \emph{strictly monotone} if the above is satisfied with $c >0$. 

\begin{rmk}{\rm 
The above terminology, is usual in convex analysis or optimal transport theory. See \cite{DontchRock} for more details.}
\end{rmk}

The next lemma has as hypothesis condition (2) of Subsection \ref{sec_condV}.

\begin{lemma}[Local minimum at the square]\label{lem_minsquare}
Let $V,\VV\in C^2_{pw}((0,\infty))$ be related by \eqref{WfromV}. The undeformed square $\boxtimes$ is a critical point of $\mathcal E_4[V]$ if and only if
\begin{equation}\label{cond_crit}
 \VV'(1)+2\VV'(2)=0.
\end{equation}
There exists $C_4>0$ depending only on the dimension such that $\nabla\mathcal E_4[V]$ is $c'$-monotone at $\boxtimes$ if
\begin{equation}\label{conditions}
\begin{array}{ll}
\VV_-''(1)+\VV_+''(1)+2\VV'(1)>C_4c', &\qquad\VV_-''(2)+\VV_+''(2)>C_4c',\\[3mm]
\VV_-''(1)+\VV_+''(1)>C_4c', &\qquad \VV_\pm''(1)+4\VV_\pm''(2)>C_4c'.
\end{array}
\end{equation}
In particular in this case $\boxtimes$ is a strict local minimum of $\mathcal E_4[V]$ up to rotations and translations.
\end{lemma}
\begin{proof}[Sketch of proof:]
It is sufficient to verify that the directional one-sided double derivatives of $\mathcal E_4[V]$ along any direction are strictly positive.

\medskip

If in a neighborhood of $\{1,2\}$ the function $\VV$ happens to be $C^2$ then the statement follows directly from Lemma \ref{lemma_taylorsquare}. If not, note that the formulas \eqref{hes_e4h} still hold separately for the cases that $h_j$ infinitesimally increase/decrease the lengths of the sides of the square $\{-1/2,1/2\}^2$ which we consider. We next discuss what happens along the infinitesimal deformations distinguished in the lemma, not coming from translations or rotations. Below, the constant $C_4$ accounts for deformations coming from bounds on $\mathrm{exp}_\boxtimes$ in $U$ and from applying the chain rule, and thus ultimately depends only on the dimension.
\begin{itemize}
 \item The infinitesimal dilations $\vec h=\lambda\vec q$ either contemporarily increase or contemporarily decrease all lengths of all sides and diagonals, thus we get the strict local minimum conditions $\VV_+''(1)+4\VV_+''(2)>C_4{c'}$ and $\VV_-''(1)+4\VV_-''(2)>C_4{c'}$.
 \item The infinitesimal perturbations along $(q_4,q_3,q_2,q_1)$ infinitesimally preserve lengths of diagonals, squeeze two sides and dilate the other two, thus we get the condition $\VV_-''(1)+\VV_+''(1)>C_4{c'}$.
 \item The infinitesimal perturbations along $(q_1,q_4,q_3,q_2)$ infinitesimally preserve lengths of sides, squeeze one diagonal and dilate the other, and give the strict local minimum condition $\VV_-''(2)+\VV_+''(2)>C_4{c'}$.
 \item Finally, the case $\vec h=(v,-v,v,-v)$ does not alter the lengths of diagonals, and thus contributes by only the term with $\VV_\pm''(1)$ in \eqref{hes_e4h}, which gives the contribution proportional to $I(v)$, where with notation $v:=(v_1,v_2)$ and $\sigma_j:=\mathrm{sign}v_j$ we have
 \[
 \begin{aligned}
  I(v)&:=\VV_{\sigma_2}''(1)|v_2|^2 + \VV_{-\sigma_1}''(1)|v_1|^2 + \VV_{-\sigma_2}''(1)|v_2|^2+\VV_{\sigma_1}''(1)|v_1|^2\\
  &\phantom{:}=(\VV_+''(1)+\VV_-''(1))|v|^2,
\end{aligned}
 \]
where in the first above sum, the contributions of indices $(i,j)=(1,2),(2,3),(3,4),(4,1)$ are summed in this order. Summing the above to the $\VV'(1)$-term, we get the condition $\VV_+''(1)+\VV_-''(1)+2\VV'(1)>C_4{c'}$.
\end{itemize}
The fact that $C_4>0$ as in the statement can be encountered follows by standard Taylor-type approximation of $\mathcal E_4[V]$.
\end{proof}

\subsubsection{Perturbations of $\mathcal E_4[V]$ near its minimum}\label{sec_perturb}

We recall that $E_\alpha, E_\alpha^1$ and $E_\alpha^2$ are defined by \eqref{eq_Ealpha}. Taking indices $1,\ldots,4$ modulo $4$, we set
\begin{equation}\label{defsqu}
\begin{aligned}
\mathscr{Q}_\alpha:=\{\eta&=\{\eta_1,\eta_2,\eta_3,\eta_4\equiv\eta_0\}\in\mathcal X_4(\R^2)\,:\,
 |\eta_{i}-\eta_{i+1}|\in E^1_\alpha\quad\textrm{for all }i=1,\ldots,4,\\
 &|\eta_1-\eta_3|, |\eta_2-\eta_4|\in E^2_\alpha\},\qquad 
  \overline{\mathscr{Q}}_{\alpha}:=\mathscr{Q}_{\alpha}/\mathrm{Isom}(\R^2),
\end{aligned}
\end{equation}
and 
\begin{equation}\label{defsqu2}
\begin{aligned}
\mathscr{S}_\alpha:=\{&\eta=\{\eta_1,\eta_2,\eta_3,\eta_4\equiv\eta_0\}\in \mathscr{Q}_\alpha\,:\, |\eta_{i}-\eta_{i+1}|=l \quad\textrm{for all }i=1,\ldots,4,\\
&|\eta_1-\eta_3|, |\eta_2-\eta_4|=\sqrt2 l\textrm{ for some }l\in E_\alpha^1 \cap \frac1{\sqrt2}E_\alpha^2\}, \quad \overline{\mathscr{S}}_{\alpha}:=\mathscr{S}_{\alpha}/\mathrm{Isom}(\R^2).
\end{aligned}
\end{equation}
 With a little abuse of notations, we will call quadrilaterals the $4$-ples $\eta$ of $\mathscr{Q}_\alpha$; moreover, for every $\eta\in \mathscr{Q}_\alpha$ we refer to  the quantities $|\eta_i-\eta_{i+1}|$  as side-lengths of $\eta$ and to the quantities $|\eta_1-\eta_3|$ and $|\eta_2-\eta_4|$ as diagonal-lengths.

The following lemma {holds} under the hypothesis from condition $(1)$ in Subsection \ref{sec_condV}.
\begin{lemma}
\label{lem:onlysquare}
 Let $V\in C^2_{pw}((0,\infty))$ and let $\alpha''\in (0,\alpha_0')$ (with $\alpha_0'$ given by Lemma \ref{cor_squarerigid}) be such that 
 \begin{enumerate}
 \item $V$ is convex in $E_{\alpha''}$,
 \item $V$ is strictly convex in $ E_{\alpha''}\setminus[1,\sqrt 2]$.
 \end{enumerate}
Then there exists at most one global minimizer $\tilde Q$ of $\mathcal E_4[V]$ in  $\overline{\mathscr{Q}}_{\alpha''}$.
Moreover, if such a minimizer exists, then it belongs to  $\overline{\mathscr{S}}_{\alpha''}$.
Furthermore, for every $\alpha'\in (0,\alpha'')$ there are {no} minimizers of $\mathcal E_4[V]$ in $\overline{\mathscr{Q}}_{\alpha'}\setminus\{\tilde Q\}$. 
\end{lemma}
\begin{proof}
 Let $\eta:=\{x_1,x_2,x_3,x_4\}\in\mathscr{Q}_{\alpha''}$ and let $a\le b\le c\le d$ denote the sidelengths of $\eta$.

{\textbf{Case 1.} \emph{$V$ is strictly convex in $E_{\alpha''}$.}}
We will show that if $\eta\notin\mathscr{S}_{\alpha''}$, then the energy $\mathcal E_4[V](\eta)$ can be decreased to first order by infinitesimal perturbations.

 \begin{itemize}
  \item If $|x_1-x_3|<|x_2-x_4|$ then there exists an infinitesimal perturbation which increases $|x_1-x_3|$ and decreases $|x_2-x_4|$ while keeping the remaining distances fixed. The convexity of $V$ for in $E^2_{\alpha''}$ shows that under this deformation $\mathcal E_4[V]$ decreases. Thus $\eta$ has equal length diagonals.
  \item If $a<d$, then there exists an infinitesimal perturbation which preserves the length of the diagonals, preserves the ordering of sidelengths and increases $a,b$ while diminishing $c,d$, and such that at least one of them increases/decreases is strict. The convexity of $V$ in $E^1_{\alpha''}$ shows that $V(a)+V(d)$ and $V(b)+V(c)$ strictly decrease under this perturbation. Thus $\eta$ has equal sidelengths.
 \end{itemize}
The above two points show that $\eta\in\mathscr{S}_{\alpha''}$.

\medskip

If there were two minima $\eta,\eta'$
with sidelengths $r<r'\in E^1_{\alpha''}$ then we would have $2V(r)+2V(\sqrt2 r)=2V(r')+2V(\sqrt2 r')$ and by strict convexity of $V$ in the intervals $(r,r')$ and $(\sqrt2 r, \sqrt2 r')$ we would have that the square of any intermediate sidelength would have smaller energy $\mathcal E_4[V]$, thus giving a contradiction to the minimality of $\eta,\eta'$. This completes the proof of the uniqueness result.
 Finally, since the constraints defining $E_\alpha$ are open {and $V$ is continuous}, the same reasoning proves also the last sentence of the statement.
\medskip

\textbf{Case 2.} \emph{ General case.} To adapt the proof from case 1 to a proof under more general hypotheses (1),(2), we proceed in {two} steps. 
\begin{enumerate}
\item[-] First note that a contradiction is reached if we know that $V$ is strictly convex at \emph{just a single} point of the subset $[a,d]\cup [|x_1-x_3|, |x_2-x_4|]$, while being convex on the whole subset.
\item[-] Finally, as a consequence of Lemma \ref{cor_squarerigid}, if $\{x_1,x_2,x_3,x_4\}$ were not a unit square, then $[a,d]\cup [|x_1-x_3|, |x_2-x_4|]$ would intersect $(1-\alpha'',1]\cup[\sqrt2,\sqrt2+\alpha'')$, and then the strict convexity hypothesis as described in the first item allows to find a contradiction, as desired.
\end{enumerate}
\end{proof}

By combining Lemma \ref{lem_minsquare}and Lemma \ref{lem:onlysquare}, {the first results insuring the undeformed square $\boxtimes$ to be a strict local minimizer of $\mathcal{E}_4[V]$ in $\overline{\mathscr{Q}}_{\alpha_0''}$ and the second one giving its uniqueness,} we have the following result.
\begin{proposition}\label{newlucia}
Let $V,W\in C^2_{pw}((0,\infty))$ be related as in \eqref{VforW}. Assume that $V$ satisfies the assumptions of Lemma \ref{lem:onlysquare} and $W$ satisfies \eqref{cond_crit} and \eqref{conditions}. Then, there exists $\alpha_0''\in (0,\alpha_0')$ (with $\alpha_0'$ given by Lemma \ref{cor_squarerigid}) such that $\boxtimes$ is the only minimizer of $\E_4[V]$ in $\overline{\mathscr{Q}}_{\alpha_0''}$. Moreover, for every $\alpha''\in (0,\alpha_0'')$ there exists no minimizer of $\E_4[V]$ in $\overline{\mathscr{Q}}_{\alpha''}\setminus\{\boxtimes\}$.
\end{proposition}

The {following} results Lemma \ref{lemmorobusto} and Proposition \ref{prop_square} below are generalizations of Proposition \ref{newlucia} to perturbations of potentials $W$ satisfying \eqref{cond_crit} and \eqref{conditions}. Such results are used in the proof of Theorem \ref{thm_lr_crystal} which deals with asymptotic crystallization for {long-range} potentials, but we decide to include them in this subsection since they concern the minimization of a $4$-point energy.

For every $\alpha>0$, we set 
\begin{equation}\label{ealphasq}
E_\alpha^{sq}:=((1-\alpha)^2,(1+\alpha)^2)\cup ((\sqrt{2}-\alpha)^2,(\sqrt{2}+\alpha)^2).
\end{equation}

\begin{lemma}[general perturbations of potentials]\label{lemmorobusto}
Let $c', \alpha'>0$. Let $\VV\in C^{2}_{pw}((0,\infty))$ satisfy \eqref{cond_crit} and \eqref{conditions}.  
Then there exists $c_5=c_5( c',\alpha',W)>0$ such that if $\VV_*$ is a perturbation of $\VV$ with $\|\VV'-\VV_*'\|_{C^1( E^{sq}_{\alpha'})}<c_5$,
then the $4$-point energy $\mathcal E_4[V_*]$ has precisely one local minimum $\overline\boxtimes$ in $\overline{\mathscr{Q}}_{\frac{\alpha'}{2}}$ and $\VV_*$ satisfies \eqref{conditions} with $c=\frac{c'}{2}$.
\end{lemma}
\begin{proof}
Due to the assumption that \eqref{cond_crit} holds, it follows from Lemma \ref{lem_minsquare} that $\mathcal E_4[V]$ has a critical point at $\boxtimes$. 
 
\medskip
 
We can then apply the implicit function theorem for strictly monotone functions (see for instance \cite[Thm. 1H.3]{DontchRock}, whose proof directly extends to the case of $F(w, x):=\nabla \mathcal E_4[v](\boxtimes + x)$, where $w$ is a perturbation of $\VV$ in $C^2$-norm, $v(r)=w(r^2)$, and $\mathcal E_4[v]$ is defined using formula \eqref{defEW} with $w$ instead of $\VV$). This allows to verify that there exists $c_5=c_5(c',\alpha', W)>0$ such that for each $\VV_*$ such that $\VV_*'$ is closer than $c_5$ to $\VV'$ in $C^1$-norm, there exists a unique zero of $\nabla \mathcal E_4[V_*]$ in $\overline{\mathscr{Q}}_{\frac{\alpha'}{2}}$. Up to restricting $c_5$, the bounds \eqref{conditions} hold also for $\VV_*$ with $c=\frac{c'}{2}$ if $\|\VV'_*-\VV'\|_{C^1(E^{sq}_{\alpha'})}<c_5$. Thus the monotonicity conditions continue to hold for $\nabla \mathcal E_4[V_*]$ in $\overline{\mathscr{Q}}_{\frac{\alpha'}{2}}$ as well, allowing to verify the uniqueness and strict minimality claim of the lemma.
\end{proof}

By combining the above lemmas we have the following result, which applies to potentials $V$ as in Figure \ref{V2}.
We recall that $\alpha'_0>0$ and $C_3'\in[1,\frac{1}{\alpha'_0})$ are the constants found in Lemma \ref{cor_squarerigid}. 

\begin{proposition}[unique square minimum robust under $C^2$ perturbations]\label{prop_square}
Let $\alpha''\in (0,\alpha'_0)$ and $0<\alpha'\le \alpha\le\frac{\alpha''}{C'_3}$. 
Let $V,\VV\in C^2_{pw}((0,\infty))$ be related by \eqref{WfromV}. Assume that there exists $c, c', c''>0$ such that: 
 \begin{itemize}
\item[(0)] $\min_{s>0}W(s)=-1$;
\item[(1)] $V$ is convex in $E_{\alpha''}$ and $V$ satisfies $\inf_{r\in E_{\alpha''}\setminus [1,\sqrt 2]}V''_{\pm}(r)\ge c$;
\item[(2)] $\VV$ satisfies \eqref{cond_crit} and \eqref{conditions} with constant $c'$;
 \item[(3)] $\sup_{r\in E_{\alpha'}}V(r)<-\frac{15}{16}-c''$;
 \item[(4)] $V(r)>-\frac12$ if  $r\notin (1-\alpha,\sqrt2+\alpha)$.
 \end{itemize}
 
Recalling that $c_5$ is the constant defined in {Lemma \ref{lemmorobusto}} and setting $c''':=\min\{c_5(c',\alpha' ,W),c''/2, c/2\}$, for every perturbation $V_*\in C^2_{pw}((0,\infty))$ {of $V$} with 
\begin{equation}\label{perturbhp}
\|\VV-\VV_*\|_{C^2((0,\infty))}<c''',
\end{equation}
 we have that the $4$-point energy $\mathcal E_4[V_*]$ has exactly 
 one global minimizer in $\mathcal X_4(\R^2)/\mathrm{Isom}(\R^2)$ and such a minimizer, denoted by $\overline\boxtimes$, lies actually in $\mathscr{S}_{\frac{\alpha'}{2}}$.
\medskip

If $V_*\in C^2_{pw}((0,\infty))$ satisfies only
\begin{equation}\label{perturbhpused}
\|\VV - \VV_*\|_{C^2(((1-\alpha'')^2,+\infty))}<c''',
\end{equation}
then the energy $\mathcal E_4[V_*]$ has exactly one minimum with sidelengths in the interval $(1-\alpha,+\infty)$, which is a square. 
\end{proposition}
\begin{proof}
We will only prove the first part of the statement, as the proof for the second part is analogous, restricting the domain to $(1-\alpha'',+\infty)$ instead.

\medskip

We first show that all the points in a minimizer of $\mathcal E_4[V_*]$ stay at a distance in $(1-\alpha,\sqrt2+\alpha)$ from each other. Indeed, suppose we have a minimizing configuration $Q$ which has two vertices $x,x'$ whose distance is outside this interval. By using in order of appearance, \eqref{perturbhp}, (3), and the fact that {$c'''\leq \frac{c''}{2}$}, we get 
\begin{equation}\label{unrobusto}
\mathcal E_4[V_*](\boxtimes)\le 4\sup_{r\in E_{\alpha'}} V(r) + 4c'''<-\frac{15}{4}-4c''+4c'''\le-\frac{15}{4}-2c'';
\end{equation}
moreover, by assumptions (0) and (4), \eqref{perturbhp}, and using again the fact that $c'''<\frac{c''}{2}$, we obtain 
\begin{equation}\label{deuxrobusto}
\E_4[V_*](Q)\ge\frac72\min_{\rho>0}V_*(\rho)+\frac12V_*(|x-x'|)\ge- \frac72-\frac14 -4c''' \geq -\frac{15}{4}-2c'',
\end{equation}
which combined with \eqref{unrobusto} contradicts the assumption on the minimality of $Q$.

Since $\alpha<\alpha_0'$, by Lemma \ref{cor_squarerigid}, the edge lengths of any minimizer are at most $C_3'\alpha$-deformed compared to the edge lengths of a unit square, which by the upper bound on $\alpha$ implies that all minimizers have edge lengths in $E_{\alpha''}$, as desired. 

It follows that every minimizer of $\E_4[V_*]$ lies in $\mathscr{Q}_{\alpha''}$. Therefore, by assumption (1) and \eqref{perturbhp} we have that $V_*$ still satisfies the assumptions of Lemma \ref{lem:onlysquare}.
Therefore, there exists at most one minimizer of $\E_4[V_*]$ in $\mathcal{X}_4(\R^2)/\mathrm{Isom}(\R^2)$ and that such minimizer, if exists, lies in $\overline{\mathscr{S}}_{\frac{\alpha'}{2}}$.

Finally, by assumption (2) and by \eqref{perturbhp}, we can apply  Lemma \ref{lemmorobusto} thus obtaining that a unique minimizer of $\E_4[V_*]$ exists in  $\overline{\mathscr{Q}}_{\frac{\alpha'}{2}}$; this fact together with the last part of Lemma \ref{lem:onlysquare} implies the full claim.

\end{proof}

\subsection{Crystallization under smooth one-well potentials with finite range}
By using the tools from Sections \ref{sec_mindist}, \ref{sec_combinsetup} and \ref{sec_rigidloc} we can prove our second main theorem, stated below in detail. 

\medskip

Recall that the constant $\overline{\alpha}$ is given by Lemma \ref{lem_lowerbound},
 the constants $\alpha_0>0$, $C_3>1$ are provided by Lemma \ref{lem:combintometric}, $\alpha'_0>0$ is given by Lemma \ref{cor_squarerigid} , and $\alpha_0''$ is given by Proposition \ref{newlucia}.
\begin{theorem}[crystallization under finite-range smooth potentials]\label{thm_fr_crystal}\hfill
Let $\alpha,\alpha',\alpha''>0$ be such that $\alpha'<\alpha<\frac{1}{C_3}\alpha''<\frac{1}{C_3}\min\{(2-\sqrt2)/4,\overline{\alpha}, \alpha_0,\alpha'_0, \alpha_0''\}$.

Set $r_{min}:=1-\alpha$, $C_2:=1$ and $r_0:=\sqrt{2}+\alpha''$, and let $K'>0$ be as in Lemma \ref{lem:mindistance}.

There exist constants $c', K''>0$ such that the following holds true: 
if  $V,W\in C^2_{pw}((0,\infty))$ are related by \eqref{WfromV} and satisfy
\begin{itemize}
\item[(0)] $\min_{s>0}W(s)=-1$,
\item[(1)] $V$ is convex in $E_{\alpha''}$ and strictly convex in $E_{\alpha''}\setminus[1,\sqrt 2]$,
\item[(2)] $\VV$ satisfies \eqref{cond_crit} and \eqref{conditions} with constant $c'$,
 \item[(3)] $\sup_{r\in E_{\alpha'}}V(r)<-\frac{15}{16}$;
 \item[(4)] $V(r)>-\frac12$ if  $r\notin (1-\alpha,\sqrt2+\alpha)$,
\item[(5)] $V(r)>\max\{K',K''\}$ for $r\le 1-\alpha$,
\item[(6)] $V(r)=0$ for $r\ge \sqrt{2}+ \alpha''$,
\end{itemize}
then
\begin{equation}\label{energy_fr}
N
\overline{\mathcal E}_{\mathrm{sq}}[V]\le \mathcal E[V](N)\le N
\overline{\mathcal E}_{\mathrm{sq}}[V]+O(N^{1/2})\qquad\textrm{as }N\to+\infty,
\end{equation}
and furthermore
\begin{equation}\label{energy_fr2}
\min\mathcal E_4[V] = \overline{\mathcal E}_{\mathrm{sq}}[V].
\end{equation}
\end{theorem}
\begin{rmk}
Assumption (3) corresponds to assumption (B) of the introduction with $\epsilon'=\frac{1}{16}$. Furthermore, the constant $K''$ is here to ensure that the minimal energy per point of a square $\overline{\mathcal E}_{\mathrm{sq}}[V]$ is achieved for some $\underline t>1-\alpha$. The number $K>0$ appearing in Theorem \ref{thm2} and Figure \ref{V2}(left) is then $\max\{K',K''\}$.
\end{rmk}
\begin{proof}
We first prove \eqref{energy_fr2}. Let $\underline t>0$ be the value at which the minimum from \eqref{eppz2_intro} in the definition of $\overline{\mathcal E}_{\mathrm{sq}}[V]$ is achieved.

\medskip

We first show that there exists $K''>0$, independent of the remaining parameters, so that $\underline t>1-\alpha$. This will fix the choice of $K''$, and we use a strategy reminiscent of Lemma \ref{lem:mindistance}. Note first that $\overline{\mathcal E}_{\mathrm{sq}}[V]\le 0$, by considering condition (6) and testing the minimization over $t>0$ with $t=1$.
Next, using assumptions (0), (4), (5) and (6), we write
\begin{eqnarray}\label{try3}
\E[V](\underline t\Z^2\cap B_R)& =& \frac{1}{2}\sum_{\substack{x\neq y\in\underline t\Z^2\cap B_R\\ |x-y|\le 1-\alpha}}V(|x-y|)+\frac{1}{2}\sum_{\substack{x, y\in\underline t\Z^2\cap B_R\\ |x-y|\in (1-\alpha, \sqrt2 +\alpha'')}}V(|x-y|)\label{initial}\\
&\ge&  \frac12 K''\sharp\{x\in \underline t\Z^2\cap B_{R-1-\alpha}\}\left(\sharp\{x\in\underline t\Z^2:\ |x|\le 1-\alpha\}-1\right)\label{part1}\\
&&-\frac12\sharp\{x\in \underline t\Z^2\cap B_{R+\sqrt2+\alpha''}\}\sharp\{x\in\underline t\Z^2:\ |x|\le \sqrt2+\alpha''\}\label{part2}.
\end{eqnarray}
To obtain \eqref{part1} we used (5) and estimated the first sum in \eqref{initial} from below by noting that for each $p\in \underline t\Z^2\cap B_{R-1-\alpha}$, all points $x\in\underline t\Z^2\cap \overline B_{1-\alpha}(p)$ contribute at least an energy of $K''$, and we double-count contributions at most twice. Similarly, for \eqref{part2} we used (0) and perform a similar bound from above, in which we do not account for the double-counting anymore. Now it remains to note that the first factors in \eqref{part1}, \eqref{part2} are both of the form
\[
 \pi\left(\frac{R}{\underline t}\right)^2 +O\left(\frac{R}{\underline t}\right) \quad \mbox{ as }R\to\infty.
\]
Finally, by a very rough packing bound, if $\underline t<1-\alpha$ 
(condition required for the term \eqref{part1} to be nonzero) the second factors in \eqref{part1}, \eqref{part2} both lie in the interval $[c/\underline t^2, C/\underline t^2]$, where $0<c<C$ are constants which can be chosen independently of the choices of $\alpha,\alpha''$ in the interval $(0,(2-\sqrt2)/4)$. These considerations allow to continue the estimate \eqref{initial} and obtain
\begin{equation}
\E[V](\underline t\Z^2\cap B_R)\ge  \left[\frac12\pi\left(\frac{R}{\underline t}\right)^2 +O\left(\frac{R}{\underline t}\right)\right] (cK''- C), 
\end{equation}
which for $K''>C/c$ contradicts our former conclusion $\overline{\mathcal E}_{\mathrm{sq}}[V]\le 0$, as desired. It follows that $\underline t>1-\alpha$. 
Since $\alpha<\frac{\alpha''}{C_3}<\alpha''$ and $0<\alpha''<\frac{2-\sqrt 2}{4}$, we get
$$
2\underline t>2(1-\alpha'')>2-4\alpha''+\alpha''>\sqrt 2 +\alpha'',
$$
from which we deduce that only distances $\underline t, \underline{t}\sqrt{2}$ can participate to the computation of the energy and thus
\begin{equation}\label{reexprmin}
\begin{aligned}
 \mathcal E[V]\left(\underline t\mathbb Z^2\cap B_R\right)&= \sharp\left(\underline t\mathbb Z^2\cap B_R\right)\mathcal E_4[V](\{0,\underline t\}^2) + O(R)\\
 &\ge \sharp\left(\underline t\mathbb Z^2\cap B_R\right)\min\mathcal E_4[V]+ O(R).
 \end{aligned}
\end{equation}
Therefore, by dividing both terms by $\sharp\left(\underline t\mathbb Z^2\cap B_R\right)\sim O(R^2)$ and sending $R\to \infty$, we get the inequality ``$\le$'' in \eqref{energy_fr2}. On the other hand, by assumptions (1) and (2), we can apply Proposition \ref{newlucia} thus obtaining $\min \mathcal E_4[V]=\mathcal E_4[V] (\boxtimes)$. Then, doing the same computation as in the first line of \eqref{reexprmin} with $\underline t$ replaced by $1$ we get the inequality ``$\ge$'' in \eqref{energy_fr2}.
\medskip

We now prove the first inequality in \eqref{energy_fr}. Let then $\mathcal S_\alpha$ be as in \eqref{not_graph}. By assumptions (0), (5) and (6), we can apply Lemma \ref{lem:mindistance}, thus getting that 
\begin{equation}\label{mindistalfa}
\min_{\substack{p,q\in\Xi_N\\p\neq q}}|x_p-x_q|>1-\alpha\qquad\textrm{ for any minimizer $X_N=X(\Xi_N)$ in $\mathcal{X}_N(\R^2)$.}
\end{equation}
Setting
$$
\mathcal S_{\alpha'',\alpha}:=\{\{p,q\}\,:\, p,q\in \Xi_N\,,\sqrt 2+\alpha\le |x_p-x_q|<\sqrt2+\alpha''\},
$$
by \eqref{mindistalfa} and by the fact that $V(r)=0$ for $r>\sqrt 2+ \alpha''$, we deduce that
\[
 \mathcal E[V](X_N)=\sum_{\{p,q\}\in\mathcal S_\alpha\cup\mathcal S_{\alpha'',\alpha}}V(|x_p-x_q|).
\]
Then by Lemma \ref{lem:combintometric} (in particular by \eqref{deform00}) following from the fact that $\alpha''<\alpha_0$, we find that for every $p\in\{1,\ldots,N\}$ with $\mathcal{N_\alpha}(p)\cap\partial \mathcal G_\alpha=\emptyset$, the neighborhood $\mathcal N_{\alpha}(p)$ of $p$ is locally a $C_3\alpha$-deformation of $\{-1,0,1\}^2$. We next write $\mathcal E[V](X_N)$ as a sum over $C_3\alpha$-deformations of $\{0,1\}^2$. We will use the following notation:
\begin{equation}\label{defq1}
 \mathcal Q_1:=\left\{Q\subset\Xi_N:\ X(Q)\sim_{C_3\alpha}\{0,1\}^2\right\}.
\end{equation}
We then rewrite
\begin{equation}\label{resumq1}
 \mathcal E[V](X_N)=\sum_{Q\in\mathcal Q_1}\mathcal E_4[V](Q) + \frac12\sum_{\{x,y\}\in\mathcal{NC}^{(1)}}V(|x-y|)+ \sum_{\{x,y\}\in\mathcal{NC}^{(2)}}V(|x-y|),
\end{equation}
in which $\mathcal{NC}^{(j)}$ is the collection of all pairs $\{x,y\}$ which appear in the first sum on the right with multiplicity $2-j$, for $j=1,2$. We will also denote $\mathcal{NC}:=\mathcal{NC}^{(1)}\cup\mathcal{NC}^{(2)}$.

\medskip

Since $C_3\alpha<\alpha'' <\alpha''_0$, we get that every $Q\in\mathcal Q_1$ satisfies $X(Q)\in \overline{\mathscr{S}}_{\alpha_0''}$, which in view of Proposition \ref{newlucia} implies
\begin{equation}\label{resumq1-2}
 \mathcal E_4[V](Q)\ge \min\mathcal E_4[V]=\mathcal E_4[V]({\boxtimes}) \qquad \forall Q\in\mathcal Q_1,
\end{equation}
and 
\begin{equation}\label{resumq1-3}
\mathcal E_4[V]({\boxtimes})\le 4  \sup_{\rho\in E_{\alpha'}} V(\rho):=4m_{\alpha',V}.
\end{equation}
\medskip

Note that if $\{p,q\}\in\mathcal{NC}$ then we have $\mathcal N_\alpha(p)\cap\partial \mathcal G_\alpha\neq \emptyset$ or $\mathcal N_\alpha(q)\cap\partial \mathcal G_\alpha\neq \emptyset$ or $\{p,q\}\in \mathcal S_{\alpha'',\alpha}$. We will use below without mention the consequences that all such edges belong to $\mathcal S_{\alpha''}$. Now we denote 
\[
 \mathcal E^p[V](X_N):=\frac12\sum_{\substack{q\in \Xi_N\\ q\neq p}}V(|x_p-x_q|)=
 \frac12\sum_{\substack{q\in \mathcal{N}_{\alpha''}(p)\\q\neq p}}V(|x_p-x_q|),
\]
and we note that $\mathcal E[V](X_N)$ is the sum of $\mathcal E^p[V](X_N)$ over all labels $p$. Furthermore we have the following bound, which uses assumptions (0) and (4):
\begin{eqnarray}\label{boundsEp}
 \mathcal E^p[V](X_N)
 &\ge&\frac18\mathcal E_4[V]( \boxtimes)\ \sharp\{q\in \mathcal N_\alpha(p)\setminus\{p\}:\ \{p,q\}\notin\mathcal{NC}\}\nonumber\\
 &&-\frac12\sharp\{q\in\mathcal N_\alpha(p)\setminus\{p\}:\ \{p,q\}\in\mathcal{NC}\}\nonumber\\
 &&-\frac14\sharp\{q\in\mathcal N_{\alpha''}(p)\setminus\mathcal N_\alpha(p)\}.
\end{eqnarray}
We concentrate first on $p$ such that $\mathcal N_\alpha(p)\cap\partial\mathcal G_\alpha=\emptyset$. Note that due to the hypothesis $\alpha''<\overline{\alpha}$, by Lemma \ref{lem_lowerbound} we have $\sharp \mathcal N_{\alpha''}(p)\le 9$, and since $p\notin\partial \mathcal G_\alpha$, we have $\sharp \mathcal N_\alpha(p)=9$, so that all $8$ edges containing $p$ in $\mathcal G_{\alpha''}$ are actually in $\mathcal G_\alpha$. By the hypothesis $\mathcal N_\alpha(p)\cap\partial\mathcal G_\alpha=\emptyset$ all such edges are covered with full multiplicity in the first sum on the right of \eqref{resumq1}, so they don't belong to $\mathcal{NC}$. Thus, the lower bound in \eqref{boundsEp} is in this case just $\mathcal E_4[V]( \boxtimes)$.

\medskip

Now consider the remaining points, for which we used the rough bounds corresponding to the last two lines in \eqref{boundsEp}: either they have at least one ``long'' edge, with ``good'' energy bound  $-1/4$ or this never happens and we have the ``nasty'' bound $-1/2$ for some edge. In the latter case, by the previous paragraph we also must have the ``good'' property $\mathcal N_\alpha(p)\cap\partial\mathcal G_\alpha\neq\emptyset$, and some $q\in\mathcal N_\alpha(p)$ has $\mathcal N_\alpha(q)$ of lower than maximum cardinality.

\medskip

To make the reasoning of the previous paragraph rigorous, we enrich the graph $\mathcal G_\alpha$ by adding (i) the edges from $\mathcal S_{\alpha'',\alpha}$ and (ii) a set of $9-\sharp \mathcal N_{\alpha''}(p)$ new vertices denoted $\bar q_p$, and together with each of them we add what we call a ``missing edge'' of the form $\{p,\bar q_p\}$. Let $\overline{\mathcal G}$ be the new graph and $\overline{\mathcal N}(p)$ the neighborhood of $p\in \Xi$ in $\overline{\mathcal G}$. Then \eqref{boundsEp} can be re-expressed as
\begin{subequations}
\begin{equation}\label{boundsEp2}
 \mathcal E^p[V](X_N)\ge \mathcal E_4[V]( \boxtimes)+\sum_{q\in\overline{\mathcal N}(p)\setminus\{p\}}\overline w(\{p,q\}),
\end{equation}
where in order to recover \eqref{boundsEp} we define
\begin{equation}\label{defbarw}
\overline w(\{p,q\}):=\left\{\begin{array}{ll}
 0&\mbox{ if }\{p,q\}\in\mathcal S_\alpha\setminus\mathcal{NC},\\
 -\frac{\mathcal E_4[V](\boxtimes)}{2}-\frac12&\mbox{ if }\{p,q\}\in\mathcal{NC}\cap \mathcal S_\alpha,\\
 -\frac{\mathcal E_4[V](\boxtimes)}{2}-\frac14&\mbox{ if }\{p,q\}\in\mathcal S_{\alpha'',\alpha},\\
 -\frac{\mathcal E_4[V](\boxtimes)}{2}&\mbox{ if }\{p,q\}\mbox{ missing edge}.
 \end{array}
\right.
\end{equation}
\end{subequations}
We concentrate on edges from $\mathcal{NC}$ only, and note that these edges only satisfy the last three cases in \eqref{defbarw}. In these three cases we use the bound 
\begin{equation}\label{barwbound}
-\frac{\mathcal E_4[V](\boxtimes)}{2}-\frac12\ge-\frac{m_{\alpha',V}}{2}-\frac12,
\end{equation}
following from \eqref{resumq1-3} and assumption (3) that implies the fact that $m_{\alpha',V}<0$. The bound ensuing from \eqref{defbarw} via \eqref{barwbound} for edges in $\mathcal{NC}\cap\mathcal S_\alpha$ is the most negative. Differences in the weight $\overline w$ compared to this value will be called ``lost weight'' and denoted $\mathrm{lw}(\{p,q\})$. In other words, edges in $\mathcal S_{\alpha'',\alpha}$ have lost weight $\frac14$ and missing edges have lost weight $\frac12$. The other edges from $\mathcal{NC}$ will have by definition zero lost weight, and the edges not contained in $\mathcal{NC}$.

\medskip

With this terminology, we bound the total lost weight, amongst points $p$ such that $\mathcal N_\alpha(p)\cap\partial\mathcal G_\alpha\neq\emptyset$ or $\mathcal N_{\alpha''}(p)\setminus\mathcal N_\alpha(p) \neq\emptyset$ (which, due to the previous discusson, include all the vertices participating to $\mathcal{NC}$), as follows:
\begin{multline}\label{lb2}
\sum_{\{p,q\}\in\mathcal{NC}}\mathrm{lw}(\{p,q\}) + \sum_{\{p,\bar q_p\}\mbox{ missing edge}}\mathrm{lw}(\{p,\bar q_p\})\\
\geq \frac{1}{16}\sharp\{p:\mathcal N_\alpha(p)\cap \partial\mathcal{G}_\alpha\neq \emptyset\mbox{ or }\mathcal N_{\alpha''}(p)\setminus\mathcal N_\alpha(p)\neq\emptyset\}.
\end{multline}
Let $w$ be the map assigning to every $p$ the total lost weight summed over edges containing $p$. In particular, $w(p)=0$ if and only if $p\notin\partial\mathcal G_\alpha$, whereas $w(p)\ge \frac12$, whenever $p\in\partial\mathcal G_\alpha$ and $w(p)\ge \frac14$ whenever $\mathcal N_{\alpha''}(p)\setminus\mathcal N_\alpha(p)\neq\emptyset$.
Then, moving, for every $p\in\partial\mathcal{G}_\alpha$, a weight of $1/16$ from $p$ along each edge from $\mathcal G_\alpha$ containing it, will also leave at least a weight $1/16$ at $p$ itself, since there are at most $7$ such edges. We leave the weight at points having ``long'' edges fixed instead. If we do this operation contemporarily for each $p$ we end up with a total weight of at least $1/16$ at each one of the points appearing on the left in \eqref{lb2}. This proves \eqref{lb2}.

\medskip

Now, summing \eqref{boundsEp2} over $p\in\Xi_N$ and using \eqref{defbarw} and \eqref{lb2} together with the previous observation that $\mathcal{NC}\cup\mathcal S_{\alpha'',\alpha}$ has at most $8$ edges per vertex, we have 
\begin{eqnarray}\label{lb3}
\mathcal E[V](X_N)&\ge& N\mathcal E_4[V](\boxtimes) \nonumber\\
&&+ \sharp\{p:\mathcal N_\alpha(p)\cap \partial\mathcal{G}_\alpha\neq \emptyset\mbox{ or }\mathcal N_{\alpha''}(p)\setminus\mathcal N_\alpha(p)\neq\emptyset\}\ \left(-4 m_{\alpha',V} -4+\frac{1}{16}\right)\nonumber\\
&\ge& N\mathcal E_4[V](\boxtimes) =N\min\mathcal E_4[V], 
\end{eqnarray}
where now we used assumption (3). This, together with \eqref{energy_fr2}, proves the first inequality in \eqref{energy_fr}.

\medskip

To show the second inequality in \eqref{energy}, let $\underline t$ be the value which realizes the minimum in the definition of $\overline{\mathcal E}_{\mathrm{sq}}[V]$. It suffices to construct a competitor to the minimization problem solved by $X_N$, in which $\sharp \partial \mathcal{G}_\alpha=O(N^{1/2})$, by considering configurations given by a subset of $\underline t\mathbb Z^2$ of cardinality $N$. To see that such subset exists, simply note that there exists $\widetilde X_N\subset\underline t\mathbb Z^2$ of cardinality $N$ such that
\[
\{\underline t x:\ x\in \mathbb Z^2,\ |x|\le \sqrt N-\sqrt2\}\subset \widetilde X_N\subset \{\underline t x:\ x\in \mathbb Z^2,\ |x|\le \sqrt N+\sqrt2\},
\]
for which all labels $p$ such that $\mathcal N_\alpha(p)\cap \partial \mathcal{G}_\alpha\neq \emptyset$ are assigned to points in a - at most - $5\sqrt2$-neighborhood of $\partial B(0,\sqrt N)$ for $N$ large enough  (note that we are considering here a metric neighborhood in  $\R^2$). Such $\widetilde X_N$ forms a subset of $\underline t\mathbb Z^2$ of cardinality $O(N^{1/2})$. This allows to prove the second inequality in \eqref{energy} and to conclude the proof.
\end{proof}

Actually, we can obtain an easier proof of Theorem \ref{thm_fr_crystal} if we assume that 
\begin{equation}\label{easier}
W(1)=W(2)=-1=\min_{r>0}W(r),
\end{equation}
as shown by the following result.

\begin{proposition}
Let $0<\alpha<\min\{\overline{\alpha},\frac{2-\sqrt 2}{2}\}$, with $\overline\alpha$ given by Lemma \ref{lem_lowerbound}. Set $r_{min}:=1-\alpha$, $C_2:=1$, $r_0:=\sqrt 2+\alpha$ and let $K'$ be given by Lemma \ref{lem:mindistance}.
If $W:(0,\infty)\to\R$ {satisfies} \eqref{easier} and \eqref{hyp_minsepar0},
then  \eqref{energy_fr} holds true with 
\begin{equation}\label{triviafr2}
\overline{\E}_{\mathrm{sq}}[V]=\min\E_4[V]=\E_4[V](\boxtimes)=-4.
\end{equation}
\end{proposition}
\begin{proof}
The proof of \eqref{triviafr2}, as well as the one of the second inequality in \eqref{energy_fr}, is a triviality {following from the fact that $\sqrt{2}<\sqrt{2}+\alpha<2$} and is left to the reader.

We just prove the first inequality in \eqref{energy_fr}. 
Let $N\in\N$ and let $X_N\in\mathcal{X}_N(\R^2)$ be a minimizer of $\E[V]$ in $\mathcal{X}_N(\R^2)$, i.e.,
\begin{equation}\label{zerofin}
\E[V](N)=\E[V](X_N).
\end{equation}
Then,  by \eqref{hyp_minsepar0}  and by Lemma \ref{lem:mindistance}, we have
\begin{equation}\label{unfin}
\E[V](X_N)=\sum_{\{p,q\}\in\mathcal S_\alpha}V(|x_p-x_q|).
\end{equation}
Since $\alpha<\overline{\alpha}$, by Lemma \ref{lem_lowerbound}, we deduce that {$\sharp \mathcal N_\alpha(p)\le 9$}, so that
\begin{equation}\label{deuxfin}
\sharp\mathcal{S}_\alpha=\frac{1}{2}\sum_{a\in\Xi_N}\left(\sharp\mathcal{N}_\alpha(a)-1\right)
\le 4N.
\end{equation}

By \eqref{unfin}, \eqref{deuxfin}, and {\eqref{triviafr2}} we get
$$
\E[V](N)=\E[V](X_N)\ge-\sharp\mathcal{S}_\alpha\ge -4N=N\E_4[V]=N\E_{\mathrm{sq}}[V].
$$
\end{proof}

\section{Long-range potentials and proof of Theorem \ref{thm3}}

 This section is devoted to the proof of the crystallization result in the thermodynamic limit for long-range potentials.

\subsection{Distortion estimates at larger scales}
We  first  organize the information from $\mathcal{G}_\alpha$ in order to be able to compare it to $\mathcal Z_\boxtimes$. 
\begin{definition}[combinatorial embedding]\label{def:discreteemb}
{\rm Let $\mathcal{G}_\alpha'\subset \mathcal{G}_\alpha$ be a subgraph, with vertex labels $\Lambda\subset\Xi$ and edges $\mathcal S_\alpha'\subset \mathcal S_\alpha$. A map $\Phi:\Lambda\to \mathbb Z^2$ gives a \emph{combinatorial embedding} if it is a graph isomorphism between $\mathcal{G}_\alpha'$ and a subgraph of $\mathcal Z_\boxtimes$, i.e. it is injective and $\{p,q\}\in \mathcal S_\alpha'$ if and only if $\{\Phi(p),\Phi(q)\}$ is an edge of $\mathcal Z_\boxtimes$.}
\end{definition}
Note that in \cite{T} ``discrete embeddings'' were defined differently, with a slightly more complicated definition involving directly $X$ and not only graphs. However, the graph-only Definition \ref{def:discreteemb} does still allow to recover all the information relevant to our problem, and in our view makes the structure of the argument clearer.

\medskip

We also need the following classical combinatorial definitions.
\begin{definition}[elementary topology in $\mathcal{G}_\alpha$]\label{def:pathconn}\hfill
{\rm \begin{enumerate}\item A \emph{path in $\mathcal{G}_\alpha$} is an sequence of edges of the form 
\begin{equation}\label{path}
\mathcal P=(\{p_0,p_1\},\{p_1,p_2\},\{p_2,p_3\},\ldots,\{p_{n-1},p_n\})\in (\mathcal S_\alpha )^n, n\in\mathbb N.
\end{equation}
\item Let $\Lambda\subset \Xi$. A path $\mathcal P$ as above is a path \emph{through $\Lambda$} if $p_0,p_1,\ldots,p_n\in\Lambda$. Equivalently $\mathcal P$ is also a path in $\mathcal{G}_\alpha|_\Lambda$.
 \item We say that a subset of vertices $\Lambda\subset\Xi$ is \emph{path-connected} if for each $p\neq q\in\Lambda$ there exists a path $\mathcal P$ through $\Lambda$ such that $p_0=p, p_n=q$. 
 \item An \emph{elementary move in $\mathcal{G}_\alpha$} consists in replacing successive edges $\left\{ \{p,q\},\{q,r\} \right\}\rightarrow\{p,r\}$ or viceversa $\{p,r\}\mapsto\left\{\{p,q\},\{q,r\}\right\}$, provided $\{p,q,r\}$ forms a triangle  (i.e. all pairs of points give an edge) in $\mathcal{G}_\alpha$.
 \item A \emph{discrete homotopy} between two paths $\mathcal P,\mathcal Q$ in $\mathcal{G}_\alpha$ is a sequence of paths connected by elementary moves, which starts at $\mathcal P$ and ends at $\mathcal Q$.
 \item A subset of vertices $\Lambda\subset\Xi$ is \emph{simply connected} if in $\mathcal{G}_\alpha|_\Lambda$ any two paths in $\Lambda$ are connected by a discrete homotopy.
\end{enumerate}
}
\end{definition}
Regions of $\mathcal{G}_\alpha$ which are away from $\partial \mathcal{G}_\alpha$ have a unique combinatorial embedding in the following sense:
\begin{proposition}[combinatorial version of Theorem \ref{thm:crystal_infinite} part (ii)]\label{prop:discremb}
 Let $\Lambda\subset \Xi$ be a path-connected subset such that for each $p\in \Lambda$ there holds $\mathcal N_\alpha(p)\cap\partial\mathcal{G}_\alpha=\emptyset$. Then:
 \begin{enumerate} 
 \item There exists a combinatorial embedding $\Phi$ of $\mathcal{G}_\alpha|_\Lambda$ into $\mathcal Z_\boxtimes$. 
 \item Such $\Phi$ is unique up to composition with a combinatorial embedding of $\mathcal Z_\boxtimes$ into itself.
 \end{enumerate}
\end{proposition}
\begin{proof}
For every $p\in\Lambda$ we set $\phi_p:=\phi$, where $\phi$ is the bijection constructed in Lemma \ref{lem:combintometric}. 
If $\Lambda=\{p\}$, it is enough to consider $\Phi:=\phi_p$ which in view of Remark \ref{remun} is unique up to a composition with a combinatorial embedding of $\mathcal{Z}_{\boxtimes}$ with itself.
Otherwise, we construct $\Phi$ in the following way: Let $p\in\Lambda$, and let $\Phi(q)=\phi_p(q)$ for every $q\in\mathcal{N}_\alpha(p)\cap \Lambda$. For every $q\in\mathcal{N}_\alpha(p)\cap\Lambda$ we set $\Phi(r)=\phi_{q}(r)+\phi_p(q)$ and we proceed so forth. Since $\Lambda$ is path-connected such a procedure stops after a finite number ($\le \sharp\Lambda$) of steps.

 Using Remark \ref{nring} one can easily check that the function $\Phi$ is well-defined, i.e., $\phi_p(r)=\phi_q(r)+\phi_p(q)$ for every $r\in\mathcal{N}_\alpha(p)\cap \mathcal{N}_\alpha(q)\cap\Lambda$. Moreover, $\Phi$ is a combinatorial embedding from $\mathcal{G}_\alpha$ into $\mathcal{Z}_{\boxtimes}$. The uniqueness in (2) is again a consequence of the fact that $\Lambda$ is path-connected.
\end{proof}

The above proposition shows that under path-connectedness and neighborhood closure of $\Lambda$ we actually have an identification of the whole $\Lambda$ with a patch in $\mathcal Z_\boxtimes$. Thus we introduce the following concept:
\begin{definition}[discrete $\mathcal Z_\boxtimes$-charts]\label{def:charts}
{\rm If $\Lambda\subset \Xi$ is such that
 \begin{itemize}
  \item $\mathcal N_\alpha(p)\cap\partial\mathcal{G}_\alpha=\emptyset$ for every $p\in\Lambda$,
  \item there exists a combinatorial embedding $\Phi:\Lambda\to\mathbb Z^2$ of $\mathcal{G}_\alpha|_\Lambda$ whose image is simply connected in $\mathcal Z_\boxtimes$, in the sense of Definition \ref{def:pathconn},
 \end{itemize} then we call $(\Phi,\Lambda)$ a \emph{discrete $\mathcal Z_\boxtimes$-chart} of $\mathcal{G}_\alpha$ (or simply a \emph{discrete chart}). 
 
Moreover, we define a triangle in $\mathcal Z_\boxtimes$ as a triple of distinct vertices $\{z^1,z^2,z^3\}$ of $\mathcal Z_\boxtimes$ such that $\{z^i,z^j\}$ are bonds in $\mathcal Z_\boxtimes$
for every $i,j=1,2,3$ with $i\neq j$ .
 Let $\mathscr{T}'$ be the (non-planar) triangulation of $\Phi(\Lambda)$ induced by $\mathcal {Z}_{\boxtimes}$, i.e. $\mathscr{T}'$ is the set of all the triangles with vertices in $\Phi(\Lambda)$ that are half-unit-squares.
A planar triangulation $\mathscr{T}=\mathscr{T}_{(\Phi,\Lambda)}$ associated to $(\Phi,\Lambda)$  is obtained by removing one arbitrarily chosen diagonal from each (unit) square of $\Phi(\Lambda)\cap\mathcal{Z}_{\boxtimes}$. Furthermore, we set $\mathscr{T}^{-1}:=X^{-1}(\mathscr{T})$.
}
\end{definition}
Now we introduce some notation that will be useful in the results of this section.

\medskip

For every $d\in\N$ and for every $a,b\in\R^d$, we define the ellipsoid with foci $a$ and $b$ and ellipticity $\alpha>0$ as
\begin{equation}\label{includellips}
\mathrm{Ell}_\alpha(a,b):=\left\{x\in\mathbb R^d:\ |x-a|+|y-b|\le\frac{1+\alpha}{1-\alpha}|a-b|\right\}.
\end{equation}
 We denote by $\mathcal{D}$ the set of all (non-trivial) vectors in $\Z^2$, i.e.,
\begin{equation}\label{distz2}
\mathcal{D}:=\{|p|\,:\,p\in\Z^2\setminus\{0\}\}.
\end{equation}
For every $r\in\mathcal{D}$ we denote by $Q'_r$ each square of sidelength $r$ having vertices in $\Z^2$ and we denote by $\mathcal{Q}'_r$ the family of such squares $Q'_r$. Furthermore, we set
\begin{equation}\label{z2squares}
\begin{aligned}
\mathrm{Sides}(Q'_r)&:=\{\{p,q\}\,:\,p,q\in Q'_r,\,|p-q|=r\} \quad\textrm{for every }Q'_r\in\mathcal{Q}'_r,\\
\mathrm{Sides}(\mathcal{Q}'_r)&:=\bigcup_{Q'_r\in\mathcal {Q}'_r}\mathrm{Sides}(Q'_r),\\
\mathrm{Diag}(Q'_r)& :=\{\{p,q\}\,:\,p,q\in Q'_r,\,|p-q|=\sqrt 2r\} \quad\textrm{for every }Q'_r\in\mathcal{Q}'_r,\\
\mathrm{Diag}(\mathcal{Q}'_r)&:=\bigcup_{Q'_r\in\mathcal {Q}'_r}\mathrm{Diag}(Q'_r).
\end{aligned}
\end{equation}
\begin{definition}[squares of scale $r$]\label{distgmr}
{\rm For $r\in\mathcal D$ we say that $Q_r\subset \Xi$ is a \emph{square of scale $r$} if there exists a discrete $\mathcal Z_\boxtimes$-chart $(\Phi, \Lambda)$ such that
\begin{itemize}
\item $\Phi(Q_r)=:Q'_r\in \mathcal{Q}'_r$;
\item $\Z^2\cap\mathrm{Conv}(Q'_r)\subset\Phi(\Lambda)$;
\item $\mathrm{Ell}_\alpha(X(p),X(q))\subset\mathrm{Conv}(X(\Lambda))$ for every $p,q\in Q_r$ with $p\neq q$.
\end{itemize}
We denote by $\mathcal Q_r$ the families of the squares $Q_r$ of sidelength $r$ in $\mathcal{G}_\alpha$ 
and by $\mathcal{Q}$ the union of the families $\mathcal Q_r$, for $r$ varying in $\mathcal D$. 
In analogy with \eqref{z2squares} we also set
\begin{equation}\label{z2squaresdef}
\begin{aligned}
\mathrm{Sides}(Q_r)&:=\{\{\Phi^{-1}(a),\Phi^{-1}(b)\}\,:\,\{a,b\}\in  \mathrm{Sides}( \Phi(Q_r))\}\phantom{\quad\textrm{for every }Q_r\in\mathcal{Q}_r,}\\
& \phantom{:} =\{\{p,q\}\,:\,\{\Phi(p),\Phi(q)\}\in  \mathrm{Sides}(\Phi(Q_r))\} \quad\textrm{for every }Q_r\in\mathcal{Q}_r,\\
\mathrm{Sides}(\mathcal{Q}_r)&:=\bigcup_{Q_r\in\mathcal {Q}_r}\mathrm{Sides}(Q_r),\\
\mathrm{Diag}(Q_r)&:=\{\{\Phi^{-1}(a),\Phi^{-1}(b)\}\,:\,\{a,b\}\in \mathrm{Diag}(\Phi(Q_r))\} 
\phantom{\quad\textrm{for every }Q_r\in\mathcal{Q}_r,}\\
&\phantom{:} =\{\{p,q\}\,:\,\{\Phi(p),\Phi(q)\}\in  \mathrm{Diag}(\Phi(Q_r))\} \quad\textrm{for every }Q_r\in\mathcal{Q}_r\\
\mathrm{Diag}(\mathcal{Q}_r)& :=\bigcup_{Q_r\in\mathcal {Q}_r}\mathrm{Diag}(Q_r).
\end{aligned}
\end{equation}
In the following, the {\it $r$-neighborhood} of a square at scale $r$ is the set of all points in $\mathcal {G}_\alpha$ which can be connected to a point in $Q_r$ through a combinatorial path of length $\le r$.
 }
\end{definition}
Our next goal is to prove that $X$-images of squares $Q_r\in\mathcal Q_r$ are actually $L\alpha$-deformations of metric squares $Q_r\subset (\mathbb Z^2,\ell_2)$, with $L$ independent of $\alpha$.
The case $r=1$ already follows from Lemma \ref{lem:combintometric}. Next, we consider the usual isometric embedding $\Z^2\subset\R^2$, seen here as a \emph{labeling of a configuration}, with label set $\mathbb Z^2$:
\begin{equation}\label{labelz}
 \iota_{\mathbb Z^2}:\mathbb Z^2\to\mathbb R^2 ,\qquad a\mapsto \iota(a)=a_*.
\end{equation}
\begin{proposition}[discrete charts become metric charts]\label{prop:discrmetricchart}
There exists a constant $L\ge 1$ with the following properties. Let $\alpha\in [0,\alpha_0)$, with $\alpha_0$ given by Lemma \ref{lem:combintometric}. Let $(\Phi,\Lambda)$ be a discrete $\mathcal Z_\boxtimes$-chart of $\mathcal{G}_\alpha$ and let $\mathscr T$ be a planar triangulation associated to $(\Phi,\Lambda)$.  Assume that
\[
 \overline{\Phi(\Lambda)}:=\bigcup\left\{\overline{\mathrm{Conv}(\{a^1_*,a^2_*,a^3_*\})}:\ \{a^1,a^2,a^3\}\mbox{ triangle in }\mathscr{T}\right\}.
\]
Then, there exists a {Lipschitz continuous map $u:\overline{\Phi(\Lambda)}\to \mathbb R^2$} which satisfies the following  conditions:
\begin{enumerate}
 \item $u(\Phi(p))=X(p)$ for all $p\in\Lambda$;
 \item $u$ is piecewise affine on $\mathrm{Conv}(\{a^1_*,a^2_*,a^3_*\})$ for every  triangle $\{a^1,a^2,a^3\}\in\mathscr{T}$;
 \item $u$ satisfies
 \begin{equation}\label{disto2}\sup_{x\in\overline{\Phi(\Lambda)}}\mathrm{dist}(Du(x),SO(2))< L\alpha.
 \end{equation}
\end{enumerate}
\end{proposition}

Note that in \cite{T} the map $u$ was going in the opposite direction than our map, but since $u$ is bijective and $Du$ is invertible, this actually makes not much difference.
\begin{figure}[h]
\includegraphics[width=10cm]{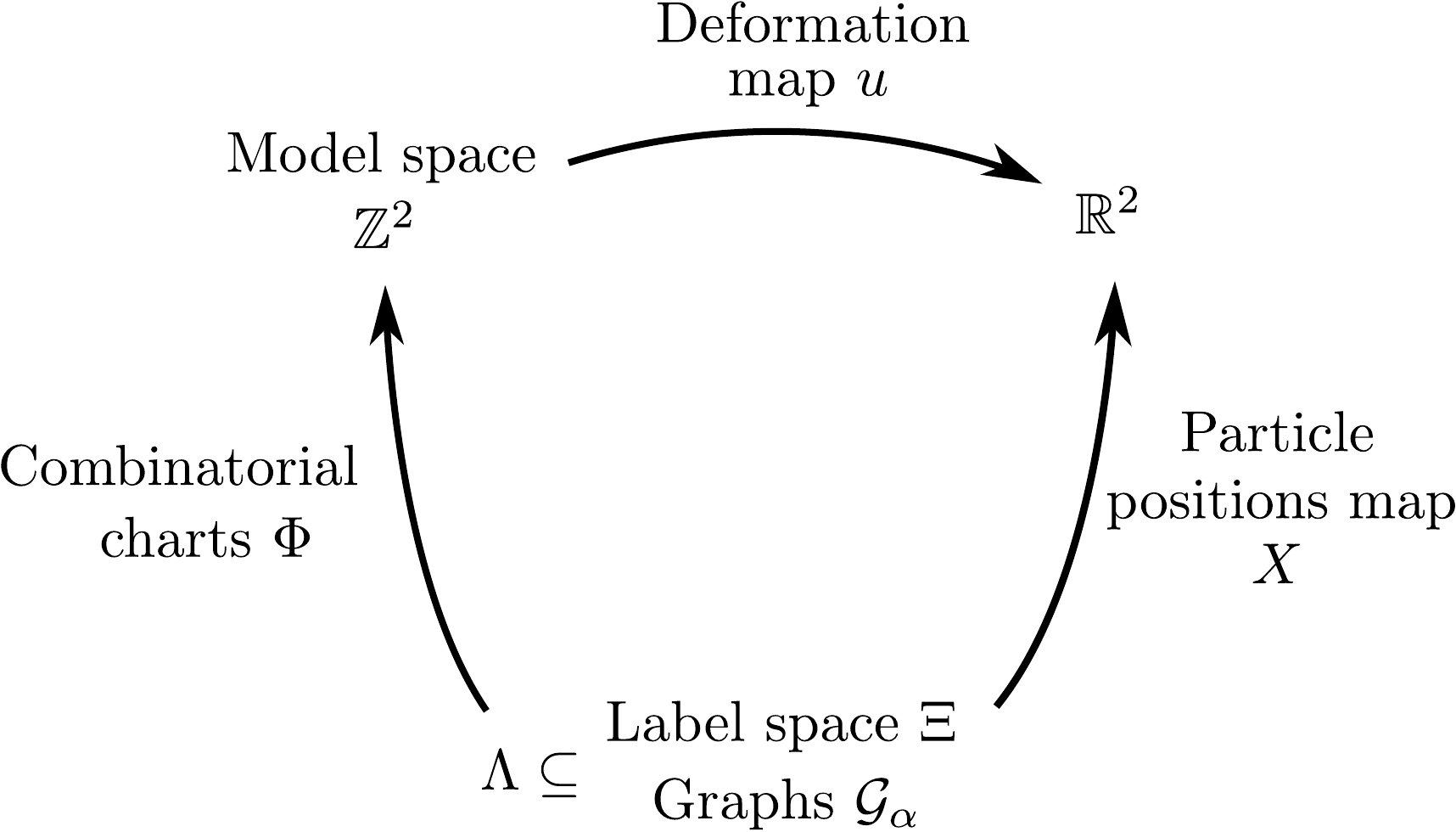}
\caption{Summary of the different maps and spaces used in this section.}\label{fig-summary}
\end{figure}

\begin{proof}
For every $p\in\Lambda$ we set $u(\Phi(p)):=X(p)$ and we extend $u$ affinely over each triangle in $\mathscr{T}$.
\medskip
Let $T^{-1}:=\{ p^1,p^2,p^3\}\in\mathscr{T}^{-1}$ and set $a^i:=\Phi(p^i)$ for $i=1,2,3$; set moreover $T:=\{a^1_*,a^2_*,a^3_*\}$. By construction, $T\in\mathscr{T}$.
By definition of $\mathcal{G}_\alpha$, the image $\{X(p^1),X(p^2),X(p^3)\}$ of $\{p^1,p^2,p^3\}$ through $X$ equals, up to a rotation, to a small deformation of $\{z^1,z^2,z^3\}$ with $z^1=(0,0), z^2=(0,1), z^3=(1,0)$, in the sense that up to reassigning the labels $p^1,p^2,p^3$, we have
\[
 \frac{|X(p^1)-X(p^2)|}{| z^1-z^2|},\ \frac{|X(p^1)-X(p^3)|}{| z^1-z^3|},\ \frac{|X(p^2)-X(p^3)|}{|z^2-z^3|}\in E^1_{\alpha}.
\]
Then due to Proposition \ref{prop:discremb} and to Definition \ref{def:charts} of discrete $\mathcal Z_{\boxtimes}$-chart, $\Phi$ sends $T^{-1}$ to a congruent copy of $\{z^1,z^2,z^3\}$. 
Therefore, we can define $u$ over $\mathrm{Conv}( T)$ as the affine map with gradient $Du(x)=:F^T$, where
\[
F^T(X(p^2)-X(p^1))=(0,1),\quad\textrm{and}\quad F^T(X(p^3)-X(p^1))=(1,0).
\]
By using the cosine rule (with details left to the reader), it follows that for every $T\in\mathscr{T}$ there exist two vectors $v_1^T,v_2^T\in B_1$ such that
\[
 F^T\in \left(e_1+\alpha v^T_1, e_2+\alpha v^T_2\right)O(2),
\]
thus for a value of $L$ independent of $\alpha$ (and of $T$) it holds 
\begin{equation}\label{almostright}
\mathrm{dist}(F^T, O(2))\le L\alpha\qquad\textrm{for all }T\in\mathscr{T}.
\end{equation}

By the first bullet in the  Definition \ref{def:charts}, we can apply Lemma \ref{lem:combintometric} to all $p\in\Lambda$, and we find that maps $F$ corresponding to neighboring triangles either all preserve orientation or all reverse orientation. By the connectedness of $\Phi(\Lambda)$, which follows from the second point in Definition \ref{def:charts}, we find inductively that this is also true for the collection of maps $F$ corresponding to all triangles in $\mathscr{T}$. Thus $Du(x)$ stays $L\alpha$-close either to $SO(2)$ or to $\{M\in O(n):\ \det(M)=-1\}$. In the latter case, we may compose the discrete chart $\Phi$ with the self-embedding of $\mathcal Z_\boxtimes$ given by the map $\Phi_-(a,b):=(-a,b)$, which has the effect of making all $F$ orientation-preserving. Thus we have from \eqref{almostright} that
\[
L\alpha\ge\mathrm{dist}(F^T,O(2))=\mathrm{dist}(F^T,SO(2))\qquad\textrm{for all }T\in\mathscr{T};
\]
whence \eqref{disto2} follows. This completes the proof.
\end{proof}
We will use different distortion bounds for treating linear and quadratic deformations. The first one is \cite[Lemma III]{john} which was slightly extended by \cite[Proposition 4.1]{T}, and gives the following result.

\begin{lemma}[John distortion {\cite{john}}]\label{johnlemma}
 For every $d\in\N$ there exists $\alpha_1=\alpha_1(d)>0$ such that for each $\alpha\in[0,\alpha_1)$ the following holds. 
Given $a,b\in\R^d$, if {$u:\mathrm{Ell}_\alpha(a,b)\to \mathbb R^d$ is Lipschitz continuous and} satisfies
 \begin{equation}\label{distdu}
  \sup_{x\in\mathrm{Ell}_\alpha(a,b)}\mathrm{dist}(Du(x),SO(d))<\alpha,
 \end{equation}
then, with notation of Definition \ref{def:deform}, we have
\begin{equation}\label{bound}
\delta_u(a,b)\le \alpha.
\end{equation}
\end{lemma}
The next result allows to obtain good enough bounds for quadratic distortions. Although it holds in general dimension, we prove it only in the special $2$-dimensional case, because this is the version that we require in the remainder of the paper.
\begin{proposition}[Quadratic distortion estimate]\label{prop:distortion}
There exists a constant $C_6>0$ depending only on the dimension such that if $\alpha\in(0,\alpha_1)$, with $\alpha_1$ as in  Lemma \ref{johnlemma}, the following holds true. Let $(\Phi,\Lambda)$ is a discrete $\mathcal Z_{\boxtimes}$-chart of $\mathcal G_\alpha$ and let $u:\overline{\Phi(\Lambda)}\to \R^2$ be the map constructed in Propostion  \ref{prop:discrmetricchart}; then for every $r\in\mathcal D$ and for every $\{a,b\}\in \mathrm{Sides}(\mathcal Q_r) \cup\mathrm{Sides}(\mathcal Q_{\sqrt 2 r})$ it holds
 \begin{equation}\label{fjm}
 \delta^2_{u\circ \Phi}(a,b):=  \delta^2_{u}(\Phi(a),\Phi(b))\le C_6\ r\ \sum_{\substack{\{p,q\}\in\mathcal S_\alpha\\ \mathrm{dist}(X(\{p,q\}),[X(a),X(b)])<4}}\delta^2_{u\circ\Phi}(p,q).
 \end{equation}
\end{proposition}
\begin{proof}
Throughout the proof all our constants are either explicit or they depend only on $L, \alpha_1$ above, which in turn depend only on the dimension.

\medskip

Let $\mathscr{T}$ be a triangulation associated to $(\Phi,\Lambda)$ according to the Definition \ref{def:charts}.

\medskip
 
By Definition \ref{def:charts}, and more precisely by Definition \ref{def:discreteemb}, we have that $\{p,q\}\in\mathcal S_\alpha$ for every $p,q\in\Lambda$ with $|\Phi(p)-\Phi(q)|\in\{1,\sqrt 2\}$.
Moreover, again by Definition \ref{def:charts}, the union of the triangles $T\in\mathscr{T}$ contains a neighborhood of the segment $[X(a),X(b)]=[u(\Phi(a)),u(\Phi(b))]$. Thus, all triangles $\tau=\{p_1,p_2,p_3\}\in\mathscr{T}^{-1}$ with $\mathrm{Conv}(X(\tau))\cap [X(a),X(b)]\neq\emptyset$ satisfy $\mathrm{dist}(X(p_i),[X(a),X(b)])<2$ because $2>\sqrt{2}+\alpha$ for $\alpha$ sufficiently small.
Next note that, since $u$ is affine, the maximum 
\[
\max_{x,y\in \mathrm{Conv}(X(\tau))}\delta_u(x,y)
\]
is achieved at the vertices of the simplex $\mathrm{Conv}(X(\tau))$, i.e. for $x=X(p_i),y=X(p_j)$, for some $i,j\in\{1,2,3\}$ with $i\neq j$, i.e., 
\begin{equation}\label{bounda}
\max_{x,y\in\mathrm{Conv}(X(\tau))}\delta_u(x,y)\le \max_i\delta_{u\circ \Phi}(p_i,p_{i+1})\le \sum_{i=1}^3\delta_{u\circ \Phi}(p_i,p_{i+1}),
\end{equation}
in which the indices $i$ are intended modulo $3$. 
Let now $\mathscr{T}^{-1}_{a,b}$ denote the set of the triangles $\tau\in\mathscr{T}^{-1}$ with $\mathrm{Conv}(X(\tau))\cap [X(a),X(b)]\neq\emptyset$ and let $[\xi^\tau,\eta^\tau]=\mathrm{Conv}(X(\tau))\cap [X(a),X(b)]$ where $\xi^\tau$ and $\eta^\tau$ are not necessarily distinct. Then $\mathscr{T}^{-1}_{a,b}=\{\tau^m\}_{m\in M}$ for some $M\subset\N$, where the indices $m$ are chosen in such a way that the subsegments $\sigma^m:=[\xi^{\tau^m},\eta^{\tau^m}]$ are concatenated.

\medskip

By \eqref{bounda} and by triangular inequality, it follows that
\begin{eqnarray}\label{bound1}
 |X(a)-X(b)|&=&\sum_{m=1}^M|u(\sigma^m)|\ge \sum_{m=1}^M\left(|\sigma^m|-\delta_u(\sigma^m)\right)\ge |\Phi(a)-\Phi(b)| -\sum_{m=1}^M\delta_u(\sigma^m)\nonumber\\
 &\ge&|\Phi(a)-\Phi(b)| -2\sum_{\substack{\{p,q\}\in\mathcal S_\alpha\\ \mathrm{dist}(X(\{p,q\}),[X(a),X(b)])<4}}\delta_{u\circ\Phi}(p,q),
\end{eqnarray}
where in the last inequality we have used that each edge used can appear at most $2$ times, i.e., at most one for each triangle of which it is an edge. 
Proceeding symmetrically and considering segments $\widetilde\sigma^k,\, k=1,\ldots,K$, which are intersections of $[\Phi(a),\Phi(b)]$ with successive triangles of $\mathscr{T}$, and arguing as in \eqref{bound1} we get
\begin{eqnarray}\label{bound2}
|\Phi(a)-\Phi(b)|&=&\sum_{k=1}^{K}|\widetilde\sigma^k|\ge \sum_{k=1}^{K}\left(|u(\widetilde\sigma^k)|-\delta_u(\widetilde\sigma^k)\right)\ge |X(a)-X(b)| -\sum_{k=1}^{K}\delta_u(\widetilde\sigma^k)\nonumber\\
 &\ge&|X(a)-X(b)| -3\sum_{\substack{\{p,q\}\in\mathcal S_\alpha\\ \mathrm{dist}(X(\{p,q\}),[X(a),X(b)])<4}}\delta_{u\circ\Phi}(p,q).
\end{eqnarray}
Combining \eqref{bound1} and \eqref{bound2} we find
\begin{equation}\label{bound3}
\begin{aligned}
\delta_{u\circ\Phi}(a,b)&=\left|\left|X(a)-X(b)\right|-\left|\Phi(a)-\Phi(b)\right|\right|\\
&\le 3\sum_{\substack{\{p,q\}\in\mathcal S_\alpha\\ \mathrm{dist}(X(\{p,q\}),[X(a),X(b)])<4}}\delta_{u\circ\Phi}(p,q).
\end{aligned}
\end{equation}
Using that $\min_{\{p,q\}\in\mathcal{S}_{\alpha}}|X(p)-X(q)|> 1-\alpha>1-\alpha_1$ and a packing bound, we find that the number of terms of the sum in \eqref{bound3} is at most $Cr$, with $C$ a constant depending only on $L, \alpha_1$, and thus only on the dimension. Thus, we may apply the Cauchy-Schwarz inequality to \eqref{bound3} in order to obtain \eqref{fjm}.

\end{proof}
\begin{rmk}{\rm The content of Proposition \ref{prop:distortion} is similar in spirit to \cite[Proposition 4.3]{T} but it presents some differences. 
On the one hand, the result in \cite{T} yields the scaling $\mathrm{log}(r)$ which is better than the scaling  $r$ we achieve. On the other hand, the proof of the upper bound in \cite{T} relies on the rigidity estimate by Friesecke-James-M\"uller \cite{fjm}. Here we preferred to include the ``worse'' upper bound in \eqref{fjm} - which however does not affect our final result - and to provide the more transparent proof above. }
\end{rmk}
The last result of this subsection deals with the  partitioning of the edge set of the complete graph generated by $\Z^2$ into edges coming from sides and diagonals of squares of sidelength $r$. We first need the following definition.
\begin{definition}[Rescaled copies of $\mathcal Z_\boxtimes$]\label{selfsim}
{\rm Let $\mathcal D$ be defined as in \eqref{distz2}. For every $r\in\mathcal D$ we define 
\begin{subequations}\label{distz2mr}
 \begin{equation}\label{calLr}
\mathcal{L}_r:=\{\mbox{sublattice $\Lambda\subset\mathbb Z^2$ such that }\Lambda\simeq r \mathbb Z^2\},\quad\mathcal{L}=\bigcup_{r\in\mathcal D}\mathcal L_r,
\end{equation}
\begin{equation}\label{mr}
\quad m(r):=\sharp\mathcal{L}_r=\frac{1}{4}\sharp \left\{x\in\Z^2:\ |x|=r\right\}.
\end{equation}
Let $\mathcal K_{\mathbb Z^2}$ be the complete graph associated to $\Z^2$, i.e.
\[
\mathcal K_{\mathbb Z^2}:=(\mathbb Z^2, \{\{a,b\}:\ a\neq b\in\mathbb Z^2\}).
\]
\end{subequations}}
\end{definition}
We note the following well-known number-theoretical result:
\begin{lemma}\label{lem:mrm2r}
 With the above definitions, for all $r\in\mathcal D$ we have $m(r)=m(\sqrt2 r)$.
\end{lemma}
\begin{proof}
 This amounts to prove that for each integer $n>1$ the number of ways to write $n$ as a sum of two squares equals the number of ways to write $2n$ as the sum of two squares. 
 
 \medskip 
 
 By a theorem of Euler, $n$ is the sum of two squares if and only if all its prime factors equal to $3$ modulo $4$ occur to an even power. If this condition is met, then the ways of writing $n$ as two squares are given (see \cite{hw}, Thm. 278) by 
 \[
  4\prod_{j=1}^s(b_j+1),\quad \mbox{ if }\quad n=2^{a_0}\prod_{i=1}^r p_i^{2a_i}\prod_{j=1}^s q_j^{b_j},
 \]
where $a_i,b_j$ are integers, $q_j$ are distinct primes all equal to $1$ modulo $4$ and $p_i$ are distinct primes equal to $3$ modulo $4$. In particular, the number of ways to write $n$ as a sum of two squares does not depend on $a_0$, as desired.
\end{proof}

Now we can give {the} splitting result announced earlier, which will be useful for organizing the values of $V$ taken on our configurations.
\begin{lemma}[Covering of $\mathcal K_{\mathbb Z^2}$ by lattices and squares]\label{lem_splitgraph}
There exists $\widetilde{\mathcal D}\subset\mathcal D$ such that
\begin{equation}\label{splitgraph}
\{\{0,a\}\,:\, a\in\Z^2\setminus\{0\}\}=\bigsqcup_{r\in\widetilde{\mathcal D}}\bigsqcup_{\Lambda\in\mathcal L_r}\left\{\{0,x\}: x\in\Lambda,\ |x|\in\{r,\sqrt2 r\}\right\},
\end{equation}
where the symbol $\bigsqcup$ denotes the pairwise disjoint union. Moreover, we necessarily have $1\in\widetilde{\mathcal D}$ and
\begin{equation}\label{rbigg1}r\in\widetilde{\mathcal D}\setminus\{1\}\quad\Rightarrow\quad r\ge 2,
\end{equation}
whereas for edge multiplicities we have 
\begin{equation}\label{splitgraph2}
 \sum_{\{a,b\}:\ a\neq b\in\Z^2}\delta_{\{a,b\}}=\frac12 \sum_{r\in\widetilde{\mathcal D}}\sum_{Q\in \mathcal Q_r'\cup\mathcal Q_{\sqrt2 r}'}\sum_{\{a,b\}\in\mathrm{Sides}(Q)}\delta_{\{a,b\}}.
\end{equation}

\end{lemma}
\begin{proof}
We already know that each $\Lambda\in\mathcal{L}$ is isomorphic to $\mathbb Z^2$. Let 
\[
Q_{++}:=\{(a,b)\in\Z^2:\ a>0, b\ge 0\}.
\]
Let $B:\mathcal L\to Q_{++},\,\Lambda\mapsto B(\Lambda):=v_\Lambda$, where $|v_\Lambda|=\min\{|v|\,:\,v\in \Lambda\cap Q_{++}\}$. By the very definition of $Q_{++}$, $B$ is well-defined and bijective, so that we can write $\Lambda[v]:=B^{-1}(v)$ for $v\in Q_{++}$.
\medskip 

Define then a map $F:Q_{++}\to Q_{++}$ by taking $F(v)=v'$ to be the shortest vector in $\Lambda[v]\cap(Q_{++}\setminus\{0,v\})$. Note that  $v'$ is one of the $4$ vectors in $\Lambda[v]$ that have length $\sqrt2 |v|$, and furthermore it forms an angle of $45^\circ$ with $v$. We notice that
\[
v'=F(v)\quad\Leftrightarrow\quad\Lambda[v']=\left(\begin{array}{cc}
1&1\\-1&1\end{array}\right)\Lambda[v],
\]
which shows also that $F$ is injective, as $B$ is bijective and the above relation is one-to-one. 

\medskip 

Now note that $Q_{++}$ can be partitioned into maximal orbits of $F$, i.e., there exists a set $\mathcal{V}\subset \Z^2$ made of distinct vectors such that
$$
Q_{++}:=\bigcup_{v\in\mathcal V}\bigcup_{n\in\N} F^n(v),
$$
where the vectors $v\in\mathcal{V}$ are such that  $F(q)\neq v$ for every $q\in Q_{++}$. Indeed, any point in $Q_{++}$ belongs to an orbit $\{F^n(v)\,:\,n\in\N\}$ for some $v\in\mathcal V$ and if two orbits meet at $v$ then they coincide on all the ``positive direction'' $\{F^n(v):\ n\in\mathbb N\}$, because $F$ is well defined, and in the ``negative direction'', because $F$ is injective.

\medskip

Next, we can then split each maximal orbit 
\[
\{v,F(v),F^2(v),\ldots\}={\bigsqcup_{n\ge 0}\{F^{2n}(v),F^{2n+1}(v)}\},
\]
where we have set $F^0(v):=v$. Therefore, there exists $\tilde V\subset Q_{++}$ such that

\begin{equation}\label{partq++}
 Q_{++}=\bigsqcup_{v\in \widetilde V}\{v,F(v)\},\quad \Z^2\setminus\{0\}=\bigsqcup_{v\in \widetilde V}\left\{x\in \Lambda[v]:\ |x|\in\{|v|,\sqrt2 |v|\}\right\},
\end{equation}
where the second equality follows by covering $\Z^2\setminus\{0\}$ by four rotations of $Q_{++}$ by $\pi/2$ and observing that (the restriction of) this operation on each $\Lambda[v]$ translates. We now define 
\begin{equation}\label{defdtilde}
\widetilde{\mathcal D}:=\{|v|:\ v\in \widetilde V\},\quad \mbox{so that  }\sqrt2 \widetilde{\mathcal D}=\{|F(v)|:\ v\in\widetilde V\},
\end{equation}
where $\sqrt2 \widetilde{\mathcal D}:=\{\sqrt2r:\ r\in\widetilde{\mathcal D}\}$. Directly from \eqref{defdtilde} we have
\begin{equation}\label{inclusionsdv}
\bigsqcup_{r\in\widetilde{\mathcal D}}\mathcal L_r \supseteq \bigcup_{v\in\widetilde V}\Lambda[v]\quad\mbox{ and }\quad \bigsqcup_{r\in\sqrt2\widetilde{\mathcal D}}\mathcal L_r \supseteq \bigcup_{v\in\widetilde V}\Lambda[F(v)].
\end{equation}
Due to the bijectivity of the map $Q_{++}\ni v\mapsto \Lambda[v]\in \mathcal L$ and to \eqref{partq++}, we have 
\begin{equation}\label{splitv}\mathcal L=\{\Lambda[v]:\ v\in\widetilde V\}\sqcup\{\Lambda[F(v)]:\ v\in\widetilde V\}.
\end{equation}
Now consider the map $\Lambda[v]\mapsto \Lambda[F(v)]$, which as we saw is well defined over $\mathcal L$ and injective. By Lemma \ref{lem:mrm2r} we know $m(r)=m(\sqrt2 r)$, and in particular for each $r\in\widetilde{\mathcal D}$ the above map restricts to a bijection $\mathcal L_r\to \mathcal L_{\sqrt2 r}$, thus giving
\begin{equation}\label{splitd}
\mathcal D=\widetilde{\mathcal D}\bigsqcup(\sqrt2 \widetilde{\mathcal D}),
\end{equation}
which in combination with \eqref{splitv} implies that the inclusions \eqref{inclusionsdv} are equalities. Therefore we can rewrite the second formula from \eqref{partq++} by taking the union over $\mathcal L_r, r\in\widetilde{\mathcal D}$ instead of $\widetilde V$, and we get \eqref{splitgraph}.

\medskip

Now note that $r=1$ necessarily belongs to $\widetilde{\mathcal D}$, because $\Lambda=\mathbb Z^2$ is the only element of $\mathcal L$ that contains the edge $\{(0,0),(0,1)\}$. Then all edges of $\Z^2$ of length $\sqrt2$ from $\Z^2$ are covered by the choice $\Lambda=\Z^2$ in \eqref{splitgraph}, and thus because the union in \eqref{splitgraph} must be disjoint, the next $r\in\widetilde{\mathcal D}$ must be $r\ge 2$, as claimed in \eqref{rbigg1}.

\medskip

For \eqref{splitgraph2}, note that each edge $\{a,b\}$ from the left hand side of \eqref{splitgraph2} has length $r$ for some $r\in \mathcal D$. Due to \eqref{splitd} two mutually excluding cases can happen:
\begin{itemize} 
\item $r\in\widetilde{\mathcal D}$, in which case $\{a,b\}$ is the side of precisely two squares congruent to $\{0,r\}^2$, and is counted exactly twice in the sum from the right hand side of \eqref{splitgraph2}.
 \item $r\in\sqrt 2\widetilde{\mathcal D}$, in which case $\{a,b\}$ is the side of precisely two squares congruent to $\{0,\sqrt2 r'\}^2$ and $ r'=r/\sqrt2\in\widetilde{\mathcal D}$. Thus again $\{a,b\}$ is counted exactly twice in the sum from the right hand side of \eqref{splitgraph2}.
\end{itemize}
In both cases multiplicities on the two sides of \eqref{splitgraph2} coincide, and the equation is proved.
\end{proof}

We finally notice that, due to Proposition \ref{prop:discremb}, we are able to pass the combinatorial structure \eqref{distz2mr} from $\mathbb Z^2$ to $\mathcal{G}_\alpha$ in a robust way in the presence of sufficiently extended charts as in Definition \ref{def:charts}. Lemma \ref{johnlemma} allows to add to this a metric structure.
In particular, as a direct corollary of Lemma \ref{johnlemma} and Proposition \ref{prop:discrmetricchart} we then obtain the following result.
\begin{lemma}\label{lem_squaredeform}
With the constants $\alpha_1, L>0$ as in Lemma \ref{johnlemma} and Proposition \ref{prop:discrmetricchart} for all $ \alpha\in[0,\alpha_1)$ and whenever $Q_r$ is a square of scale $r$ in $\mathcal{G}_\alpha$ and $r\in\mathcal D$, it holds
\begin{equation}
X(Q_r)\sim_{L\alpha}\{0,r\}^2\subset \mathbb R^2. 
\end{equation}
\end{lemma}

\subsection{Boundary error bounds}\label{bdryerrorbounds}

We are going to write the energy of our configuration as the sum of contributions coming from sides (i.e. edges different than diagonals) of squares $Q_r$ with $r\in\mathcal D$ and the remainder:
\begin{equation}\label{energysplit}
\begin{split}
\mathcal E[V](X_N)=&\frac12\sum_{r\in\mathcal D}\sum_{Q_r\in\mathcal Q_r}\sum_{\{a,b\}\in \mathrm{Sides}(Q_r)}V(|x_a-x_b|)\\ &+ \sum_{\{a,b\}\in\mathcal{NQ}^{(2)}}V(|x_a-x_b|)+\frac12\sum_{\{a,b\}\in\mathcal{NQ}^{(1)}}V(|x_a-x_b|),
\end{split}
\end{equation}
where $\mathcal{NQ}^{(j)}$ are the edges that are sides of $2-j$ squares, for $j=1,2$.

\medskip

The first result of this subsection is Proposition \ref{prop:card_bounds} which allows to compare, for every $r\in\mathcal D$, the cardinality and the area of the squares at scales $r$ and $1$.
We preliminarily introduce some notations.

\medskip

If $P\subset \mathbb R^2$ is a polygon, i.e. a finite set of points $P=\{p_j:\ j\in\mathbb Z/n\mathbb Z\}$ ordered in cyclical order such that the associated polygonal line $\gamma(P):=\cup_j [p_j,p_{j+1}]$ does not self-intersect), then we denote as usual by 
\[
 \mathrm{Area}(P):=\left|\left\{x:\ x \mbox{ belongs to the bounded connected component of $\mathbb R^2\setminus\gamma(P)$}\right\}\right|.
\]
If $Q\subset \R^2$ is a small deformation of a square then in order to define $\mathrm{Area}(Q)$, unless otherwise specified, we always consider it with the cyclic order along the perimeter of the square.

For every $r\in\mathcal D$ and for every $Q_1\in\mathcal{Q}_1$ we set 
\begin{equation}\label{mub}
\mathcal{Q}^{b}_{r}(Q_1):= \left\{Q_r\in \mathcal Q_r:\ \bigcup_{\{p,q\}\subset Q_r}\{z\in \mathrm{Conv}(X(Q_1)):\ \mathrm{dist}(z, [X(p),X(q)])<4\}\neq \emptyset, \,\,\, Q_r\supset Q_1\right\}.
\end{equation}
Finally, the symbol $\Delta$ denotes the symmetric difference between sets $A\Delta B:=(A\setminus B)\cup(B\setminus A)$.

\begin{proposition}\label{prop:card_bounds}
There exists $C_7>0$ such that, for all $\alpha\in (0,\min\{\alpha_0,\alpha_1\})$, where $\alpha_0$ is as in Lemma \ref{lem:combintometric} and $\alpha_1$ is as in Lemma \ref{johnlemma}, the following holds. If $X$ satisfies \eqref{mindistnew}, then for all $r\in \mathcal{D}$ we have:
\begin{subequations}\label{eq:cardbounds}
\begin{eqnarray}
0&\le&m(r)\sharp \mathcal Q_1 - \sharp \mathcal Q_r\le C_7r^2m(r)\sharp \partial \mathcal{G}_\alpha;\label{card_r_bd}\\
0&\le&r^2m(r)\sum_{Q_1\in \mathcal Q_1}\mathrm{Area}(X(Q_1))-\sum_{Q_r\in\mathcal Q_r}\mathrm{Area}(X(Q_r))\le C_7r^4m(r)\sharp \partial \mathcal{G}_\alpha;\label{area_r_bd}
\end{eqnarray}
\begin{equation}
\sharp \{\{p,q\}\in\mathrm{Sides}(\mathcal Q_1):\ [X(p),X(q)]\cap [X(a),X(b)]\neq \emptyset\}\le C_7 r\quad\forall\{a,b\}\in\mathrm{Sides}(\mathcal Q_r);\label{side_r_bd}
\end{equation}
\begin{equation}
\sharp\mathcal{Q}^{b}_{r}(Q_1)\le C_7 r\ m(r)\qquad\textrm{for all }Q_1\in \mathcal Q_1;\label{length_r_bd}
\end{equation}
\begin{equation}
\sharp({\mathrm{Sides}(\mathcal Q_{\sqrt2 r})\Delta \mathrm{Diag}(\mathcal Q_r)})\le C_7r^2\sharp\partial\mathcal{G}_\alpha\qquad \textrm{for all }r\in\widetilde{\mathcal D}.\label{badsides_r_bd}
\end{equation}
\end{subequations}
\end{proposition}
\begin{proof}
Note that below by abuse of notation we denote by $C$ a constant that can possibly change at each step. For \eqref{card_r_bd} we proceed as in \cite[Proposition 2.9]{T}, and start with a double count of 
\[
 \{(x,Q_r)\,:\, Q_r\in\mathcal Q_r,\ x\in X(Q_r)\}.
\]
On the one hand, setting $s(x,r):=\sharp\{Q_r\in\mathcal Q_r:\ x\in X(Q_r)\}$ for every $r\in\mathcal D$ and for every $x\in X(\Xi)=:X$, we get that
\begin{equation}\label{cardrbd1}
 \sum_{x\in X}s(x,r)=4\sharp \mathcal Q_r.
\end{equation}
On the other hand, we will check that 
\begin{equation}\label{srmrs1}
s(x,r)\le m(r)s(x,1)\qquad\textrm{for every }x\in X,\, r\in\mathcal D,
\end{equation}
which together with \eqref{cardrbd1}, and summing over $x\in X$, implies the first inequality in \eqref{card_r_bd}.
To prove \eqref{srmrs1}, we preliminarily note that if
$s(x,r)\neq 0$ then for $p:=X^{-1}(x)$ there holds $\mathcal N_\alpha(p)\cap \partial \mathcal{G}_\alpha=\emptyset$, by definition of $\mathcal Q_r$, thus $s(x,1)=4$. 
Moreover, if $r>1$ only two facts can happen: either $s(x,r)=4m(r)=m(r)s(x,1)$, corresponding to the case that there are four squares at scale $r$ having $x$ as a vertex; or, $s(x,r)<4m(r)=m(r)s(x,1)$, corresponding to the case that  not all the four squares at scale $r$ having $x$ as a vertex are in the discrete $\mathcal{Z}_{\boxtimes}$-chart. This argument implies \eqref{srmrs1}.

In order to get the second inequality in \eqref{card_r_bd}, we notice that if  $s(x,r)<4m(r)=m(r)s(x,1)$, then there exists a point $y=X(q)$ with $q\in\partial \mathcal{G}_\alpha$ at distance at most $Cr$ from $x$. Due to Lemma \ref{lem:mindistance} and by a packing bound, this can happen for at most $Cr^2\sharp \partial \mathcal{G}_\alpha$ points $x$. By \eqref{cardrbd1}, summing over $X$, we get
$$
4\sharp\mathcal{Q}_1-4\frac{\sharp\mathcal{Q}_r}{m(r)}=\sum_{x\in X} \left(s(x,1)-\frac{s(x,r)}{m(r)}\right)\le C r^2\sharp\partial\mathcal{G_\alpha},
$$
namely the second inequality in \eqref{card_r_bd}.

\medskip
To prove \eqref{area_r_bd}, we first set
\[
 \mu(Q_1,Q_r):=\left|\mathrm{Conv}(X(Q_1))\cap \mathrm{Conv}(X(Q_r))\right|.
\]
We note that for every $Q_1\in\mathcal Q_1$
\begin{equation}\label{maybeuseful}
\sum_{\substack{Q_r\in\mathcal Q_r\\ \Phi(Q_1)\cap\mathrm{Conv}(\Phi(Q_r))\neq \emptyset}}\mu(Q_1,Q_r)\le r^2 m(r)\mathrm{Area}(X(Q_1)),
\end{equation}
from which, summing over $Q_r\in\mathcal Q_r$ we deduce
\begin{equation*}
\begin{aligned}
&&\sum_{Q_r\in\mathcal Q_r}\mathrm{Area}(X(Q_r))=\sum_{Q_r\in\mathcal Q_r}\sum_{\substack{Q_1\in\mathcal Q_1\\ \Phi(Q_1)\cap\mathrm{Conv}(\Phi(Q_r))\neq \emptyset}}\mu(Q_1,Q_r)\\
&=&\sum_{Q_1\in\mathcal{Q}_1}\sum_{\substack{Q_r\in\mathcal Q_r\\ \Phi(Q_1)\cap\mathrm{Conv}(\Phi(Q_r))\neq \emptyset}}\mu (Q_1,Q_r)\le r^2 m(r) \sum_{Q_1\in\mathcal{Q}_1}r^2 m(r)\mathrm{Area}(X(Q_1)),
\end{aligned}
\end{equation*}
thus yielding the first inequality in \eqref{area_r_bd}.
As for the second inequality  in \eqref{area_r_bd}, it is enough to notice that the inequality \eqref{maybeuseful} is in fact an equality if $\mathrm{dist}(X(Q_1),X(\partial \mathcal G_\alpha))>Cr$ for some universal constant $C>0$.

\medskip

We now show \eqref{side_r_bd}. Since $\{a,b\}\in\mathrm{Sides}(\mathcal{Q}_r)$,  the segment $[X(a),X(b)]$ is included in the image of a discrete chart, and then $\{p,q\}\in \mathcal{S}_\alpha$ for all $\{p,q\}\in\mathrm{Sides}(\mathcal{Q}_1)$.

\medskip

A direct packing bound implies that the number of edges in this neighborhood is of order $r$. Since their length is of order $1$ we get the desired upper bound \eqref{side_r_bd}.

\medskip

To prove \eqref{length_r_bd}, first note that it suffices to bound the number of squares $Q_r\in\mathcal Q_r$ such that $[X(a),X(b)]\cap \mathrm{Conv}(X(Q_1))$ meets the convex hull of $Q_1$ (see proof of Proposition \ref{prop:distortion} for further details).

\medskip

We further reduce the problem noting that to intersect the hull of $Q_1$ is equivalent to intersecting one of its sides, so it suffices to bound the number of squares of scale $r$ one of whose edges meets an edge of scale $1$. The desired bound \eqref{length_r_bd} then follows from \eqref{side_r_bd} via a double counting procedure similar to the proof of \eqref{card_r_bd}. The details are left to the reader.

\medskip

Finally, to prove \eqref{badsides_r_bd}, we note that {if $\{p,q\}\in \mathrm{Sides}(\mathcal Q_{\sqrt2 r})\Delta \mathrm{Diag}(\mathcal Q_r)$, then there exists a point of $\partial\mathcal{G}_\alpha$ which is mapped to the $4r$-neighborhood of $\{p,q\}$.} We then proceed by a packing argument as before, and we obtain \eqref{badsides_r_bd}.
\end{proof}

\subsection{Control of large-scale deformation errors by short-scale deformations}\label{sec_controlarea}
In this section we improve upon the finite-range result of Theorem \ref{thm_fr_crystal}. This is done by effectively controlling short-distance interactions via the {nearest-neighbors} interaction, and treating all the large-distance interactions effectively as perturbation terms.

\medskip

The following result is the ``square lattice version'' of \cite[Proposition 2.10]{T}. It uses as the Euclidean geometry basic tool again Heron's formula, but the reasoning is different. Combinatorially, if $\boxtimes$ is the complete graph on $4$ vertices $\{a,b,c,d\}$ there are $4$ distinct triangles in this graph, and each edge is covered by $2$ triangles. On the metric side, if $\widetilde\boxtimes=X(\{a,b,c,d\})$ is realized as a plane quadrilateral which is a small deformation of a square $\boxtimes$ congruent to $\{0,1\}^2\subset \mathbb R^2$, then pairs of triangles which have in common only a diagonal make a decomposition of $\widetilde\boxtimes$. Then we compute the area by using Heron's formula on all triangles $\widetilde\Delta \subset \widetilde\boxtimes$, where $s_{\widetilde\Delta}$ is the semiperimeter of $\widetilde\Delta$:
\begin{equation}\label{heron}
 \mathrm{Area}(\widetilde\boxtimes)=\frac12\sum_{\widetilde \Delta\subset\widetilde\boxtimes}\mathrm{Area}(\widetilde\Delta)=\frac12\sum_{\widetilde\Delta\subset \widetilde\boxtimes}\sqrt{s_{\widetilde\Delta}\prod_{\widetilde e \in\widetilde\Delta}(s_{\widetilde\Delta}-|\widetilde e|)}.
\end{equation}
We apply Taylor expansion around the sidelengths $|e|$ corresponding to $\boxtimes\simeq \{0,1\}^2$ and call $|\widetilde e|=|e|+\delta_e$ the perturbed sidelengths. Then for $\Delta$ a triangle of sidelengths $1,1,\sqrt2$ we get $\frac{d}{d|e|}\left.\mathrm{Area}(\widetilde\Delta)\right|_{\widetilde\Delta=\Delta}$ equal to $0$ if $|e|=\sqrt2$ and equal to $\mathrm{Area}(\Delta)$ if $|e|=1$, thus the diagonal contributions disappear (this is due to the fact that $\Delta$ has a right angle opposite to the sides of length $\sqrt2$). Thus Taylor expansion gives using \eqref{heron}, for a function $O(\cdot)$ which is bounded if $\alpha<\alpha_0$ with $\alpha_0$ small enough,
$$
 \mathrm{Area}(\widetilde\boxtimes)=1 + \sum_{e\in \boxtimes, |e|=1} \delta_e +O\left(\sum_{e\in\boxtimes}\delta_e^2\right).
$$
Given $\alpha>0$, introducing the scaling factor $r$, we denote by $r\widetilde\boxtimes=:\widetilde{Q}_r$ the deformations of the square with sidelength $r$, i.e., $\widetilde{Q}_r\sim_\alpha\{0,r\}^2\subset \mathbb R^2$. 
Moreover, denoting by $\phi_r:\{0,r\}^2\to \widetilde{Q}_r$ the $\alpha$-deformation map given by Definition \ref{def:deform}, whenever $\widetilde{Q}_r$ is clear from the context, we set $\delta:=\delta_{\phi_r}$ (with $\delta_{\phi_r}$ given in Definition \ref{def:deform}). Notice that, by construction,  $\delta(x,y)=||x-y|-r|$ for $\{x,y\}$ sides of $\widetilde{Q}_r$ and $\delta(x,y)=||x-y|-\sqrt 2r|$ for $\{x,y\}$ diagonals of $\widetilde{Q}_r$.

\medskip

With this notation (for $\widetilde{Q}_r=X(Q_r)$), by Proposition \ref{prop:distortion} and by \eqref{length_r_bd}, for every $r\in\mathcal D$, $Q_r\in\mathcal Q_r$, we have, {writing $Q_1:=\widetilde\boxtimes$},
 \begin{equation}\label{fjm2}
\sum_{\substack{x,y\in X(Q_r)\\ x\neq y}} \delta^2(x,y)\le \bar C m(r)\ r^2\ \sum_{\substack{\{x,y\}\in X(Q_1)\\ x\ne y}}\delta^2(x,y),
 \end{equation}
 for some universal constant $\bar C>0$.

\medskip

Then using the chain rule for derivatives like in Theil's \cite[Proposition 2.10]{T}, one can easily get the following result.
\begin{proposition}\label{prop:taylor}
 There exist $C_8,\alpha_2>0$ such that for all $\alpha\in [0,\alpha_2)$ for $v\in C^2([0,\infty))$ and $\lambda>0$, if $\widetilde{Q}_r\subset \mathbb R^2$ satisfies $\widetilde{Q}_r\sim_\alpha\{0,r\}^2\subset \mathbb R^2$ then we have
 \begin{equation}\label{tayloreq}
 \begin{split}
 \left|\frac{v'(r)}{r}\left(\mathrm{Area}(\widetilde{Q}_r) - r^2\right) + 4v(r) - \sum_{\substack{x,y\in \widetilde{Q}_r\\ \{x,y\}\sim_\alpha\{0,r\}}}v(|x-y|)\right|\\
 \le C_8\left(\frac{|v'(r)|}{r}+\|v''\|_{L^\infty(rE_\alpha^1)}\right)\sum_{\substack{x,y\in \widetilde{Q}_r\\ x\neq y}}\delta^2(x,y):=e(v,r)\sum_{\substack{x,y\in \widetilde{Q}_r\\ x\neq y}}\delta^2(x,y).
 \end{split}
 \end{equation}
\end{proposition}
In what follows, with an abuse of notation, we still write $e(v,r)$ for the error term as in \eqref{tayloreq}, even if the constant $C$ changes from line to line, as long as $C$ remains independent of $X$.

\medskip

We are now ready to state the long-range version of Theorem \ref{thm_fr_crystal}.
We recall that the constant $\overline{\alpha}$ is given by Lemma \ref{lem_lowerbound},
 the constants $\alpha_0>0$, $C_3>1$ are provided by Lemma \ref{lem:combintometric}, and $\alpha'_0>0$ is given by Lemma \ref{cor_squarerigid}.
 
Given $\alpha,\alpha'',\epsilon>0$, $p>4$, we denote by $K=K(\alpha,\alpha'',p)$ the constant given by Lemma \ref{lem:mindistance} for  $r_{min}:=1-\alpha$, $C_2:=1$, $C_1:=\epsilon$ and $r_0:=\sqrt{2}+{\alpha''}$.
\begin{theorem}[crystallization for smooth one-well potentials]\label{thm_lr_crystal}
Let $\alpha,\alpha',\alpha''>0$ be such that $\alpha'<\alpha<\frac{1}{C_3}\alpha''<\frac{1}{C_3}\min\{(2-\sqrt2)/4,\overline{\alpha}, \alpha_0,\alpha'_0\}$ and let $p>4$. 

Set $r_{min}:=1-\alpha$, $C_1:=\epsilon$, $r_0:=\sqrt 2+\alpha''$, $C_2:=1$ and let $K>0$ be as in Lemma \ref{lem:mindistance}.
There exist two constants $\overline{c},\overline{\epsilon}>0$ such that for every $c',c''\in(0,\overline {c}]$ and $\epsilon\in (0,\overline\epsilon]$ the following result holds true:
If  $V,W\in C^2_{pw}((0,\infty))$ are related by \eqref{WfromV} and satisfy
\begin{itemize}
\item[(0)] $\min_{s>0}W(s)=-1$,
\item[(1)] $V$ is convex in $E_{\alpha''}$ and $V$ satisfies $\inf_{r\in E_{\alpha''}\setminus [1,\sqrt 2]}V''_{\pm}(r)\ge c$,
\item[(2)] $\VV$ satisfies \eqref{cond_crit} and \eqref{conditions} with constant $c'$,
 \item[(3)] $\sup_{r\in E_{\alpha'}}V(r)<-\frac{15}{16}-c''$,
 \item[(4)] $V(r)>-\frac12$ if  $r\notin (1-\alpha,\sqrt2+\alpha)$,
 \item[(5)] $V(r)\ge K$ if $r\le 1-\alpha$,
\item[(6')] $V(r)\le 0$ for every $r\ge 1$ and $|V(r)|,\,r|V'(r)|,\,r^2|V''(r)|<\epsilon r^{-p}$ for $r\ge \sqrt{2}+{\alpha''}$,
\end{itemize}
then
\begin{equation}\label{energy_lr}
N\overline{\mathcal E}_{\mathrm{sq}}[V]\le \mathcal E[V](N)\le N\overline{\mathcal E}_{\mathrm{sq}}[V]+O(N^{1/2})\qquad\textrm{as }N\to +\infty.
\end{equation}
\end{theorem}

Before the proof we connect $\overline{\mathcal{E}}_{\mathrm{sq}}[V]$ to an $\mathcal E_4$-minimization problem, in a self-contained result.

\medskip

We first introduce some notations. Recalling the decomposition \eqref{splitgraph}, for every $t>0$ we set 
\begin{subequations}\label{defntildev}
\begin{equation}
\widetilde\VV(t^2):=\sum_{r\in\widetilde{\mathcal D}\setminus\{1\}}m(r)\VV(t^2r^2),\quad \VV_*(t^2):=\sum_{r\in\widetilde{\mathcal D}}m(r)\VV(t^2r^2)=\VV(t^2)+\widetilde \VV(t^2),
\end{equation}
and, as above,
\begin{equation}
\widetilde V(t):=\widetilde\VV(t^2),\quad V_*(t):=\VV_*(t^2)\quad\mbox{ for all }t>0.
\end{equation}
\end{subequations}

Finally, for every $\Lambda\in\mathcal{L}$ with $0\in\Lambda$, we denote by $\mu(\Lambda)$ the set of shortest vectors in $\Lambda\setminus\{0\}$, i.e.,
$$
\mu(\Lambda):=\{z\in\Lambda\setminus\{0\}\,:\,|z|\le |w|\qquad\textrm{for all }w\in\Lambda\setminus\{0\}\}.
$$
\begin{proposition}\label{squarelatticeasy}
Let $\alpha'',\epsilon>0$ and $p>4$.
Let $V\in C^2_{pw}((0,\infty))$ satisfy assumption (6') of Theorem \ref{thm_lr_crystal}. Assume that the minimum $\min \mathcal E_4[V_*]$ is achieved at a square $\overline{\boxtimes}$. Then we have
\begin{equation}\label{energy_lr2}
\mathcal E_4[V_*](\overline{\boxtimes})=\overline{\mathcal E}_{\mathrm{sq}}[V].
\end{equation}
\end{proposition}
\begin{proof}
Let $\underline t>0$ be the value at which the minimum in the definition \eqref{eppz2_intro} is achieved. Up to scaling, we may, and will, assume that $\underline t=1$. Let $W\in C^2_{pw}((0,\infty))$ be defined by \eqref{WfromV}.

By \eqref{splitgraph} of Lemma \ref{lem_splitgraph} and by Lemma \ref{lem:mrm2r}, we have 
\begin{equation}\label{Vstar}
\begin{aligned}
\sum_{z\in \mathbb Z^2\setminus\{0\}}\VV(|z|^2)&=\sum_{\substack{\Lambda\in \mathcal L\\ 0\in\Lambda}}\sum_{z\in\mu(\Lambda)}(\VV(|z|^2)+\VV(2 |z|^2))\\
&=4\sum_{r\in\widetilde{\mathcal D}}(m(r)\VV(r^2)+m(\sqrt2 r)\VV(2r^2))= 4(\VV_*(1)+\VV_*(2)).
\end{aligned}
\end{equation}
For every $R>0$, by \eqref{Vstar} we have
\begin{equation}\label{sumcompare}
\begin{aligned}
&\sum_{ x\in\mathbb Z^2\cap B_R }\sum_{ y\in(\mathbb Z^2\cap B_R)\setminus\{x\} } \VV(|x-y|^2)\\ 
=&\sum_{x\in \mathbb Z^2\cap B_R}\sum_{y\in\mathbb Z^2\setminus\{x\}}\VV(|x-y|^2)-\sum_{x\in\mathbb Z^2\cap B_R}\sum_{y\in\mathbb Z^2\setminus B_R}\VV(|x-y|^2)\\
=&4\sharp(\mathbb Z^2\cap B_R)(\VV_*(1)+\VV_*(2))-S ,
\end{aligned}
\end{equation}
where we have set
$$
S:=\sum_{x\in\mathbb Z^2\cap B_R}\sum_{y\in\mathbb Z^2\setminus B_R}\VV(|x-y|^2).
$$
By assumption (6') we have that the series $\sum_{z\in \mathbb Z^2\setminus\{0\}}|\VV(|z|^2)|$ converges so that
\[
 \frac{|S|}{\sharp(\mathbb Z^2\cap B_R)}\le\sum_{z\in \mathbb Z^2\setminus B_{R}}|V(|z|)|\le \omega_R,
\]
where $\omega_R\to 0$ as $R\to +\infty$.
Thus from \eqref{sumcompare} we find
\begin{equation}\label{point1}
\begin{aligned}
\overline{\mathcal{E}}_{\mathrm{sq}}[V]=& \lim_{R\to\infty}\frac{1}{\sharp(\mathbb Z^2\cap B_R)}\sum_{\substack{x,y\in\mathbb Z^2\cap B_R\\x\neq y}}V(|x-y|)\\
=& \frac12\lim_{R\to\infty}\frac{1}{\sharp(\mathbb Z^2\cap B_R)}\sum_{x\in\mathbb Z^2\cap B_R}\sum_{y\in(\mathbb Z^2\cap B_R)\setminus\{x\}}V(|x-y|)\\
=&2(V_*(1)+V_*(\sqrt{2}))=\mathcal E_4[V_*](\boxtimes)\ge \mathcal E_4[V_*](\overline{\boxtimes}),
\end{aligned}
\end{equation}
where $\boxtimes$ is a unit square configuration, which shows the inequality ``$\le$'' in \eqref{energy_lr2}.

\medskip 

On the other hand, using the assumption that $\mathcal E_4[V_*]$ achieves its minimum at a square, say it achieves the minimum at $\{0,\tilde t\}^2=\overline{\boxtimes}$. Then we renormalize $\tilde t=1$ and we can repeat the above computations verbatim, getting the inequality ``$\ge$'' in \eqref{energy_lr2}, thus concluding the proof.
\end{proof}

\subsection{Proof of Theorem \ref{thm_lr_crystal}}
Let $X_N:=X(\Xi_N)$ be a minimizer of $\E[V]$ in $\mathcal{X}_N(\R^2)$. To ease the notations, we set $x_a:=X(a)$ for every $a\in\Xi_N$. Here and in the whole section an \emph{edge} $\{a,b\}$ is any pair of distinct points $a,b\in\Xi_N$.
We denote by 
{
\[
\overline{\mathcal{S}}=\overline{\mathcal{S}}(\Xi_N):=\{\{a,b\}:\ a\neq b\in \Xi_N\},
\]}
the set of all the edges associated to $\Xi_N$. In what follows, for every square $Q_r:=\{\xi_1,\xi_2,\xi_3,\xi_4:=\xi_0\}$ at scale $r$ {(also called $r$-square)}, the {\it sides} of $Q_r$ are $\{\xi_{i-1},\xi_i\}$ for $i=1,\ldots,4$, whereas the edges of $Q_r$ are given by the sides plus the diagonals $\{\xi_1,\xi_3\}$ and $\{\xi_2,\xi_4\}$.

\medskip 

With a little abuse of notations the (universal) constants appearing in the estimates may change from line to line.

\medskip

{\bf Step 1: Decomposition of $\E[V]$ into contributions of type $\E_4[V]$}

\medskip

Using assumption (6') and Lemma \ref{lem:mindistance}, we find that the minimal distance between points in $X(\Xi_N)$ is strictly larger than  $1-\alpha$. This and the fact that $V(r)$ is negative for $r\ge 1$ gives that, via \eqref{energysplit}, and denoting by $\mathcal{NQ}:=\mathcal{NQ}^{(1)}\cup\mathcal{NQ}^{(2)}$ with notation as in \eqref{energysplit}, there holds:
\begin{eqnarray}\label{energysplit2}
\sum_{\{a,b\}\in\mathcal{S}}V(|x_a-x_b|)
&\ge&\frac{1}{2}\sum_{{r\in\widetilde{\mathcal D}\setminus\{1\}}}\ \sum_{Q\in\mathcal Q_r\cup\mathcal Q_{\sqrt2 r}}\ \sum_{\{a,b\}\in \mathrm{Sides}(Q)}V(|x_a-x_b|)\nonumber\\
&&+\sum_{Q_1\in\mathcal Q_1}\E_4[V](X(Q_1))\ \ +\sum_{\{a,b\}\in\mathcal{NQ}}V(|x_a-x_b|).
\end{eqnarray}
To justify the above inequality, note that (i) as a consequence of Lemma \ref{lem_splitgraph}, all possible scales from $\mathcal D$ are partitioned into pairs $\{r,\sqrt 2 r\}$ for $r\in\widetilde{\mathcal D}$, (ii) each edge can be covered by at most $2$ squares, (iii) diagonals of $1$-squares include all sides of $\sqrt2$-squares. Then we can treat separately the multiplicity of edges which are either (a) the side of two $r$-squares with $r\ge 2$, (b) the side of exactly one $r$-square with $r\ge 2$, (c) the side of two $\sqrt 2$-squares, (d) the side of exactly one $\sqrt2$-square, (e) the diagonal of one $1$-square but not the side of any $\sqrt2$-square, (f) the side of two $1$-squares (g) the side of precisely one $1$-square, (h) not the side or diagonal of any $r$-square for $r\in \mathcal D$. The two sides of \eqref{energysplit2} account for multiplicity $1$ precisely, except for cases (d) and (e), in which we have multiplicity $2$. But in these cases, $V$ takes a negative sign for the corresponding edge lengths, and we get the desired inequality.

\medskip

We now treat the first sum on the right hand side of \eqref{energysplit2} for  $r\ge 2$.
Recalling the notations introduced before Proposition \ref{prop:taylor}, applying Proposition \ref{prop:taylor} with $\widetilde{Q}_r=X(Q_r)$ 
 and $v=V$, we get
\begin{eqnarray*}
\lefteqn{\sum_{Q_r\in \mathcal Q_r}\sum_{\{a,b\}\in\mathrm{Sides}(Q_r)}V(|x_a-x_b|)}\nonumber\\
&\ge&\sum_{Q_r\in\mathcal Q_r}\left(4V(r) + \frac{V'(r)}{r}\left(\mathrm{Area}(X(Q_r))-r^2\right)-e(V,r)\sum_{\substack{x,y\in{X(Q_r)}\\ x\neq y}}\delta^2(x,y)\right),
\end{eqnarray*}
which together with \eqref{card_r_bd} and \eqref{area_r_bd} of Proposition \ref{prop:card_bounds}, implies
\begin{eqnarray*}
 &&\sum_{Q_r\in \mathcal Q_r}\sum_{\{a,b\}\in\mathrm{Sides}(Q_r)}V(|x_a-x_b|)\ge m(r)\sum_{Q_1\in\mathcal Q_1}\left(4V(r)+\frac{V'(r)}{r}\left(r^2\mathrm{Area}(X(Q_1))-r^2\right)\right)\nonumber\\
 &&-\ C\left(|V(r)|+r{|V'(r)|}\right) m(r)r^2\sharp \partial\mathcal{G}_\alpha\quad -\quad e(V,r)\sum_{Q_r\in\mathcal Q_r}\sum_{\substack{ x,y\in X(Q_r)\\ x\neq y}}\delta^2(x,y).
 \end{eqnarray*}
Using again Proposition \ref{prop:taylor} with  $r=1$, $\widetilde Q_1=X(Q_1)$, and $v(x)=V_r(x):=V(rx)$, together with the $2$-homogeneity $r^2\mathrm{Area}(X(Q_1))=\mathrm{Area}(r X( Q_1))$, we get
\begin{equation}\label{pass3}
\begin{aligned}
&\sum_{Q_r\in \mathcal Q_r}\sum_{\{a,b\}\in\mathrm{Sides}(Q_r)}V(|x_a-x_b|)
\ge m(r)\sum_{\{a,b\}\in\mathrm{Sides}(\mathcal Q_1)}V\left(r|x_a-x_b|\right)\\
&\phantom{\sum_{Q_r\in \mathcal Q_r}\sum_{\{a,b\}\in\mathrm{Sides}(Q_r)}V(|x_a-x_b|)}- m(r)e(V_r,1)\sum_{Q_1\in\mathcal Q_1}\sum_{\substack{x,y\in X(Q_1)\\ x\neq y}}\delta^2(r x,r y) \\
 &-\ C\left(|V(r)|+r|V'(r)|\right)m(r)r^2\sharp \partial\mathcal{G}_\alpha\quad-\quad e(V,r)\sum_{Q_r\in\mathcal{Q}_r}\sum_{\substack{x,y\in X(Q_r)\\ x\neq y }}\delta^2(x,y).
\end{aligned}
\end{equation}
In view of \eqref{fjm2} the error terms from the last sum in \eqref{pass3} can also be re-interpreted as a scale-$1$ error term, i.e., 
\begin{equation}\label{pass3more}
\sum_{Q_r\in\mathcal Q_r}\sum_{\substack{x,y\in X(Q_r)\\ x\neq y}}\delta^2(x,y)\le \bar C m(r)r^2 \sum_{Q_1\in\mathcal Q_1}\sum_{\substack{x,y\in X( Q_1)\\ x\neq y}}\delta^2(x,y).
\end{equation}
Therefore, noting that $\delta^2(r x,r y)=r^2\delta^2(x,y)$, and setting 
\begin{equation}\label{errbdterms}
\begin{aligned}
\mathrm{err}_1(r,V):=&C\left(|V(r)| +r |V'(r)|\right)m(r)r^2,\\ 
\mathrm{err}_2(r,V):=&m(r)r^2 \left(e(V_r,1)+\bar Ce(V,r)\right),
\end{aligned}
\end{equation}
we can reorder terms in \eqref{pass3}{, and the final inequality we get is}
\begin{multline}\label{notationpass}
{\sum_{Q_r\in \mathcal Q_r}\sum_{\{a,b\}\in\mathrm{Sides}(Q_r)}V(|x_a-x_b|)\ge m(r)\sum_{\{a,b\}\in\mathrm{Sides}(\mathcal Q_1)}V\left(r|x_a-x_b|\right)}\\
-\mathrm{err}_1(r,V)\sharp\partial\mathcal{G}_\alpha - \mathrm{err}_2(r,V)\sum_{Q_1\in\mathcal Q_1}\sum_{\substack{x,y\in X( Q_1)\\ x\neq y}}\delta^2(x,y).
\end{multline}
Similarly, for $\sqrt2r$-square contributions in \eqref{energysplit2}, we get
\begin{multline}\label{notationpass2}
\sum_{Q_{\sqrt{2}r}\in \mathcal Q_{\sqrt{2}r}}\sum_{\{a,b\}\in\mathrm{Sides}( Q_{\sqrt{2}r})}V(|x_a-x_b|)\ge m(\sqrt2 r)\sum_{\{a,b\}\in\mathrm{Sides}(\mathcal Q_{\sqrt 2})}V\left(r|x_a-x_b|\right)\\
-\mathrm{err}_1(\sqrt 2 r,V)\sharp\partial\mathcal{G}_\alpha - \mathrm{err}_2(\sqrt 2 r,V)\sum_{Q_{\sqrt 2}\in\mathcal Q_{\sqrt 2}}\sum_{\substack{x,y\in X( Q_{\sqrt 2})\\ x\neq y}}\delta^2(x,y).
\end{multline}

We now estimate {the above} error terms. 

\medskip

Let $\rho\ge 2$.
Using that $m(\rho)\le C \rho$ and assumption (6'), we get that there exist two constants $ C,C_{\mathrm{pot}}> 0$ (depending only on the dimension) such that
\begin{subequations}\label{err23bd}
\begin{eqnarray}
\mathrm{err}_1(\rho,V)&\le&C\epsilon \rho^{3-p}\label{err2bd}\\
\mathrm{err}_2(\rho,V)&\le&C_{\mathrm{pot}}\epsilon \rho^{3-p}\label{err3bd}.
\end{eqnarray}
\end{subequations}

Now \eqref{notationpass} and \eqref{notationpass2}, together with \eqref{err23bd}, give that for every $r\in\mathcal D$ with $r\ge 2$
\begin{subequations}\label{notationpass2.0}
\begin{eqnarray}
\nonumber
&&\sum_{Q_r\in\mathcal Q_r}\sum_{\{a,b\}\in \mathrm{Sides}(Q_r)}V(|x_a-x_b|)+\sum_{Q_{\sqrt 2r}\in\mathcal Q_{\sqrt 2r}}\sum_{\{a,b\}\in \mathrm{Sides}(Q_{\sqrt 2r})}V(|x_a-x_b|)\\
\label{notationpass2.0a}
&\ge&m(r)\sum_{\{p,q\}\in\mathrm{Sides}(\mathcal Q_1)}V\left(r|x_p-x_q|\right)+m(\sqrt 2 r)\sum_{\{p,q\}\in\mathrm{Sides}(\mathcal Q_{\sqrt 2})}V\left(r|x_p-x_q|\right)\\
\label{notationpass2.0b}
&&-C_{\mathrm{pot}}\epsilon r^{3-p}\left[\sum_{Q_{1}\in\mathcal Q_{1}}\sum_{\substack{x,y\in X( Q_{1})\\ x\neq y}}\delta^2(x,y)+\sum_{Q_{\sqrt 2}\in\mathcal Q_{\sqrt 2}}\sum_{\substack{x,y\in X( Q_{\sqrt 2})\\ x\neq y}}\delta^2(x,y)\right]\\
\nonumber
&&-C\epsilon r^{3-p}\sharp\partial\mathcal{G}_\alpha.
\end{eqnarray}
\end{subequations}

As for the term in \eqref{notationpass2.0a},
since {$m(\sqrt 2 r)=m(r)$}, for every $r\in\mathcal{D}$ with $r\ge 2$ we have that 
\begin{equation}\label{notationpass2.1}
\begin{aligned}
&m(r)\sum_{\{p,q\}\in\mathrm{Sides}(\mathcal Q_1)}V\left(r|x_p-x_q|\right)+m(\sqrt 2 r){\sum_{\{p,q\}\in\mathrm{Sides}(\mathcal Q_{\sqrt 2})}V\left(r|x_p-x_q|\right)}\\
=&m(r)\left[\sum_{\{p,q\}\in\mathrm{Sides}(\mathcal Q_1)}V\left(r|x_p-x_q|\right)+{\sum_{\{p,q\}\in\mathrm{Sides}(\mathcal Q_{\sqrt 2})}V\left(r|x_p-x_q|\right)}\right]\\
=&2m(r)\sum_{Q_1\in\mathcal Q_1}\E_4[V_r](X(Q_1))
-{2 m(r)\sum_{\{p,q\}\in\mathrm{Diag}(\mathcal{Q}_1)}} V_r(|x_p-x_q|)\\
&\phantom{2m(r)\sum_{Q_1\in\mathcal Q_1}\E_4[V_r](X(Q_1))
}+{m(r)\sum_{\{p,q\}\in\mathrm{Sides}(\mathcal Q_{\sqrt 2})}} V_r(|x_p-x_q|)\\
 \ge & 2m(r)\sum_{Q_1\in\mathcal Q_1}\E_4[V_r](X(Q_1))
-{ m(r)\sum_{\{p,q\}\in\mathrm{Diag}(\mathcal{Q}_1)}} V_r(|x_p-x_q|)\\
&\phantom{2m(r)\sum_{Q_1\in\mathcal Q_1}\E_4[V_r](X(Q_1))
}+{m(r)\sum_{\{p,q\}\in\mathrm{Sides}(\mathcal Q_{\sqrt 2})}} V_r(|x_p-x_q|)\\
:=& 2m(r)\sum_{Q_1\in\mathcal Q_1}\E_4[V_r](X(Q_1))+\mathrm{err}_3(r,V),
\end{aligned}
\end{equation}
where  the inequality is a consequence of assumption (6') -- in particular the fact that $V(r)\leq 0$ for $r\geq 1$ -- and the last line is a definition of $\mathrm{err}_3(r,V)$.

\medskip

To justify the {second} equality in \eqref{notationpass2.1} note that, by definition  \eqref{defEW}, $\E_4[V_r](Q)$ has coefficient $1/2$ for terms coming from the sides of $Q$, and coefficient $1$ in front of diagonal terms.

\medskip

By using \eqref{badsides_r_bd}, the bound $m(r)\le Cr$, the fact that $r\ge 2$, and again assumption (6'), we have
\begin{equation}\label{estimerr1}
|\mathrm{err}_3(1,V_r)|\le C\epsilon r^{3-p}\sharp{\partial}\mathcal{G}_\alpha.
\end{equation}

We finally pass to the estimate of the two sums in \eqref{notationpass2.0b}.
By applying \eqref{fjm2} with $r=\sqrt 2$ we have
\begin{equation}\label{fjm2consequence}
\sum_{Q_{\sqrt 2}\in\mathcal Q_{\sqrt 2}}\sum_{\substack{x,y\in X( Q_{\sqrt 2})\\ x\neq y}}\delta^2(x,y)\le 
{\sum_{Q_{1}\in\mathcal Q_{1}}}\sum_{\substack{x,y\in X( Q_{1})\\ x\neq y}}\delta^2(x,y){\le 2}{\sum_{Q_{1}\in\mathcal Q_{1}}}\E_4[\overline{\delta^2}](X(Q_1)),
\end{equation}
where $\overline{\delta^2}:E_\alpha\to[0,\infty)$ is defined by $\overline{\delta^2}(t)=(1-t)^2$ for $t\in E^1_\alpha$ and $\overline{\delta^2}(t)=(\sqrt2-t)^2$ for $t\in E^2_\alpha$.

\medskip

By \eqref{notationpass2.0}, \eqref{notationpass2.1}, \eqref{estimerr1}, and \eqref{fjm2consequence} we can conclude that for every $r\in\widetilde{\mathcal{D}}$ with $r\ge 2$ we have
\begin{equation}\label{firstdecomp}
\begin{aligned}
&\sum_{Q_r\in \mathcal Q_r}\sum_{\{a,b\}\in\mathrm{Sides}(Q_r)}V(|x_a-x_b|)+\sum_{Q_{\sqrt{2}r}\in \mathcal Q_{\sqrt{2}r}}\sum_{\{a,b\}\in\mathrm{Sides}( Q_{\sqrt{2}r})}V(|x_a-x_b|)\\
\ge& 2m(r)\sum_{Q_1\in\mathcal Q_1}\E_4[V_r](X(Q_1))-C_{\mathrm{pot}}\epsilon r^{3-p}\sum_{Q_1\in\mathcal Q_1}\E_4[\overline{\delta^2}](X(Q_1))-C\epsilon r^{3-p}\sharp\partial\mathcal{G}_\alpha.
\end{aligned}
\end{equation}

\medskip

{\bf Step 2: The sum over scales.}

\medskip

By summing \eqref{firstdecomp} over all the scales $r\in\widetilde{\mathcal{D}}\setminus\{1\}$ {and} using the linearity of $\E_4[V]$ with respect to $V$, we get
\begin{multline}\label{sumscale}
\frac 1 2\sum_{r\in\widetilde{\mathcal D}\setminus\{1\}}\left[\sum_{Q_r\in \mathcal Q_r}\sum_{\{a,b\}\in\mathrm{Sides}(Q_r)}V(|x_a-x_b|)+\sum_{Q_{\sqrt{2}r}\in \mathcal Q_{\sqrt{2}r}}\sum_{\{a,b\}\in\mathrm{Sides}( Q_{\sqrt{2}r})}V(|x_a-x_b|)\right]\\
\ge \sum_{Q_1\in\mathcal Q_1}\E_4[\widetilde{V}-C_{\mathrm{pot}}\epsilon \overline{\delta^2}](X(Q_1))-C\epsilon \sharp\partial\mathcal{G}_\alpha,
\end{multline}
where we have used also the very definition of $\widetilde{V}$ in \eqref{defntildev}  and the fact that $p>4$.
Therefore, by \eqref{energysplit2} {and \eqref{defntildev}}, using again the linearity of $\E_4$ with respect to $V$, we deduce
\begin{equation}\label{sumscale2}
\frac 1 2\sum_{\{a,b\}\in\mathcal{S}}V(|x_a-x_b|)\ge \sum_{Q_1\in\mathcal Q_1}\E_4[{V}_*-C_{\mathrm{pot}}\epsilon \overline{\delta^2}](X(Q_1))-C\epsilon \sharp\partial\mathcal{G}_\alpha+\sum_{\{a,b\}\in\mathcal{NQ}}V(|x_a-x_b|).
\end{equation}

{\bf Step 3: Existence of a minimizer for $\E_4[V_*-C_{\mathrm{pot}}\ep\overline{\delta^2}]$.}

\medskip

Here we show that there exists a unique minimizer of $\E_4[V_*-C_{\mathrm{pot}}\ep\overline{\delta^2}]$ in $\mathcal{X}_4(\R^2)/\mathrm{Isom}(\R^2)$ and that such a minimizer is a square. To this purpose, we notice that the assumptions (0)-(4) and (6') allow us to profit of the results in Subsection \ref{sec_4ptmin}, and in particular of Proposition \ref{prop_square}.
Indeed, by \eqref{defntildev} and assumption (6'), choosing  $\epsilon>0$ small enough we have
\begin{equation}\label{pertbd1}
\begin{aligned}
\|V-V_*\|_{C^{2}(1-\alpha,\infty)}=&\|\widetilde{V}\|_{C^{2}(1-\alpha,\infty)}=\left\|\sum_{r\in\widetilde{\mathcal{D}}\setminus\{1\}} m(r)W(r^2\cdot)\right\|_{C^{2}(1-\alpha,\infty)}\le \bar C\epsilon< c''',
\end{aligned}
\end{equation}
where $c'''$ is the constant in Proposition \ref{prop_square}.

By the  second part of the statement of Proposition \ref{prop_square} we get that there exists a unique minimizer of $\E_4[V_*]$ in $\mathcal{X}_4(\R^2)/\mathrm{Isom}(\R^2)$ that is a square with sidelengths in
$(1-\alpha,+\infty)$.

On the other hand, by a direct calculation one can easily check that $\overline{\delta^2}$ satisfies the assumptions of  Lemma \ref{lem_minsquare} and Lemma \ref{lem:onlysquare} so that $\boxtimes$ is the unique global minimizer of $\E_4[\overline{\delta^2}]$ in $\overline{\mathscr{S}}_\alpha$. Moreover, again by a direct calculation (or using Lemma \ref{lem_minsquare}), it easily follows that
$\mathcal E_4[\overline{\delta^2}](\boxtimes)=0$ and $\nabla\mathcal E_4[\overline{\delta^2}](\boxtimes)=0$, and thus $\mathcal E_4[\overline{\delta^2}](\widetilde\boxtimes)=O(|\boxtimes -\widetilde \boxtimes|^2)$ for a small perturbation $\widetilde\boxtimes$ of $\boxtimes$. 

This implies that for $\epsilon>0$ small enough the function $\mathcal E_4[V_*-C_{\mathrm{pot}}\epsilon \overline{\delta^2}]$ still has a strict minimum at $\overline\boxtimes$. By Proposition \ref{prop_square}, which can then be applied to $V_*$ and $V_* -C_{\mathrm{pot}}\epsilon\overline{\delta^2}$ up to diminishing $\epsilon$, we find that $\overline\boxtimes$ is also the unique minimum of $\mathcal E_4[V_*-C_{\mathrm{pot}}\epsilon \overline{\delta^2}]$.

\medskip

{\bf Step 4: Estimate of the last term in \eqref{sumscale2}.}

\medskip

Now we bound the terms from the last sum in \eqref{sumscale2}, i.e. the contributions not attributable to squares. Setting $\mathcal{NQ}_{\alpha''}:=\mathcal{NQ}\cap\mathcal{S}_{\alpha''}$,
we have
\begin{eqnarray}\label{boundnq} 
\lefteqn{ \sum_{\{a,b\}\in \mathcal{NQ}}V(|x_a-x_b|)}\nonumber\\
 &=&\sum_{\{a,b\}\in \mathcal{NQ}_{\alpha''}}V(|x_a-x_b|) + \sum_{d\in\sqrt2 +\alpha'' +\frac12+\mathbb N}\sum_{\substack{\{a,b\}\in\mathcal{NQ}\\ \ |x_a-x_b|\in\left[d-\frac12,d+\frac12\right]}}V(|x_a-x_b|)\nonumber\\
 &=:&I+\mathrm{err}_4(V,X).
\end{eqnarray}
As for $\mathrm{err}_4(V,X)$, since $X$ is a minimizer, by the separation result of Lemma \ref{lem:mindistance} and a packing bound, there exists $C>0$ depending only on the dimension and on $\alpha_0>0$ such that the $r$-neighborhood of $X$ contains at most $Cr^2$ points from $X$. This implies that
\begin{equation}\label{boundsnq2}
\sharp \left\{\{a,b\}\in\mathcal{NQ}:\ |x_a-x_b|\in\left[d-\frac12,d+\frac12\right]\right\}\le Cd^3\sharp \partial \mathcal{G}_\alpha. 
\end{equation}
The consequence of \eqref{boundnq} and \eqref{boundsnq2} is the following bound, valid under hypothesis $(6')$ with $p>4$:
\begin{eqnarray}\label{finalnq}
\lefteqn{\mathrm{err}_4(V,X)\ge \sum_{d\in\sqrt2+\alpha+\frac12+\mathbb N}\sharp \left\{\{a,b\}\in\mathcal{NQ}:\ ||x_a-x_b| - d|\le \frac12\right\}\min_{r:|r-d|\le\frac12}V(r)}\nonumber\\
&\ge&C\sharp \partial\mathcal{G}_\alpha \sum_{d\in\sqrt2+\alpha+\frac12+\mathbb N}d^3\min_{r\in\left[d-\frac12,d+\frac12\right]}V(r)\ge-C\epsilon\sharp\partial\mathcal{G}_\alpha.
\end{eqnarray}
\textbf{Step 5: End of proof, following the strategy of Theorem \ref{thm_fr_crystal}}

\medskip

By \eqref{sumscale2}, \eqref{boundnq} and \eqref{finalnq}, we have

\begin{eqnarray}\label{bound_esplit4old}
\mathcal E[V](X_N)&\ge& \sum_{Q_1\in\mathcal Q_1}\mathcal E_4[V_*- C_{\mathrm{pot}}\epsilon\overline{\delta^2}](X(Q_1))+I- C\epsilon\sharp\partial\mathcal{G}_\alpha,
\end{eqnarray}
where both the constants do only depend on $p$ and on the dimension. Now, in order to continue precisely along the strategy used for Theorem \ref{thm_fr_crystal}, we re-express the sum \eqref{bound_esplit4old} via a potential which vanishes at distance larger that $\sqrt2+\alpha''$. We define
\begin{equation}\label{v**}
 V_{**}(r):=(V_*(r)-C_{\mathrm{pot}}\epsilon\overline{\delta^2}(r))\mathds{1}_{r<\sqrt2+\alpha''}(r).
\end{equation}
For estimating $I$ {from \eqref{boundnq}} we apply a similar setup as in \eqref{resumq1} from the proof of Theorem \ref{thm_fr_crystal}. Indeed, a similar decomposition as in \eqref{resumq1} can be used also in this case. 

\medskip

Then \eqref{bound_esplit4old} can be rewritten in a form very similar to \eqref{resumq1}:
\begin{eqnarray}\label{bound_esplit4}
\mathcal E[V](X_N)&\ge& \sum_{Q_1\in\mathcal Q_1}\mathcal E_4[V_{**}](X(Q_1))+\sum_{\{a,b\}\in{\mathcal{NQ}_{\alpha''}}}V(|x_a-x_b|) - C\epsilon\sharp\partial\mathcal{G}_\alpha.
\end{eqnarray}
Then we have, directly from the Step 3 of the proof, that $\min\mathcal E_4[V_*]=\min\mathcal E_4[V_{**}]$, and moreover we have from Proposition \ref{squarelatticeasy} that 
\begin{equation}\label{minVss}
\mathcal E_4[V_{**}](\overline\boxtimes)=\min\mathcal E_4[V_{**}]=\overline{\mathcal E}_{\mathrm{sq}}[V].
\end{equation}
We next proceed exactly as in the proof of Theorem \ref{thm_fr_crystal} with $V$ replaced by $V_{**}$. We define the energy contribution of each point 
\[
 \mathcal E^p[V_{**}](X_N):=\frac12\sum_{q:q\neq p}V_{**}(|x_p-x_q|)=\frac12\sum_{q\in\mathcal N_{\alpha''}(p)\setminus\{p\}}V(|x_p-x_q|),
\]
and after enriching the graph $\mathcal G_\alpha$ by adding long edges and missing edges, we obtain a graph still denoted by $\overline{\mathcal G}$ and we reach the following version of \eqref{boundsEp2}, in which we use the same notation $\overline{\mathcal N}(p)$ as for \eqref{boundsEp2}: 
\begin{equation}\label{boundsEp2-2}
 \mathcal E^p[V_{**}](X_N)\ge \mathcal E_4[V_{**}](\overline\boxtimes) + \sum_{q\in\overline{\mathcal N}(p)\setminus\{p\}}\overline{\overline w}(\{p,q\}), 
\end{equation}
where we have the following substitute for the bounds \eqref{defbarw}:
\begin{equation}\label{defbarw-2}
 \overline{\overline w}(\{p,q\}):=\left\{\begin{array}{ll}
 0&\mbox{ if }\{p,q\}\in\mathcal S_\alpha\setminus \mathcal{NQ},\\
 -\frac{\mathcal E_4[V_{**}](\overline\boxtimes)}{2}-\frac12&\mbox{ if }\{p,q\}\in \mathcal{NQ}\cap \mathcal S_\alpha,\\
  -\frac{\mathcal E_4[V_{**}](\overline\boxtimes)}{2}-\frac14&\mbox{ if }\{p,q\}\in\mathcal S_{\alpha'',\alpha},\\
 -\frac{\mathcal E_4[V_{**}](\overline\boxtimes)}{2}&\mbox{ if }\{p,q\}\mbox{ missing edge}.
 \end{array}\right.
\end{equation}
Now with the notation for $m_{\alpha',V}, m_{\alpha',V_{**}}$ as in \eqref{resumq1-3}, we get instead of \eqref{barwbound} the bound 
\begin{equation}\label{barwbound-2}
 -\frac{\mathcal E_4[V_{**}](\overline\boxtimes)}{2}-\frac12\ge -\frac{m_{\alpha',V_{**}}}{2}-\frac12 \ge -\frac{m_{\alpha',V}}{2}-\frac12 - \widetilde C\epsilon,
\end{equation}
where the first inequality uses again assumption (3) and the definition of $m_{\alpha',V_{**}}$, and for the last bound we can use the fact that $V_{**}=V_*$ over $E_{\alpha'}$, and then the bound in \eqref{pertbd1} for $\widetilde V=V_*-V$.

\medskip

Now we define the lost weight function exactly as in the proof of Theorem \ref{thm_fr_crystal}, namely $\mathrm{lw}(\{p,q\})$ is equal to $1/4$ for long edges and to $1/2$ for missing edges, and we obtain the bound \eqref{lb2}. The cardinality of the set appearing on the right in \eqref{lb2} will be denoted by
\[
 N_1:=\sharp\left\{p:\ \mathcal N_\alpha(p)\cap\partial\mathcal G_\alpha\neq\emptyset\mbox{ or }\mathcal N_{\alpha''}(p)\setminus\mathcal N_\alpha(p)\neq\emptyset\right\}.
\]
Now as in Theorem \ref{thm_fr_crystal}, we sum over $p$ equation \eqref{boundsEp2-2} in order to get all contributions $\mathcal E^p[V_{**}](X_N)$. We find from \eqref{bound_esplit4} the following analogue of \eqref{lb3} with further error terms coming from \eqref{barwbound-2} and \eqref{bound_esplit4}:
\begin{eqnarray}\label{lb3-2}
 \mathcal E[V](X_N)&\ge& N\mathcal E_4[V_{**}](\overline\boxtimes)-C\epsilon\ \sharp\partial\mathcal G_\alpha + N_1\left(-4m_{\alpha',V} - 4+\frac1{16}-\widetilde C\epsilon\right)\nonumber\\
 &\ge&N\mathcal E_4[V_{**}](\overline\boxtimes) +N_1\left(-4m_{\alpha',V} - 4+\frac1{16}-(C +\widetilde C)\epsilon\right)\\
 &\ge&N\mathcal E_4[V_{**}](\overline\boxtimes)\nonumber
\end{eqnarray}
where in the first passage we used the fact that $N_1\ge \sharp\partial \mathcal G_\alpha$ because $N_1$ measures the cardinality of a set containing $\partial\mathcal G_\alpha$, and in the last passage we used assumption (3) with $c''\ge (C +\widetilde C)\epsilon$.

\medskip

Now the property \eqref{minVss} and \eqref{lb3-2} conclude the proof of Theorem \ref{thm_lr_crystal}.

\appendix 

\section{Proof of Lemmas \ref{lem:combintometric} and \ref{cor_squarerigid}} \label{proof_lem:combintometric}
Since many constants are introduced throughout this section, for not confusing with the other constants above, we often replace $\alpha$ with $\ep$.
 For every $\ep>0$ and for every $\Xi$, $X:\Xi\to\R^2$, we recall that 
$$
\mathcal S_\ep=\{\{p,q\}\,:\,p,q\in\Xi,\,|x_p-x_q|\in(1-\ep,\sqrt 2+\ep)\},
$$
where $x_p:=X(p)$ for every $p\in\Xi$.
We first prove two preliminary Lemmas that will be useful in the proof of Lemma \ref{lem:combintometric}.
\begin{lemma}\label{trialemma}
There exists $\ep'>0$ such that for every $\ep\in (0,\ep')$ the following holds. Let $X$ satisfy \eqref{mindist} with $r_{min}=1-\ep$. Let $p_1,p_2,p_3\in \Xi$ and set $x_i:=X(p_i)$ for all $i\in \{1,2,3\}$. Then,
\begin{itemize}
\item[(i)] if $\{p_1,p_2\}\in \mathcal{S}_\ep$ and $\{p_1,p_3\}\in \mathcal{S}_\ep$ but $\{p_2,p_3\}\not\in \mathcal{S}_\ep$, then in the triangle $\{x_1,x_2,x_3\}$, the interior angles satisfy $\hat{x}_1 \ge 60^\circ$ and $\hat{x}_2,\hat{x}_3 \le\arccos\left(\frac{1}{2\sqrt 2}\right)+O(\ep)$;
\item[(ii)] if all pairs amongst $p_1,p_2,p_3$ are in $\mathcal{S}_\ep$, then the interior angles of the triangle $\{x_1,x_2,x_3\}$ are in the interval 
\[
{ \left[\arccos\left(\frac 3 4\right)-O(\ep),\quad 90^\circ+O(\ep) \right]};
\]
 \item[(iii)] if $\{p_1,p_2\},\{p_1,p_3\}\in\mathcal S_\ep$, $|x_1-x_2|=1+O(\ep)$ but $\{p_2,p_3\}\not\in \mathcal{S}_\ep$, then in the triangle $\{x_1,x_2,x_3\}$, we have $\hat{x}_1\ge \arccos\left(\frac{1}{2\sqrt 2}\right) +O(\ep)$;
\item[(iv)] if all pairs amongst $p_1,p_2,p_3$ are in $\mathcal{S}_\ep$ and $|x_1-x_2|= 1+O(\ep)$, then
\begin{equation}\label{minan}
\widehat{x_1x_2x_3},\widehat{x_3x_1x_2}\ge {45^\circ-O(\ep).}
\end{equation}
\end{itemize}
\end{lemma}
\begin{proof}
We set $a:=|x_1-x_2|,\, b:=|x_1-x_3|,\, c:=|x_2-x_3|$.
We may assume, up to relabelling the points, that 
\begin{equation}\label{order_tri2}
1-\ep <a\le b<\sqrt 2 -\ep.
\end{equation}

{\it Proof of (i).} 
The statement follows from the law of cosines, which states that 
\begin{equation}\label{lawcosines2}
\cos(\hat{x}_1)=\frac{a^2+b^2-c^2}{2ab}.
\end{equation}
If $\{p_2,p_3\}\not\in \mathcal{S}_\ep$, then due to \eqref{mindist} we need to have $c\ge \sqrt2+\ep>b$, and from \eqref{lawcosines2} {and \eqref{order_tri2}} we find $\cos(\hat{x}_1)< a/2b\le 1/2$ and thus $\hat{x}_1\ge 60^\circ$. For bounding $\hat x_2$, we observe that
\[
{ \inf}\left\{\frac{a^2+c^2-b^2}{2ac}:\ a,b\in(1-\ep,\sqrt2+\ep),\ c\ge \sqrt2 + \ep\right\}
\]
is reached as {$(a,b,c)\to(1-\ep,\sqrt{2}+\ep,\sqrt2+\ep)$}, and equals the value of the expression $(a^2+c^2-b^2)/(2ac)$ in that limit, giving the desired bound on $\hat x_2$. The bound for $\hat x_3$ works similarly, with the roles of $a,b$ interchanged.

{\it Proof of (ii).} If $\{p_2,p_3\}\in \mathcal{S}_\ep$ then $c\in(1-\ep,\sqrt2+\ep)$. Moreover \eqref{order_tri2} holds. In such a range, the sup of the right hand side of \eqref{lawcosines2} is realized by $c=1-\ep,\, a=b=\sqrt2 + \ep$, in which case
\[
\hat{x}_1=\arccos\left(\frac{3+(4\sqrt{2}+2)\ep +\ep^2}{2(\sqrt{2}+\ep)^2}  \right)=\arccos\left(\frac34+O(\ep)\right)=\arccos\left(\frac34\right)+O(\ep),
\]
whereas the inf is reached for $a=b=1-\ep,\, c=\sqrt2+\ep$, in which case 
\[\hat{x}_1=\arccos\left(\frac{-(4+2\sqrt{2})\ep+\ep^2}{2(1-\ep)^2}\right)=\arccos(-O(\ep))=90^\circ +O(\ep).
\]
 {\it Proof of (iii).}
By \eqref{lawcosines2} and by the hypothesis we have
$$
\cos(\hat{x}_1)=\frac{(1+O(\ep))^2+b^2-c^2}{2(1+O(\ep))b}\le \frac{b^2-1-O(\ep)}{2(1+O(\ep))b}.
$$
It is easy to see that the quantity on the right-hand-side is - for $\ep$ small enough - monotonically increasing with respect to $b$, so that it is maximized for $b=\sqrt 2+\ep$. From this, the claim follows.

{\it Proof of (iv).}  Again by \eqref{lawcosines2} and by the hypothesis we have
$$
\cos(\hat{x}_1)\le \frac{{(1+O(\ep))^2}+c^2-b^2}{2(1-\ep)c},
$$
where the sup of the right-hand-side is reached for $|x_2-x_3|=\sqrt{2}+\ep$ and $|x_1-x_3|=1-\ep$, thus yielding the claim.
\end{proof}

By applying verbatim the same reasoning of the proof of Lemma \ref{trialemma}(ii) one gets the following result.
\begin{corollary}\label{toprove(i)}
Let $X\subset\mathbb R^2$ and let $x,y,z\in X$ be such that $\{x,y\}, \{y,z\}\in\mathcal{S}_0(X)$. Then $\widehat{xyz}\ge \arccos\left(\frac{3}{4}\right)\sim 41.4^\circ$.
\end{corollary}
\begin{lemma}\label{qualemma}
There exists $\ep''\in(0,\ep']$ (with $\ep'$ given by Lemma \ref{trialemma}) such that for every $\ep\in (0,\ep'')$ the following holds.
Let $\Xi$ be a set of labels and let $p_1,p_2,p_3,p_4\in \Xi$ be such that $\{p_i,p_j\}\in \mathcal{S}_{\ep}$ for all $i\neq j$ and set $x_i:=X(p_i)$ for all $i=1,\ldots,4$; then, up to relabeling, for all $i\in \{1,2,3,4\}$ we have
\begin{itemize}
\item[(i)] $| x_i-x_{i+1}|= 1+ O(\ep)$,
\item[(ii)] $\widehat{x_i x_{i+1}x_{i+2}}=90^\circ+O(\ep)$,
\item[(iii)] $|x_i-x_{i+2}|\le \sqrt{2}+O(\ep)$,
\end{itemize}
where $x_{i+4}=x_{i}$ for every $i=1,\ldots,4$.
If $p_1,p_2,p_3,p_4$ are such that $\{p_i,p_j\}\in \mathcal{S}_{\ep}$ for all $i\neq j$,  then the quadrilateral $\{p_1,p_2,p_3,p_4\}$ is called an {\it $\ep$-square}.
\end{lemma}
\begin{proof}
Up to relabeling we may suppose that the points $x_1,\ldots,x_4$ are in cyclic order along the boundary of the convex hull $\mathrm{Conv}(\{x_1,\ldots,x_4\})$.

{\it Proof of (i).}  Assume that $|x_1-x_2|\ge |x_i-x_{i+1}|$ for every $i=2,\ldots,4$. Then, under the constraints $|x_i-x_{i+1}|> 1-\ep$ and $|x_i-x_{i+2}|< \sqrt2 +\ep$, the sup of $|x_1-x_2|$ is realized by $|x_2-x_3|=|x_3-x_4|=|x_4-x_1|=1-\ep$ and $|x_1-x_3|=|x_2-x_4|=\sqrt2+\ep$, which gives the desired bound.\\ 
{\it Proof of (ii).}  It follows directly by Lemma \ref{trialemma}.

{\it Proof of (iii)}. By the law of cosines, (i) {and (ii),} we have 
\begin{align*}
|x_i-x_{i+2}|^2{= 2(1+O(\ep))^2-2(1+O(\ep))^2=2+O(\ep)},
\end{align*}
which gives the claim.
\end{proof}
We are now in a position to prove Lemma \ref{lem:combintometric}.
\begin{proof}[Proof of Lemma \ref{lem:combintometric}]
We assume that $p\in \Xi$ has the maximum number of $8$ neighbors $\mathcal{G}_\varepsilon$. We write $x=X(p)$ and we set $x_i=X(p_i)$ for every $i=1,\ldots,8$. Without loss of generality the $x_i$ are ordered in counterclockwise order around $x$. We recall that $\arccos\left(\frac{1}{2\sqrt2}\right)\sim 69.2^\circ$ and $\arccos\left(\frac 3 4\right)\sim 41.4^\circ$.
Let $\ep''$ be the constant given in Lemma \ref{qualemma}.

\medskip

\textit{Claim 1: There exists $\ep_0\in (0,\ep'')$ such that for all $\ep\in  (0,\ep_0)$ at least $7$ indices $i=1,\ldots,8$ are such that $\{p_i,p_{i+1}\}\in \mathcal{S}_{\ep}$}.
We first note that if for more than two indices $i$ there holds $|x_i-x_{i+1}|\ge\sqrt2+\ep$ then by Lemma \ref{trialemma}, $\widehat{x_ixx_{i+1}}\ge 60^\circ$, and thus at least one of the remaining $6$ angles is smaller than $(360^\circ-120^\circ)/6=40^\circ<\arccos\left(\frac 3 4\right)$\,. 
As a consequence, there exists $\ep_0>0$ such that for $\ep\in [0,\ep_0)$ we get a contradiction with Lemma \ref{trialemma} and hence $|x_i-x_{i+1}|\ge\sqrt2+\ep$ may hold for at most one index $i\in\Z/8\Z$. 

\begin{figure}
\includegraphics[width=7.5cm]{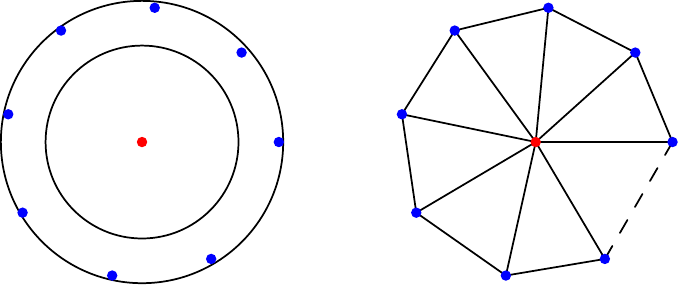}
 \caption{Illustration of Claim 1.}\label{fig-c1}
\end{figure}

\medskip

\textit{Claim 2: For all $\ep\in (0,\ep_0)$, the configuration ${X(\mathcal N_\ep(p))}=\{x,x_1,\ldots,x_8\}$ contains at least one $\ep$-square {as defined in Lemma \ref{qualemma}}.} 
Assume that this is not the case , namely that there exists $\ep\in (0,\ep_0)$ such that $X(\mathcal N_\ep(p))$ does not contain any $\ep$-square. Then, by Claim 1, each of the $p_i$'s ($i=1,\ldots,8$) has two or three neighbors in {$\mathcal N_\ep(p)$}. Note also that the sum of internal angles of the octagon $\{x_1,\ldots,x_8\}$ is $1080^\circ$, thus at least one angle is larger than $1080^\circ/8=135^\circ$, say it is the angle at $\widehat{x_1x_2x_3}$. Since $\mathcal{N}_\ep(p)\cap \partial\mathcal{G}_\varepsilon=\emptyset$, $x_2$ also has $8$ neighbors. By considering the successive angles around $x_2$ formed with the $8$ neighbors of $p_2$ in $\mathcal{G}_\varepsilon$, we have that {$6$ such angles are contained} outside the sector spanned by the angle $\widehat{x_1x_2x_3}$, therefore at least one of {these} angles is smaller than or equal to $(360^\circ-135^\circ)/6=37.5^\circ<\arccos\left(\frac 3 4\right){+O(\ep)}$ for $\ep\in[0,\ep_0)$ where $\ep_0$ is the one given Claim 1. But this fact contradicts Lemma \ref{trialemma}, and hence we get the claim.

\begin{figure}
\includegraphics[width=7.5cm]{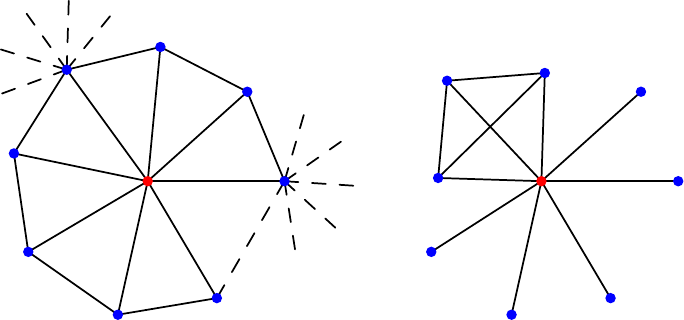}
\caption{Hypothesis and final situation reached by Claim 2.}\label{fig-c2}
\end{figure}

\medskip

\textit{Claim 3: There exists $\ep_1\in (0,\ep_0]$ such that for all $\ep\in (0,\ep_1)$, the configuration {$X(\mathcal N_\ep(p))$} contains at least two $\ep$-squares.} 
Assume that this is not the case , namely that there exists a sequence $\{\ep_n\}_{n\in\N}$ with $\ep_n\to 0^+$ as $n\to +\infty$ such that every $n\in\N$ there exists $q_n\in\Xi$ such that $\mathcal{N}_{\ep_n}(q_n)\cap\partial\mathcal{G}_{\ep_n}=\emptyset$ and $X(\mathcal N_{\ep_n}(X(q_n)))$ does not contain two $\ep_n$-squares. Fix $n\in\N$, and let $\ep=\ep_n$ and $p=q_n$ be as above. In view of Claim 2, this means that  $X(\mathcal{N}_{\ep}(p))$ contains only one $\ep$-square. Let $\{x,x_1,x_2,x_3\}$ be such $\ep$-square. 

\medskip

If $|x_1-x_8|\ge\sqrt 2+\ep$ then by Lemma \ref{trialemma}(iii), $\widehat{x_8xx_1} \ge 69^\circ$ for $\ep$ sufficiently small. But in this case, by Lemma \ref{qualemma}(ii) we conclude that
$$
\sum_{i=4}^{8}\widehat{x_{i-1}xx_{i}}\le 360^\circ-69^\circ-90^\circ{+O(\ep)}=201^\circ+{O(\ep)},
$$
which implies that the smallest angle between the $\widehat{x_{i-1}xx_i}$, {for $4\leq i\leq 8$}, is smaller than $40.2^\circ <\arccos\left(\frac 3 4\right)$, thus contradicting Lemma \ref{trialemma} for $\ep$ small enough. 

\medskip

This shows that for $\ep_1>0$ sufficiently small we have $\{p_8,p_1\}\in\mathcal S_\ep$. Similarly we find $\{p_3,p_4\}\in\mathcal S_\ep$.

\medskip

By Lemma \ref{trialemma}, we have $\widehat{xx_3x_4}, \widehat{xx_1x_8}\le 90^\circ+{O(\ep)}$, and as before, at least one of the $5$ {remaining} internal angles of the octagon $\{x_1,\ldots,x_8\}$ at vertices $x_4,x_5,x_6,x_7,x_8$ is larger than or equal to 
$$
\theta_\ep:=\frac{1}{5}\left(1080^\circ-450^\circ-{O(\ep)}\right)=126^\circ-{O(\ep)};
$$
say that $x_i$ is such a vertex. We are under the assumption that no $\ep$-square at $x$ contains $x_i$, thus by considering the possible allowed $\mathcal S_\ep$-edges between vertices in $\mathcal N_\ep(p)$ we find that $\sharp\left(\mathcal N_\ep(p_i)\cap \mathcal N_\ep(p)\setminus\{p_i\}\right)\le 3$. On the other hand, we are also under the assumption that $\mathcal{N}_\ep(p)\cap \partial \mathcal{G}_\varepsilon=\emptyset$, thus $\sharp(\mathcal N_\ep(p_i)\setminus\{p_i\})=8$. Thus there are $6$ angles at $x_i$ formed by successive neighbors of $x_i$ and not contained in $\widehat{x_{i-1}x_ix_{i+1}}$. At least one of these angles is smaller than or equal to 
$$
\beta_\ep:=(360^\circ-\theta_\ep)/6\le {39^\circ}+{O(\ep)}.
$$
{For $\ep_0$ small enough we find $\beta_\ep<\arccos(\frac 3 4) +O(\ep)$, contrary to} Lemma \ref{trialemma}, and our claim follows.

\medskip

\textit{Claim 4: There exists $\ep_2\in (0,\ep_1]$ such that for all $\ep\in [0,\ep_2)$  the configuration $X\cap B(x,\sqrt2+\ep)$ cannot contain only two $\ep$-squares with no common edges and two further successive edges {from $x$}.}
 We will call ``remaining vertices" the nearest neighbors of $x$ that do not belong to an $\ep$-square. From Claim 3, we know that there is at most 2 remaining vertices for $\ep<\ep_1$.
 Again we prove the claim by contradiction. Up to cyclic relabeling of the $x_i$ the two $\ep$-squares are $\{x,x_1,x_2,x_3\}$ and $\{x,x_4,x_5,x_6\}$. {By Lemma \ref{qualemma} and by the law of cosines we obtain 
 $$
\widehat {x_3 x x_{4}}\geq \arccos\left(\frac{2(1+O(\ep))^2-(1-\ep)^2}{2(1+O(\ep))^2}\right)=\arccos\left(\frac 1 2+O(\ep)\right)=60^\circ+O(\ep).
 $$
 }
 
Moreover, by using again Lemma \ref{qualemma}, at least one of the angles $\widehat{x_ixx_{i+1}}, i=6,7,8$ must be smaller than or equal to 
$$
\xi_\ep:=\frac13(360^\circ-2({90^\circ+O(\ep)})-60^\circ-O(\ep))=40^\circ+O(\ep)<\arccos\left(\frac{3}{4}\right)+O(\ep),
$$
for $\ep$ small enough. Therefore a contradiction to Lemma \ref{trialemma} follows.

\begin{figure}
\includegraphics[width=3cm]{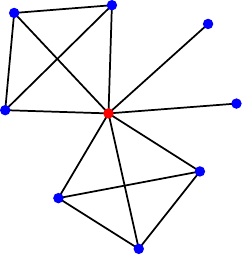}
\caption{What Claim 4 proves impossible.}
\end{figure}

\medskip

\textit{Claim 5: There exists $\ep_3\in (0,\ep_2]$ such that for all $\ep\in  (0,\ep_3)$ the following holds: if  the configuration $X\cap B(x,\sqrt2+\ep)$ contains two $\ep$-squares with no common edges and the two remaining vertices that are not {successive}, then it contains a further $\ep$-square, sharing one edge with each given $\ep$-squares.}

\begin{figure}
\includegraphics[width=7.5cm]{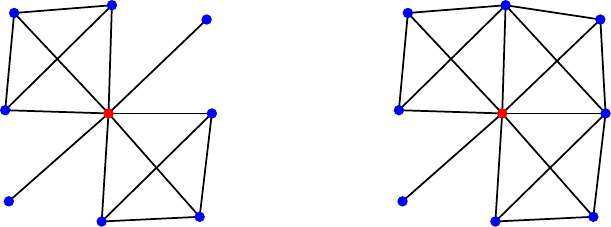}
\caption{Hypothesis and end result of Claim 5.}
\end{figure}

Up to cyclic relabeling of the $x_i$'s, the two $\ep$-squares are $\{x,x_1,x_2,x_3\}$ and $\{x,x_5,x_6,x_7\}$. We can assume without loss of generality $\beta_4:=\widehat{x_3 x x_5}\le \widehat{x_7 x x_1}=:\beta_8$, $\beta_4^-:=\widehat{x_3 x x_4}\le \widehat{x_4 x x_5}=:\beta_4^+$, and $\beta_8^-:=\widehat{x_7 x x_8}\le \widehat{x_8 x x_1}=:\beta_8^+$. By Lemma \ref{qualemma}, it follows that $\beta_4\le 90^\circ-O(\ep)$, $\beta_8\ge 90^\circ-O(\ep)$, and $\beta_4^-\le 45^\circ-O(\ep)$. By the assumption $|x_4-x|\le \sqrt 2 +\ep$\,.
By the law of cosines, we have
\begin{equation}\label{maybew}
\begin{aligned}
(1-\ep)^2\le&|x_3-x_4|^2=|x_3-x|^2+|x_4-x|^2-2 |x_3-x||x_4-x|\cos \beta_4^-\\
\le&(1+O(\ep))^2+|x_4-x|^2-2(1+O(\ep))|x_4-x|\cos\left(45^\circ+{O(\ep)}\right),
\end{aligned}
\end{equation}
whence, using  
$$
\cos\left(45^\circ+O(\ep)\right)=\frac {\sqrt{2}}{2}+O(\ep)\quad\textrm{ and }\quad(1+O(\ep))^2-(1-\ep)^2=O(\ep),
$$
we deduce the following inequality
\begin{equation*}\label{inL}
 |x_4-x|^2-({\sqrt 2}+O(\ep)) |x_4-x|+O(\ep)\ge 0;
\end{equation*}
it follows that $|x_4-x|=\sqrt 2+O(\ep)$.

Moreover, by \eqref{maybew}, it follows also that,  for $\ep$ small enough
\begin{equation*}
\begin{aligned}
\cos \beta_4^-=&\frac{|x_3-x|^2+|x_4-x|^2-|x_4-x_3|^2}{2 |x_3-x||x_4-x|}\\
\le&\frac{(1+O(\ep))^2+(\sqrt 2+O(\ep))^2-(1-\ep)^2}{2(1-\ep)(\sqrt 2+O(\ep))}=\frac {\sqrt{2}}{2}+O(\ep),
\end{aligned}
\end{equation*}
which together with the assumption on $\beta_4^-$ implies that $\beta_4^-=45^\circ+O(\ep)$.
Using again the law of cosines one can easily deduce that $|x_3-x_4|=1+O(\ep)$ and that $\widehat{x x_3 x_4}=90^\circ+O(\ep)$.
Analogously, one can see that $\beta_4^+=45^\circ+O(\ep)$ and that $|x_4-x_5|=1+O(\ep)$.
It follows that $|x_3-x_5|=\sqrt 2+O(\ep)$. 
Finally, since $p_3$ has $8$ neighbors, arguing by contradiction one can show that $|x_3-x_5|\le \sqrt{2}+\ep$.
In conclusion, $\{x,x_3,x_4,x_5\}$ is an $\ep$-square and then the Claim follows.

\medskip

\textit{Claim 6: There exists $\ep_4\in (0,\ep_3]$ such that for all $\ep\in  (0,\ep_4)$ the following holds:  $X(\mathcal N(p))$ contains at least 3  adjacent $\ep$-squares.} In view of Claims 3 and 5, the Claim needs to be proven only in the case that there are two $\ep$-squares sharing one edge.
Let $\{x,x_1,x_2,x_3\}$ and $\{x,x_3,x_4,x_5\}$ be two $\ep$-squares.
By Lemma \eqref{qualemma}, we have that 
\begin{equation}\label{q1}
\sum_{j=1}^3\widehat{x_{j}x_{j+1}x_{j+2}}\le 360^\circ+ O(\ep)\,,\qquad \widehat{xx_1x_2}+\widehat{x_4x_5x}\le 180^\circ+O(\ep),
\end{equation}
whereas by Lemma \ref{trialemma} we obtain
\begin{equation}\label{q2}
\widehat{x_8x_1x}+\widehat{xx_5x_6}\le 180^\circ+O(\ep).
\end{equation} 
Then, using again that the sum of the internal angles of the octagon is $1080^\circ$, we have
$$
\widehat{x_5x_6x_7}+\widehat{x_6x_7x_0}+\widehat{x_7x_0x_1}\ge 1080^\circ-720^\circ-O(\ep)=360^\circ - O(\ep)\,.
$$
Therefore, one of the above three angles, say $\widehat{x_{i-1}x_{i}x_{i+1}}$ is larger than 
$$
\vartheta_\ep=120^\circ-O(\ep).
$$
Since $p_i$ has eight neighbors in $\mathcal{G}_\varepsilon$ and since $x_i$ does not belong to an $\ep$-square, $p_i$ has exactly three neighbors in $\mathcal N_\ep(p)$ and their images through $X$ cover an angle at $x_i$ of at least $\vartheta_{\ep}$\,. Therefore amongst the remaining $6$ angles at $x_i$ spanned by successive neighbors of $x_i$, at least one is smaller than or equal to 
$$
\frac{360^\circ-\vartheta_{\ep}}{6}=40^\circ+O(\ep),
$$
contradicting Lemma \ref{trialemma}, and concluding the proof of Claim 6. 
\medskip

\textit{Claim 7: There exists $\ep_5\in (0,\ep_4]$ such that for all $\ep\in (0,\ep_5)$ $X(\mathcal{N}_\ep(p))$ contains four $\ep$-squares.}
 
By Claim 6, we can assume that there are three $\ep$-squares. Up to relabeling such $\ep$-squares are  $\{x,x_1,x_2,x_3\}$, $\{x,x_3,x_4,x_5\}$, and $\{x_5,x_6,x_7,x \}$.
By Lemma \ref{qualemma} we have
\begin{equation}\label{a901}
\widehat{x_7xx_1}=90^\circ+O(\ep),
\end{equation}
and hence, by the law of cosines,
$$
|x_7-x_1|=\sqrt{2}+O(\ep).
$$
Moreover, again by the law of cosines the remaining angles $\widehat{x_7 x x_0}, \widehat{x_0 x x_1}$ also are $O(\ep)$-close to $45^\circ$. By arguing as in Claim 5 one can easily get the claim.

\medskip

Set $\alpha_0:=\ep_5$.
In view of Claim 7 and of the very definition of $\ep$-square, \eqref{deform00} is satisfied for $\ep\in [0,\alpha_0)$. 
We therefore define $\phi:\mathcal{N}_\ep(p)\to \{-1,0,1\}^2$ as in \eqref{deform} and by all the Claims above,  it is easy to show that $\delta_\phi(x',x'')\leq C_3 \alpha|x'-x''|$ for all $x',x''\in \{x,x_1,...,x_8\}$ for some constant $C_3\in [1,\frac 1 {\alpha_0})$ (depending only on $\alpha_0$).
\end{proof}

 Notice that for $\ep=0$ the $\ep$-squares are nothing but the unit squares. Therefore, by the same proof as for Lemma \ref{lem:combintometric} with $\alpha=0$ we obtain the following result.

\begin{corollary}\label{toprovethm}
Let $X\in\mathcal{C}$ and let $x\in X$ have $8$ neighbors  in $\mathcal{G}_0(X)$, each of which has in turn $8$ neighbors in $\mathcal{G}_0(X)$. Let $x_1,\ldots,x_8,x_9\equiv x_1$ be the neighbors of $x$ ordered counterclockwise around $x$  and let $|x_1-x|=\min_{i=1,\ldots,8}|x-x_i|$. Then, the quadrilaterals $\{x,x_1,x_{2},x_{3}\}$, $\{x,x_3,x_{4},x_{5}\}$,  $\{x,x_5,x_{6},x_{7}\}$, $\{x,x_7,x_{8},x_{9}\}$ are all unit squares.
 \end{corollary}

We next pass to proving Lemma \ref{cor_squarerigid}.

\begin{proof}[Proof of Lemma \ref{cor_squarerigid}]
Set $\alpha'_0:=\ep''$ where $\ep''$ is the one given in Lemma \ref{qualemma}.
Let $\alpha\in (0,\alpha_0']$ and let $\{p_1,p_2,p_3,p_4\}$ denote the set of vertices of $\mathcal G_\alpha$. By hypothesis, $\{p_i,p_j\}\in\mathcal{S}_\alpha$ for every $i,j=1,\ldots,4$ with $i\neq j$. Then, the assumptions of Lemma \ref{qualemma} are satisfied (with $\ep$ replaced by $\alpha$), so that setting $x_i:=X(p_i)$ for every $i=1,\ldots,4$ and $x_{i+4}\equiv x_i$ for every $i\in\Z$, we deduce that, up to a relabeling,
 (i),(ii), and (iii) hold true. 
In particular, for every $\alpha\in[0,\alpha'_0)$ we have
\begin{equation}\label{almostcorpro}
1-\ep\le |x_{i+1}-x_{i}|=1+O(\alpha),\qquad
\sqrt{2}+O(\alpha)=|x_{i}-x_{i+2}|\le \sqrt{2}+\alpha.
\end{equation}
Therefore, by \eqref{almostcorpro}, there exists a constant $C_3'\in [1,\frac{1}{\alpha'_0})$ (depending only on $\alpha'_0$) and a map $\phi:\{p_1,p_2,p_3,p_4\}\to\{0,1\}^2$ with
$$
\phi(p_1)=(0,0), \quad \phi(p_2)=(1,0), \quad \phi(p_3)=(1,1), \quad \phi(p_4)=(0,1),
$$
such that $\delta_{\phi}(x',x'')\le C_3\alpha |x'-x''|$ for all $x',x''\in \{x_1,\ldots,x_4\}$. As a consequence, \eqref{deform0bis} holds true.
\end{proof}


\section{List of Notations} 
Below we produce a list of those notations used at several points in the paper which we feel would help the reader orient, together with the main equations in the paper in which those notations are introduced:
\begin{longtable}{rll}
$\mathcal E[V](X_N)$ & - & energy defined in \eqref{energy}\\
$\overline{\mathcal E}_{\mathrm{sq}}[V]$ & - & minimal energy per point of a square lattice, see \eqref{eppz2_intro}\\
$\mathcal E_4[V](x_1,x_2,x_3,x_4)$ & - & energy as in \eqref{defEWintro} (see also Figure \ref{fig-resum}), and later also \eqref{defEW}\\
$E_\beta, E_\beta^1, E_\beta^2$ & - & intervals of distances, see \eqref{eq_Ealpha}\\
$ W(s)$ & - & \eqref{VforW}, and later also \eqref{WfromV}\\
$\mathcal Z_\boxtimes$ & - & combinatorial model-space from \eqref{zboxtimes}\\
$\mathcal S_{\alpha}, \mathcal G_{\alpha}, \mathcal N_\alpha(p), \partial \mathcal G_{\alpha}$& - & graph data from \eqref{not_graph}\\
$ X_1\sim_\alpha X_2$ & - & $\alpha$-deformation, see \eqref{alphadef}\\
$\mathcal X_4(\mathbb R^2)/\mathrm{Isom}(\mathbb R^2)$ & - & space of $4$-point configurations first appearing in \eqref{x4}\\
$\mathscr{Q}_\alpha$ & - & small deformations of squares modulo isometry, see \eqref{defsqu}\\
$\mathscr{S}_\alpha$ & - & small dilations of squares modulo isometry, see \eqref{defsqu2}\\
$E_\alpha^{sq}$ & - & square distances corresponding to $E_\alpha$, and defined in \eqref{ealphasq}\\
$\mathcal{D}$ & - & set of distances from $\mathbb Z^2$, defined in \eqref{distz2}\\
$\mathrm{Ell}_\alpha(a,b)$ & - & ellipse defined in \eqref{includellips}\\
$\mathrm{Sides}(Q'_r), \mathrm{Sides}(\mathcal{Q}'_r), \mathrm{Diag}(Q'_r) \mathrm{Diag}(\mathcal{Q}'_r)$ & - & see \eqref{z2squares}\\
$\mathrm{Sides}(Q_r), \mathrm{Sides}(\mathcal{Q}_r), \mathrm{Diag}(Q_r), \mathrm{Diag}(\mathcal{Q}_r)$ & - & see \eqref{z2squaresdef}\\
$\mathcal{L}_r$ & - & sublattices of $\mathbb Z^2$ of scale $r$, defined in \eqref{calLr}\\
$\quad m(r) $ & - & multiplicities of sublattices, defined in \eqref{mr}\\
$\widetilde{\mathcal D}\subset\mathcal D$ & - & the subset constructed in Lemma \ref{lem_splitgraph}\\
$\mathcal{NQ}, \mathcal{NQ}^{(1)}, \mathcal{NQ}^{(2)}$ & - & sets defined in \eqref{energysplit}\\
$\mathcal{Q}^{b}_{r}(Q_1)$ & - & squares of scale $r$ intersecting $Q_1$, see \eqref{mub}\\
$e(v,r)$ & - & error term as in \eqref{tayloreq}\\
$\widetilde\VV(t^2), \VV_*(t^2)$ & - & resummed interaction potentials defined in \eqref{defntildev}\\
$\mathrm{err}_1(r,V), \mathrm{err}_2(r,V)$ & - & error terms from \eqref{errbdterms}\\
$\mathrm{err}_3(r,V)$ & - & error term from \eqref{notationpass2.1}\\
$\mathrm{err}_4(V,X)$ & - & error term from \eqref{boundnq}\\
$V_{**}(r)$ & - & interaction potential defined in \eqref{v**}
\end{longtable}



\end{document}